\documentclass[graybox,envcountchap]{svmono} %,sectrefs
\usepackage[english]{babel}

\usepackage{type1cm}
\usepackage{amsthm,amssymb}

\usepackage{makeidx}         % allows index generation
\usepackage{subcaption}
\usepackage{graphicx}        % standard LaTeX graphics tool
\usepackage{multicol}        % used for the two-column index
\usepackage[bottom]{footmisc}% places footnotes at page bottom

\usepackage{newtxtext}       % 
\usepackage[varvw]{newtxmath}       % selects Times Roman as basic font

\usepackage{dsfont}
\usepackage{hyperref}
\usepackage{enumerate}
\usepackage{cleveref}
\usepackage{pdfpages}

\usepackage{thmtools} % added because on some versions cleveref fails to pick the correct labels

\allowdisplaybreaks

\newtheorem{thm}{Theorem}[chapter]
\newtheorem{basrepthm}[thm]{Basic Representation Theorem}
\newtheorem{bernthm}[thm]{Bernstein's Theorem}

\newtheorem{havithm}[thm]{Haviland's Theorem}
\newtheorem{hausthm}[thm]{Hausdorff's Theorem}

\newtheorem{hamthm}[thm]{Hamburger's Theorem}

\newtheorem{lumathm}[thm]{Luk\'acs--Markov Theorem}

\newtheorem{stielthm}[thm]{Stieltjes' Theorem}

\newtheorem{svethm}[thm]{{\v{S}}venco's Theorem}
\newtheorem{richthm}[thm]{Richter's Theorem}
\newtheorem{lem}[thm]{Lemma}

\newtheorem{prop}[thm]{Proposition}
\newtheorem{cor}[thm]{Corollary}

\newtheorem{open}[thm]{Open Problem}

\theoremstyle{definition}
\newtheorem{dfn}[thm]{Definition}
\newtheorem{exm}[thm]{Example}
\newtheorem{exms}[thm]{Examples}

\theoremstyle{remark}
\newtheorem{rem}[thm]{Remark}

\newcommand{\exmsymbol}{\hfill$\circ$}

\newcommand{\cset}{\mathds{C}}
\newcommand{\kset}{\mathds{K}}
\newcommand{\nset}{\mathds{N}}

\newcommand{\qset}{\mathds{Q}}
\newcommand{\rset}{\mathds{R}}

\newcommand{\tset}{\mathds{T}}

\newcommand{\Ad}{\mathrm{Ad}}

\newcommand{\cone}{\mathrm{cone}\,}
\newcommand{\conv}{\mathrm{conv}\,}

\newcommand{\diff}{\mathrm{d}}

\newcommand{\dom}{\mathrm{dom}\,}

\newcommand{\pos}{\mathrm{Pos}}

\newcommand{\sign}{\mathrm{sgn}}

\newcommand{\supp}{\mathrm{supp}\,}

\newcommand{\inter}{\mathrm{int}\,}
\newcommand{\id}{\mathrm{id}}
\newcommand{\one}{\mathds{1}}

\newcommand{\Gl}{\mathrm{Gl}}
\newcommand{\Sl}{\mathrm{Sl}}
\newcommand{\gl}{\mathfrak{gl}}
\newcommand{\tr}{\mathrm{tr}}
\newcommand{\fsl}{\mathfrak{sl}}

\newcommand{\cA}{\mathcal{A}}
\newcommand{\cB}{\mathcal{B}}
\newcommand{\cC}{\mathcal{C}}
\newcommand{\cat}{\mathcal{C}}

\newcommand{\cH}{\mathcal{H}}

\newcommand{\cL}{\mathcal{L}}

\newcommand{\cS}{\mathcal{S}}
\newcommand{\cT}{\mathcal{T}}
\newcommand{\cX}{\mathcal{X}}

\newcommand{\cV}{\mathcal{V}}
\newcommand{\cW}{\mathcal{W}}
\newcommand{\cY}{\mathcal{Y}}
\newcommand{\cZ}{\mathcal{Z}}

\newcommand{\fB}{\mathfrak{B}}

\newcommand{\fD}{\mathfrak{D}}
\newcommand{\fd}{\mathfrak{d}}

\newcommand{\fg}{\mathfrak{g}}

\newcommand{\fo}{\mathfrak{o}}

\newcommand{\so}{\mathfrak{so}}

\ExplSyntaxOn
\NewDocumentCommand{\advanced}{mO{#3}m}
 {% #1 is a sectioning command
  \cs_set_eq:Nc \__advanced_save: { the \cs_to_str:N #1 }
  \cs_set:cpn { the \cs_to_str:N #1 } { \__advanced_save: \advancedmarker }
  \bool_set_true:N \__advanced_killwidth_bool
  #1[#2]{#3}
  \bool_set_false:N \__advanced_killwidth_bool
  \cs_set_eq:cN { the \cs_to_str:N #1 } \__advanced_save:
 }
\NewDocumentCommand{\advancedmarker}{}
 {
  \bool_if:NTF \__advanced_killwidth_bool { \makebox[5pt][l]{*} } { * }
 }

\bool_new:N \__advanced_killwidth_bool

\AddToHook{cmd/@starttoc/begin}{\bool_set_true:N \__advanced_killwidth_bool}
\ExplSyntaxOff

\makeindex             % used for the subject index
\date{\today}

\begin{document}

\author{Philipp J.\ di\,Dio}
\title{Linear Operators on Polynomials and $K$-Positivity Preserver}
\subtitle{A journey from analysis to algebra and back}
\maketitle

\frontmatter%%%%%%%%%%%%%%%%%%%%%%%%%%%%%%%%%%%%%%%%%%%%%%%%%%%%%%

%\include{dedication}

%\begin{dedication}
%Samuel Karlin\index{Karlin, S.}\quad {\small(June 8, 1924 -- December 18, 2007)}\\ \medskip
%He solved almost unnoticed an\\ important algebraic question.
%\end{dedication}

%\foreword
%\addcontentsline{toc}{chapter}{Foreword}

%\include{preface}
\preface
\addcontentsline{toc}{chapter}{Preface}

These are the extended lecture notes of my lecture about
\begin{center}\itshape
Linear Operators on Polynomials,\\ $K$-Positivity Preserver, and their Generators.
\end{center}
The lecture was given at the University of Konstanz in the winter semester 2025/26.
The content is mainly taken from \cite{didio24posPresConst}, \cite{didio25KPosPresGen}, \cite{didio25hadamardLanger}, and \cite{didio25PosToSoS}.

Let $n\in\nset$ and let $K\subseteq\rset^n$ be closed.
Real algebraic geometry studies the cone
\[\pos(K) := \big\{f\in\rset[x_1,\dots,x_n] \,\big|\, f(x)\geq 0\ \text{for all}\ x\in K\big\}\label{posK}\]
of polynomials which are non-negative on $K$.
There is an enormous amount of literature on this still active field.

It is therefore very surprising that linear operators 
\[T:\rset[x_1,\dots,x_n]\to\rset[x_1,\dots,x_n]\]
are hardly studied.
In fact, very little was know so far about these operators, especially when they have to preserve $K$-positivity, i.e.,
\[T\pos(K) \subseteq\pos(K).\]
The aim of this lecture is to present the recent developments in this field, i.e., the progress made in the works \cite{didio24posPresConst}, \cite{didio25KPosPresGen}, \cite{didio25hadamardLanger}, and \cite{didio25PosToSoS}.
The material presented here is for one semester of two weekly 1.5 hours lectures including a weekly exercise session of approximately 45 minutes.

The course is for master students in mathematics.
Besides a bachelor in mathematics (i.e., especially basic knowledge in analysis and linear algebra), it is recommended that the students resp.\ the reader has a basic knowledge in (real) algebraic geometry (i.e., non-negative polynomials), (partial) differential equations, Lie groups and algebras, and operator theory.

\vspace{\baselineskip}
\begin{flushright}\noindent
Konstanz, February 2026\hfill {\it Philipp J.\ di\,Dio}\\
\end{flushright}

%\include{acknowledgement}
%\extrachap{Acknowledgements}
%\addcontentsline{toc}{chapter}{Acknowledgements}

\tableofcontents

\mainmatter%%%%%%%%%%%%%%%%%%%%%%%%%%%%%%%%%%%%%%%%%%%%%%%%%%%%%%%
%\begin{partbacktext}
%\part{Part Title}
%\noindent Use the template \emph{part.tex} together with the document class SVMono (monograph-type books) or SVMult (edited books) to style your part title page and, if desired, a short introductory text (maximum one page) on its verso page.

%\end{partbacktext}

%\motto{A word concerning an incident in the last chapter.\\ \medskip
%\ \hspace{1cm} \normalfont{Herman Melville: Moby-Dick (1851, Chapter LIX)}\index{Melville, H.}}

%\motto{Statt des törichten Ignorabimus heisse im Gegenteil unsere Losung:\\ \ \; Wir m\"{u}ssen wissen,\\ \ \; wir werden wissen.\footnote{Instead of foolish ignorance, our motto should be the opposite: We must know, we will know.}\\ \medskip
%\ \hspace{1cm} \normalfont{David Hilbert: Radio Address, 8th Sept.\ 1930 \cite{hilbert30radio}}\index{Hilbert, D.}}
% https://old.maa.org/press/periodicals/convergence/david-hilberts-radio-address

%%%%%%%%%%%%%%%%%%%%%%%%%%%%%%%%%
%%%%%%%%%%%%%%%%%%%%%%%%%%%%%%%%%
\part{Introduction to Moments}%%%
%%%%%%%%%%%%%%%%%%%%%%%%%%%%%%%%%
%%%%%%%%%%%%%%%%%%%%%%%%%%%%%%%%%
\label{part:momentsIntro}

%\motto{Pure mathematics is, in its way, the poetry of logical ideas.\\ \medskip
%\ \hspace{1cm} \normalfont{Albert Einstein \cite{einstein35noether}}\index{Einstein, A.}}

%%%%%%%%%%%%%%%%%%%%%%%%%%%%%%%%%%%%%%%%%%%%
\chapter{Moments and Moment Functionals}%%%%
\label{ch:mom}%%%%%%%%%%%%%%%%%%%%%%%%%%%%%%
%%%%%%%%%%%%%%%%%%%%%%%%%%%%%%%%%%%%%%%%%%%%

Our main aim is to study linear operators
\[T:\rset[x_1,\dots,x_n]\to\rset[x_1,\dots,x_n],\]
especially with
\[T\pos(K)\subseteq\pos(K)\]
for some given closed $K\subseteq\rset^n$ and $n\in\nset$.
While this formulation seems to be a purely algebraic questions, it turns out that actually functional analytic methods are required.
These are the dual objects of $\pos(K)$, i.e., moments.

\section{Moments and Moment Functionals}

\begin{dfn}
Let $(\cX,\cA)$ be a measurable space, let $\mu$ be a measure on $(\cX,\cA)$, and let $f:\cX\to\rset$ be a $\mu$-integrable function.
The number
\[s_f := \int_\cX f(x)~\diff\mu(x)\]
is called the \emph{moment of $\mu$ with respect to $f$.}\index{moment}
\end{dfn}

\begin{exms}\label{exm:moments}
\begin{enumerate}[\bfseries (a)]
\item Let $(\cX,\cA)$ be a measurable space, $f:\cX\to\rset$ be measurable, and $\delta_y$\label{deltay} the point measure at some point $y\in\cX$, i.e.,
\[\int_\cX f(x)~\diff\delta_y(x) = f(y).\]
Then $f(y)$ is the moment of $\delta_y$ with respect to $f$.

\item Let $(\rset,\fB(\rset))$ be $\rset$ with the Borel-$\sigma$-algebra $\fB(\rset)$,\index{Borel-$\sigma$-algebra}\label{fBX} let $f(x) = x^k$ for some $k\in\nset_0$, and let $\mu$ be given by $\diff\mu(x) = \chi_{[0,1]}~\diff x$ with the characteristic function $\chi_{[0,1]}$ of the unit interval $[0,1]$.
Then
\begin{align*}
s_k &:= \int_\rset x^k~\diff\mu(x)
= \int_0^1 x^k~\diff x
= \left[\frac{1}{k+1}\cdot x^{k+1} \right]_{x=0}^1
= \frac{1}{k+1}
\end{align*}
for all $k\in\nset_0$ is the moment of $\mu$ with respect to the function $x^k$, or short the \emph{$k$-th moment of $\mu$}.

\item Let $n\in\nset$ and $\mu$ be a measure on $(\rset^n,\fB(\rset^n))$ such that $x^\alpha := x_1^{\alpha_1}\cdots x_n^{\alpha_n}$ are $\mu$-integrable for all $\alpha = (\alpha_1,\dots,\alpha_n)\in\nset_0^n$.
Then
\[s_\alpha := \int_{\rset^n} x^\alpha~\diff\mu(x)\]
is the $\alpha$-th moment of $\mu$.
The $s_\alpha$ are the \emph{classical}\index{moment!classical} or \emph{polynomial moments}\index{moment!polynomial} of $\mu$.
Here the name \emph{moment} comes from.
Given a body in $\rset^3$ with density distribution $\rho(x,y,z)$, then
\[\int_{\rset^3} (x^2 + y^2)\cdot\rho(x,y,z)~\diff x~\diff y~\diff z\]
is the \emph{moment of inertia}\index{moment!inertia} of the rotation of the body with density $\rho$ around the $z$-axis.
\exmsymbol
\end{enumerate}
\end{exms}

\begin{dfn}
Let $(\cX,\cA)$ be a measurable space, let $\cV$ be a (finite or infinite dimensional) real vector space of measurable functions $f:\cX\to\rset$, and let
\[L:\cV\to\rset\]
be a linear functional on $\cV$.
Then $L$ is called a \emph{moment functional}\index{moment!functional}\index{functional!moment}, if there exists a measure $\mu$ on $(\cX,\cA)$ with
\begin{equation}\label{eq:momFunctional}
L(f) = \int_\cX f(x)~\diff\mu(x)
\end{equation}
for all $f\in\cV$.
The measure $\mu$ in (\ref{eq:momFunctional}) is called a \emph{representing measure}\index{measure!representing} of $L$.
If $\mu$ in (\ref{eq:momFunctional}) is unique, then $L$ is called \emph{determinate}.\index{determinate!moment!functional}\index{functional!moment!determinate}\index{moment!functional!determinate}
Otherwise, $L$ is called \emph{indeterminate}.\index{indeterminate!moment!functional}\index{functional!moment!indeterminate}\index{moment!functional!indeterminate}
If $\cV$ is finite dimensional, then $L$ is called a \emph{truncated moment functional}.\index{moment!functional!truncated}\index{functional!moment!truncated}
\end{dfn}

\begin{exms}\label{exm:momFunctionals}
\begin{enumerate}[\bfseries (a)]
\item Let $(\cX,\cA)$ be a measurable space, let $y\in\cX$, and let $\cV$ be a (finite or infinite dimensional) real vector space of measurable functions $f:\cX\to\rset$.
Then
\[l_y:\cV\to\rset, \quad f\mapsto l_y(f) := f(y)\label{dfn:ly}\]
is a moment functional with representing measure $\delta_y$.

\item Let
\[L:\rset[x]\to\rset\]
be defined by
\[L(x^k) := \frac{1}{k+1}\]
for all $k\in\nset_0$ and linearly extended to all $\rset[x]$.
Then $L$ is a moment functional with representing measure $\mu$ with $\diff\mu(x) = \chi_{[0,1]}~\diff x$.
The representing measure $\mu$ is unique and $L$ is determinate, as the next result shows.
\exmsymbol
\end{enumerate}
\end{exms}

\begin{prop}\label{prop:momFuncDeterminate}
Let $n\in\nset$, let $(\cX,\cA) = (\rset^n,\fB(\rset^n))$, and let
\[L:\rset[x_1,\dots,x_n]\to\rset\]
be a moment functional which has at least one representing measure $\mu$ such that $\supp\mu$ is compact.
Then $\mu$ is unique and $L$ is determinate.
\end{prop}
\begin{proof}
See Problem \ref{prob:momFuncDeterminate}.
\end{proof}

\begin{prop}\label{prop:momFuncIndeterminate}
Let $(\cX,\cA)$ be a measurable space, let $\cV$ be a real vector space of measurable functions $f:\cX\to\rset$, and let
\[L:\cV\to\rset\]
be an indeterminate moment functional.
Then $L$ has infinitely many representing measures.
\end{prop}
\begin{proof}
See Problem \ref{prob:momFuncIndeterminate}.
\end{proof}

\section{Restrictions of the Support of Representing Measures}

For a function $f:\cX\to\rset$ on any set $\cX$, we define by
\begin{equation}\label{eq:zeroSet}
\cZ(f) := \{x\in\cX \,|\, f(x) = 0\} = f^{-1}(\{0\})
\end{equation}
the \emph{zero set of $f$}.\index{zero!set}\index{set!zero}

\begin{thm}\label{thm:supportRestriction}
Let $(\cX,\cA)$ be a measurable space, let $\cV$ be a real vector space of measurable functions $f:\cX\to\rset$, and let
\[L:\cV\to\rset\]
be a moment functional.
If there exists a function $g\in\cV$ with
\[g\geq 0 \quad\text{and}\quad L(g) = 0,\]
then 
\[\mu\big(\cX\setminus\cZ(g)\big) = 0\]
for every representing measure $\mu$ of $L$, i.e., every representing measure $\mu$ of $L$ is supported on $\cZ(g)$.
\end{thm}
\begin{proof}
Define
\[A_0 := \{x\in\cX \,|\, g(x)\geq 1\} = g^{-1}([1,\infty))\]
and
\[A_n := \left\{x\in\cX \,\middle|\, n^{-1} > g(x) \geq (n+1)^{-1}\right\} = g^{-1}([(n+1)^{-1},n^{-1}))\]
for all $n\in\nset$.
Since $g$ is measurable and $[1,\infty)$ and $[(n+1)^{-1},n^{-1})$ are Borel sets for all $n\in\nset$, $A_n\in\cA$ for all $n\in\nset_0$ .
Then
\[0\leq \mu(A_n) = \int_\cX \chi_{A_n}(x)~\diff\mu(x) \leq (n+1)\int_{A_n} g(x)~\diff\mu(x) \leq (n+1)\cdot L(g) = 0\]
for all $n\in\nset_0$, i.e., $\mu(A_n) = 0$ for all $n\in\nset_0$.
Hence,
\[\mu\big(\cX\setminus\cZ(g)\big) = \mu\left(\bigcup_{n\in\nset_0} A_n \right) = \sum_{n\in\nset_0} \mu(A_n) = 0.\qedhere\]
\end{proof}

\section{The Moment Cone and the Riesz Functional}

\begin{dfn}\label{dfn:momCone}
Let $(\cX,\cA)$ be a measurable space and let $\cV$ be a real vector space of measurable functions $f:\cX\to\rset$.
We define
\[\cL(\cV) := \{L:\cV\to\rset\ \text{linear} \,|\, L\ \text{is a moment functional}\}\]
to be the \emph{cone of moment functionals} on $\cV$.\index{cone!moment functional}\index{moment!functional!cone}\index{functional!moment!cone}
If $\cV$ is finite dimensional, then $\cL(\cV)$ is called \emph{truncated}\index{cone!moment functional!truncated}\index{moment!functional!cone!truncated}\index{truncated!moment!functional} moment cone.
\end{dfn}

\begin{lem}\label{lem:convexLV}
Let $(\cX,\cA)$ be a measurable space and let $\cV$ be a real vector space of measurable functions $f:\cX\to\rset$.
Then $\cL(\cV)$ is a convex cone.
\end{lem}
\begin{proof}
See Problem \ref{prob:convexLV}.
\end{proof}

\begin{dfn}\label{dfn:RieszFunctional}
Let $(\cX,\cA)$ be a measurable space and let $\cV$ be a real vector space of measurable functions $f:\cX\to\rset$.
If $v=(v_i)_{i\in I}$ is a basis of $\cV$ with some index set $I$ and if $s = (s_i)_{i\in I}\in\rset^I$ is a real sequence indexed by $I$, then the \emph{Riesz functional}\index{functional!Riesz}\index{Riesz!functional}
\[L_s:\cV\to\rset\index{dfn:Ls}\]
of $s$ is defined by
\[L_s(v_i) := s_i\]
for all $i\in I$ and linearly extended to $\cV$.
If $L_s$ is a moment functional, then $s$ is called a \emph{moment sequence}.\index{moment!sequence}\index{sequence!moment}
If $I$ is finite and hence $\cV$ is finite dimensional, then a moment sequence $s$ is called \emph{truncated}.\index{truncated!moment!sequence}\index{moment!sequence!truncated}\index{sequence!moment!truncated}
We call the set
\[\cS(\cV,v) := \big\{s\in\rset^I \,\big|\, s\ \text{is a momemt sequence}\}\label{dfn:SV}\]
of all moment sequences the \emph{moment cone}.\index{moment!cone}\index{cone!moment}
\end{dfn}

\begin{rem}
The moment cone $\cS(\cV,v)$ depends on the basis $v=(v_i)_{i\in I}$ of $\cV$.
The cone $\cL(\cV)$ is independent on the specific choice of basis $v=(v_i)_{i\in I}$ of $\cV$.
Since a basis $v$ of $\cV$ describes $\cV$,
\[\cL(\cV) \;\cong\; \cS(\cV,v) \;\cong\; \cS(\cV,\tilde{v})\]
for any basis $\tilde{v}$ of $\cV$.
\exmsymbol
\end{rem}

\begin{exms}
Let $(\cX,\cA) = (\rset,\fB(\rset))$ and let $\cV=\rset[x]$ be with basis $v=(v_i)_{i\in\nset_0}$ and $v_i(x) := x^i$ for all $i\in\nset_0$.
\begin{enumerate}[\bfseries (a)]
\item Let $y\in\rset$ and set
\[s := \big(y^i\big)_{i\in\nset_0}.\]
Then, by Example \ref{exm:momFunctionals} (a), $s$ is a moment sequence.

\item Let
\[s := \left(\frac{1}{1+i}\right)_{i\in\nset_0}.\]
Then, by Example \ref{exm:momFunctionals} (b), $s$ is a moment sequence.
\exmsymbol
\end{enumerate}
\end{exms}

\section{Richter's Theorem}

Moment functionals and measures can be quite difficult to describe.
However, truncated moment functionals always have very simple representing measures as the next result shows.

\begin{richthm}[{\cite[Satz 4]{richte57}}]\label{thm:richter}\index{Richter's Theorem}\index{Theorem!Richter}
Let $(\cX,\cA)$ be a measurable space, let $\cV$ be a finite-dimensional real vector space of measurable functions $f:\cX\to\rset$, and let
\[L:\cV\to\rset\]
be a moment functional.
Then $L$ has a finitely atomic representing measure
\[\mu = \sum_{i=0}^{k} c_i\cdot\delta_{x_i}\]
with at most $k\leq\dim\cV$ atoms, i.e., there exist constants $c_1,\dots,c_k>0$ and pairwise different points $x_1,\dots,x_k\in\cX$ such that
\[L(f) = \sum_{i=1}^{k} c_i\cdot f(x_i)\]
for all $f\in\cV$.
\end{richthm}
\begin{proof}
We proceed via induction over $n:=\dim\cV$.

\underline{$n=1$:}
If $L(f)=0$ for all $f\in\cV$, then $\mu = 0$ is a representing measure of $L$.
Hence, assume without loss of generality that $L(f)\neq 0$ for a $f\in\cV\setminus\{0\}$.
Fix $f\in\cV\setminus\{0\}$.
Since $L$ is a moment functional, there exists a point $x\in\cX$ such that
\[f(x)\neq 0 \quad\text{and}\quad \sign L(f) = \sign f(x).\]
Then
\[L(g) = g(x)\cdot \frac{L(f)}{f(x)} = \int_\cX g(y)~\diff\mu(y) \qquad\text{with}\qquad \mu = \frac{L(f)}{f(x)}\cdot\delta_x\]
for any $g\in\cV$, since $\cV$ is $1$-dimensional and hence $g = c\cdot f$ with $c = \frac{g(x)}{f(x)}$.

\underline{$n\to n+1$:}
Assume the theorem holds for all $\dim\cV=1,\dots,n$.
We need to prove it for $\dim\cV=n+1$.
Set
\[\tilde{\cL}:= \conv\cone\{l_x \,|\, x\in\cX\}.\]
Then
\[\tilde{\cL}\subseteq\cL.\]
Since $\cV$ is finite-dimensional, $\tilde{\cL}$ and $\cL$ are finite-dimensional and have non-empty interior with
\[\inter\tilde{\cL}\subseteq\inter\cL.\]
Assume
\[\inter\tilde{\cL}\neq\inter\cL\]
and let
\[L\in\inter(\inter\cL\setminus\tilde{\cL}) = \inter\cL\setminus\overline{\tilde{\cL}}.\]
Then there exists a $f\in\cV \cong \cV^{**}$ such that
\[f(x) = l_x(f)>0 \qquad\text{for all}\ x\in\cX\]
and
\begin{equation}\label{eq:Lfsmaller0}
L(f) < 0.
\end{equation}
Since $L\in\inter\cL\subseteq\cL$ is a moment functional and $f\in(\cV_+)$,
\[L(f)\geq 0.\]
This is a contradiction to (\ref{eq:Lfsmaller0}), i.e., we proved $\inter\tilde{\cL}=\inter\cL$.

Now let
\[L\in\partial\cL\cap\cL\]
be a moment functional on the boundary of the truncated moment cone.
Then there exists a
\[f\in\cV_+\subseteq\cV\cong\cV^{**}\]
such that
\[L(f)=0.\]
By \Cref{thm:supportRestriction}, every representing measure $\mu$ of $L$ is supported on $\cY:=\cZ(f)$.
Let $\cW$ be a $n$-dimensional subspace of $\cV$ such that
\[\cV = \cW+\rset\cdot f,\]
i.e., $L$ lives only on $\cW|_\cY$ since $L(f)=0$.
Hence, $\dim\cW|_\cY = n$ and the theorem holds by the induction hypothesis.
\end{proof}

The previous proof is the original proof by Richter and only the mathematical language is updated.

\begin{rem}[see \cite{didioCone22}]
Replacing integration by finitely many point evaluations was already used and investigated by C.\ F.\ Gau{\ss} \cite{gauss15}.\index{Gauss@Gau\ss, C.\ F.}
The $k$-atomic representing measures from \Cref{thm:richter} are therefore also called (\emph{Gaussian}) \emph{cubature formulas}.\index{cubature formula!Gaussian}

The history of \Cref{thm:richter} is confusing and the literature is often misleading.
We therefore list in chronological order previous versions or versions which appeared almost at the same time.
The conditions of these versions (including Richter) are the following:
\begin{enumerate}[\bfseries\;\;(A)]
\item A.\ Wald 1939\footnote{Received: February 25, 1939. Published: September 1939.} \cite[Prop.\ 13]{wald39}:\index{Wald, A.}\index{Theorem!Wald} $\cX = \rset$ and $f_i(x) = |x-x_0|^{d_i}$ with $d_i\in\nset_0$, $0\leq d_1 < d_2 < \dots < d_n$, and $x_0\in\cX$.\medskip

\item P.\ C.\ Rosenbloom 1952 \cite[Cor.\ 38e]{rosenb52}:\index{Rosenbloom, P.\ C.}\index{Theorem!Rosenbloom} $(\cX,\cA)$ a measurable space and $f_i$ bounded measurable functions.\medskip

\item H.\ Richter 1957\footnote{Received: December 27, 1956. Published: April, 1957.} \cite[Satz 4]{richte57}: $(\cX,\cA)$ a measurable space and $f_i$ measurable functions.\medskip

\item M.\ V.\ Tchakaloff 1957\footnote{Published: July-September, 1957} \cite[Thm.\ II]{tchaka57}: $\cX\subset\rset^n$ compact and $f_i$ monomials of degree at most $d$.\label{item:tchakaloff}\index{Tchakaloff, M.\ V.}\index{Theorem!Tchakaloff}\medskip

\item W.\ W.\ Rogosinski 1958\footnote{Received: August 22, 1957. Published: May 6, 1958.} \cite[Thm.\ 1]{rogosi58}: $(\cX,\cA)$ measurable space and $f_i$ measurable functions.\label{item:rosenbloom}\index{Rogosinski, W.\ W.}\index{Theorem!Rogosinski}
\end{enumerate}
From this list we see that Tchakaloff's result (\ref{item:tchakaloff}) from 1957 is a special case of Rosenbloom's result (\ref{item:rosenbloom}) from 1952 and that the general case was proved by Richter and Rogosinski almost about at the same time, see the exact dates in the footnotes.
If one reads  Richter's paper, one might think at first glance that he treats only the one-dimensional case, but a closer look reveals that his Proposition (Satz) 4 covers actually the  general case of measurable functions.
Rogosinski treats the one-dimensional case, but states at the end of the introduction of \cite{rogosi58}:
\begin{quote}
Lastly, the restrictions in this paper to moment problems of dimension one is hardly essential.
Much of our geometrical arguments carries through, with obvious modifications, to any finite number of dimensions, and even to certain more general measure spaces.
\end{quote}
The above proof of \Cref{thm:richter}, and likewise the one in \cite[Theorem 1.24]{schmudMomentBook}, are nothing but modern formulations of the proofs of Richter and Rogosinski without additional arguments.
Note that Rogosinki's paper \cite{rogosi58} was submitted about a half year after the appearance of Richter's \cite{richte57}.

It might be of interest that the general results of Richter and Rogosinski from 1957/58 can be derived from Rosenbloom's Theorem from 1952, see Problem \ref{prob:richterFromRosen}.
With that wider historical perspective in mind it might be justified to call \Cref{thm:richter} also the \emph{Richter--Rogosinski--Rosenbloom Theorem}.\index{Richter--Rogosinski--Rosenbloom Theorem}\index{Theorem!Richter--Rogosinski--Rosenbloom}

\Cref{thm:richter} was overlooked in the modern literature on truncated polynomial moment problems.
The problem probably arose around 1997/98 when it was stated as an open problem in a published paper.\footnote{We do not give the references for this and subsequent papers who reproved \Cref{thm:richter}.}
The paper \cite{richte57} and numerous works of J.\ H.\ B.\ Kemperman were not included back then.
Especially \cite[Thm.\ 1]{kemper68} where Kemperman fully states the general theorem (\Cref{thm:richter}) and attributed it therein to Richter and Rogosinski is missing.
Later on, this missing piece was not added in several other works.
The error continued in the literature for several years and \Cref{thm:richter} was reproved in several papers in weaker forms.
Even nowadays papers appear not aware of \Cref{thm:richter} or of the content of \cite{richte57}.
\exmsymbol
\end{rem}

\section*{Problems}%%%%%%%%%%%%%%%%%%%%%
\addcontentsline{toc}{section}{Problems}

\begin{prob}\label{prob:st}
Let $t\in [0,\infty)$ and
\[s(t) := (1,0,t,0,t^2,0,\dots).\]
\begin{enumerate}[\bfseries\qquad a)]
\item Show that $s(t)$ is a moment sequence. What is a representing measure?

\item Is $s(t)$ determinate?
\end{enumerate}
\end{prob}

\begin{prob}\label{prob:momFunctionals}
In Example \ref{exm:momFunctionals} (a), prove or disprove the following statements:
\begin{enumerate}[\bfseries\qquad a)]
\item If for every $x\in\cX\setminus\{y\}$ there exists a function $f\in\cV$ with $f(x) \neq f(y)$, then $\delta_y$ is unique and $L$ is determinate.

\item If there exists a point $x\in\cX\setminus\{y\}$ with $f(x) = f(y)$, then $\delta_x$ and $\delta_y$ are two different representing measures of $L$, i.e., $L$ is indeterminate.
\end{enumerate}
\end{prob}

\begin{prob}\label{prob:momFuncDeterminate}
Prove \Cref{prop:momFuncDeterminate}.
\emph{Hint}: Use the Stone--Weierstrass Theorem.
\end{prob}

\begin{prob}\label{prob:momFuncIndeterminate}
Prove \Cref{prop:momFuncIndeterminate}.
\end{prob}

\begin{prob}\label{prob:convexLV}
Prove \Cref{lem:convexLV}.
\end{prob}

\begin{prob}\label{prob:cLclosed}
Non-closed and closed moment cones $\cL(\cV)$.
Let $\cV$ be finite-dimensional.
\begin{enumerate}[\bfseries\qquad a)]
\item Give an example of $\cL$ which is not closed.
Give an explicit $L\in\partial\cL\setminus\cL$.

\item Show that if there exists a $e\in\cV\subseteq\cC(\cX,\rset)$ with $e(x)>0$ for all $x\in\cX$ and let $\cX\subset\rset^n$ be compact, then $\cL$ is closed.
\end{enumerate}
\end{prob}

\begin{prob}\label{prob:cLrset}
Give an example of a finite-dimensional $\cV$ with $\cL(\cV) \cong \rset^{\dim\cV}$.
\end{prob}

\begin{prob}\label{prob:richterFromRosen}
Show that \Cref{thm:richter} follows from Rosenbloom's Theorem, i.e., show that the additional assumption that all $f_i$ are bounded on the measurable space $(\cX,\cA)$ can be omitted.
\end{prob}

\begin{prob}
Let $n\in\nset$.
Let $A,B\subseteq\rset^n$ be open and convex with $B\subseteq A$.
Show that either
\[A=B \qquad\text{or}\qquad \inter (A\setminus B) = A\setminus\overline{B} \neq \emptyset.\]
This argument is implicitly used in the proof of \Cref{thm:richter}.
For more on convex sets, see e.g.\ \cite{rock72} or \cite{schne14}. 
\end{prob}

%%%%%%%%%%%%%%%%%%%%%%%%%%%%%%%%%%%%%%%%%%%%%%%%
\chapter{Adapted Spaces and Choquet Theory}%%%%%
\label{ch:adapted}%%%%%%%%%%%%%%%%%%%%%%%%%%%%%%
%%%%%%%%%%%%%%%%%%%%%%%%%%%%%%%%%%%%%%%%%%%%%%%%

In the previous chapter we introduced moments, moment sequences, and moment functional, especially the truncated versions.
Now we will go to \emph{adapted spaces} and \emph{Choquet's Theory} to determine basic criteria when a linear function
\[L:\cV\to\rset\]
is a moment functional.
The theory of adapted spaces is called after Gustave Choquet,\index{Choquet!Gustave} see \cite{choquet69,phelpsLectChoquetTheorem}.

\section{Adapted Spaces of Continuous Functions}
%%%%%%%%%%%%%%%%%%%%%%%%%%%%%%%%%%%%%%%%%%%%%%%%%

\begin{dfn}
Let $\cX$ be a set and $\cV$ be a real vector space of functions $f:\cX\to\rset$.
We define
\[\cV_+ := \{f\in\cV \,|\, f\geq 0\}.\label{dfn:cV+}\]
\end{dfn}

\begin{dfn}\label{dfn:dominate}
Let $\cX$ be a locally compact Hausdorff space and $f,g\in\cat(\cX,\rset)_+$.
If, for any $\varepsilon>0$, there is an $h_\varepsilon\in\cat_c(\cX,\rset)$ such that
\[g \leq \varepsilon f + h_\varepsilon,\]
then we say $f$ \emph{dominates}\index{dominate} $g$.
\end{dfn}

Equivalent expressions are the following.

\begin{lem}[see e.g.\ {\cite[Lem.\ 1.4]{schmudMomentBook}}]\label{lem:adapted}
Let $\cX$ be a locally compact Hausdorff space and let $f,g\in\cat(\cX,\rset)_+$.
Then the following are equivalent:
\begin{enumerate}[(i)]
\item $f$ dominates $g$.

\item For every $\varepsilon>0$, there exists a compact set $K_\varepsilon\subseteq\cX$ such that
\[g(x) \leq \varepsilon\cdot f(x)\]
for all $x\in\cX\setminus K_\varepsilon$.

\item For every $\varepsilon>0$, there exists an $\eta_\varepsilon\in\cat_c(\cX,\rset)$ with $0 \leq \eta_\varepsilon\leq 1$ such that
\[g \leq \varepsilon\cdot f + \eta_\varepsilon\cdot g.\]
\end{enumerate}
\end{lem}
\begin{proof}
See Problem \ref{prob:adaptedLem}.
\end{proof}

\begin{dfn}\label{dfn:adaptedSpace}
Let $\cX$ be a locally compact Hausdorff space and let $E\subseteq\cat(\cX,\rset)$ be a vector space.
Then $E$ is called \emph{adapted space}\index{space!adapted}\index{adapted!space}, if the following conditions hold:
\begin{enumerate}[(i)]
\item $E = E_+ - E_+$,

\item for every $x\in\cX$, there exists a $f\in E_+$ such that $f(x) > 0$, and

\item every $g\in E_+$ is dominated by some $f\in E_+$.
\end{enumerate}
\end{dfn}

The space $\cat_c(\cX,\rset)_+$ is of special interest because of the Riesz Representation Theorem.
The following result shows that any $g\in\cat_c(\cX,\rset)_+$ is dominated (and even bounded) by some $f\in E_+$ for any given adapted space $E\subseteq\cat(\cX,\rset)$.

\begin{lem}\label{lem:adaptedCompact}
Let $\cX$ be a locally compact Hausdorff space, let $g\in\cat_c(\cX,\rset)_+$, and let $E\subseteq\cat(\cX,\rset)$ be an adapted space.
Then there exists a $f\in E_+$ such that $f\geq g$.
\end{lem}
\begin{proof}
See Problem \ref{prob:adaptedCompact}.
\end{proof}

\begin{exm}
Let $n\in\nset$ and let $\cX\subseteq\rset^n$ be closed.
Then the space
\[E=\rset[x_1,\dots,x_n]\]
of polynomials on $\cX$ is an adapted space, see Problem \ref{prob:adaptedPolynomials}.
\exmsymbol
\end{exm}

\section{The Basic Representation Theorem}%%%
%%%%%%%%%%%%%%%%%%%%%%%%%%%%%%%%%%%%%%%%%%%%%

One important reason adapted spaces have been introduced is to get the following representation theorem.
It is a general version of \Cref{thm:haviland} and will be used to solve most moment problems in an efficient way.

\begin{basrepthm}[see e.g.\ {\cite[Thm.\ 34.6]{choquet69}}]\label{thm:basicrepresentation}\index{Theorem!Basic Representation}\index{Basic Representation Theorem}\index{representation!Theorem!Basic}
Let $\cX$ be a locally compact Hausdorff space, let $E\subseteq \cat(\cX,\rset)$ be an adapted subspace, and let
\[L:E\to\rset\]
be a linear functional.
Then the following are equivalent:
\begin{enumerate}[(i)]
\item The functional $L$ is $E_+$-positive, i.e., $L(f)\geq 0$ for all $f\in E_+$.

\item $L$ is a moment functional, i.e., there exists a (Radon) measure $\mu$ on $\cX$ such that
\begin{enumerate}[(a)]
\item all $f\in E$ are $\mu$-integrable and,

\item  for all $f\in E$,
\[L(f) = \int_\cX f(x)~\diff\mu(x).\]
\end{enumerate}
\end{enumerate}
\end{basrepthm}

The proof of this very general result exceeds the time available in this course.
The complete proof with all technical details can be found in my lecture notes on T-systems
\begin{center}
\href{https://arxiv.org/abs/2403.04548}{https://arxiv.org/abs/2403.04548 Thm.\ 2.9}
\end{center}
or in the corrected lecture notes on my homepage
\begin{center}
\href{https://www.uni-konstanz.de/zukunftskolleg/community/philipp-di-dio/}{https://www.uni-konstanz.de/zukunftskolleg/community/philipp-di-dio/.}
\end{center}

\section{Classical Moment Problems}
%%%%%%%%%%%%%%%%%%%%%%%%%%%%%%%%%%%%

In this section we give several classical solutions of moment problems: the Stieltjes, Hamburger, and Hausdorff moment problem, as well as Haviland's Theorem.
We ordered the results chronologically.

The first moment problem was solved by T.\ J.\ Stieltjes\index{Stieltjes, T.\ J.} in 1894 \cite{stielt94}.
He was the first who fully stated the moment problem, solved the first one, and by doing that also introduced the integral theory named after him: the \emph{Stieltjes integral}.

\begin{stielthm}[\cite{stielt94}]\label{thm:stieltjesMP}\index{Theorem!Stieltjes}\index{Stieltjes' Theorem}\index{moment!problem!Stieltjes|see{Theorem, Stieltjes}}
Let $s = (s_i)_{i\in\nset_0}$ be a real sequence.
Then the following are equivalent:
\begin{enumerate}[(i)]
\item $s$ is a $[0,\infty)$-moment sequence (Stieltjes moment sequence).

\item $L_s(p)\geq 0$ for all $p\in\pos([0,\infty))$.
\end{enumerate}
\end{stielthm}
\begin{proof}
See Problem \ref{prob:stieltjes}.
\end{proof}

The next moment problem was solved by H.\ L.\ Hamburger\index{Hamburger, H.\ L.}.

\begin{hamthm}[{\cite[Satz X and Existenztheorem (§8, p.\ 289)]{hamburger20}}]\label{thm:hamburgerMP}\index{Hamburger's Theorem}\index{Theorem!Hamburger}\index{moment!problem!Hamburger|see{Theorem, Hamburger}}
Let $s = (s_i)_{i\in\nset_0}$ be a real sequence.
Then the following are equivalent:
\begin{enumerate}[(i)]
\item $s$ is a $\rset$-moment sequence (Hamburger moment sequence or short moment sequence).

\item $L_s(p)\geq 0$ for all $p\in\pos(\rset)$.

\end{enumerate}
\end{hamthm}
\begin{proof}
See Problem \ref{prob:hamburger}.
\end{proof}

Shortly after Hamburger, in 1921, the moment problem on $[0,1]$ was solved by F.\ Hausdorff\index{Hausdorff, F.}.

\begin{hausthm}[{\cite[Satz II and III]{hausdo21}}]\label{thm:hausdorffMP}\index{Hausdorff's Theorem}\index{Theorem!Hausdorff}\index{moment!problem!Hausdorff|see{Theorem, Hausdorff}}
Let $s = (s_i)_{i\in\nset_0}$ be a real sequence.
Then the following are equivalent:
\begin{enumerate}[(i)]
\item $s$ is a $[0,1]$-moment sequence (Hausdorff moment sequence).

\item $L_s(p)\geq 0$ for all $p\in\pos([0,1])$.

\end{enumerate}
\end{hausthm}
\begin{proof}
See Problem \ref{prob:hausdorff}.
\end{proof}

A general approach to solve the $K$-moment problem for any closed $K\subseteq\rset^n$, $n\in\nset$, was presented by E.\ K.\ Haviland\index{Haviland, E.\ K.} in \cite[Theorem]{havila36}, see also \cite[Theorem]{havila35} for the earlier case $K=\rset^n$.
He no longer used continued fractions but employed the Riesz Representation Theorem, i.e., representing a linear functional by integration, and connected the existence of a representing measure to the non-negativity of the linear functional on
\begin{equation}\label{eq:posKdfn}
\pos(K) := \{f\in\rset[x_1,\dots,x_n] \,|\, f\geq 0\ \text{on}\ K\}.
\end{equation}

\begin{havithm}[{\cite[Theorem]{havila36}}]\label{thm:haviland}\index{Haviland's Theorem}\index{Theorem!Haviland}
Let $n\in\nset$, let $K\subseteq\rset^n$ be closed, and let $s = (s_\alpha)_{\alpha\in\nset_0^n}$ be a real sequence.
Then the following are equivalent:
\begin{enumerate}[(i)]
\item $s$ is a $K$-moment sequence.

\item $L_s(p)\geq 0$ for all $p\in\pos(K)$.
\end{enumerate}
\end{havithm}
\begin{proof}
See Problem \ref{prob:haviland}.
\end{proof}

In \cite[Theorem]{havila35} Haviland proves ``only'' the case $K=\rset^n$ with the extension method by M.\ Riesz.\index{Riesz, M.}
In \cite[Theorem]{havila36} this is extended to any closed $K\subseteq\rset^n$.
The idea to do so is attributed by Haviland to A.\ Wintner\index{Wintner, A.} \cite[p.\ 164]{havila36}:
\begin{quote}
A.\ Wintner has subsequently suggested that it should be possible to extend this result [\cite[Theorem]{havila35}] by requiring that the distribution function [measure] solving the problem have a spectrum [support] contained in a preassigned set, a result which would show the well-known criteria for the various standard special momentum problems (Stieltjes, Herglotz [trigonometric], Hamburger, Hausdorff in one or more dimensions) to be put particular cases of the general $n$-dimensional momentum problem mentioned above.
The purpose of this note [\cite{havila36}] is to carry out this extension.
\end{quote}

In \cite{havila36} after the general Theorem \ref{thm:haviland} Haviland then goes through all the classical results (Theorems \ref{thm:stieltjesMP} to \ref{thm:hausdorffMP}, and the Herglotz\index{Herglotz!moment problem}\index{moment!problem!Herglotz} (trigonometric) moment problem\index{trigonometric!moment problem}\index{moment!problem!trigonometric} on the unit circle $\tset$, which we did not included here) and shows how all these results (i.e., conditions on the sequences) are recovered from the at this point known representations of non-negative polynomials.

For the \emph{Hamburger moment problem}\index{Hamburger moment problem|see{Theorem, Hamburger}} (\Cref{thm:hamburgerMP}) Haviland uses
\begin{equation}\label{eq:posR}
\pos(\rset) = \left\{ f^2 + g^2 \,\middle|\, f,g\in\rset[x]\right\}
\end{equation}
which was already known to D.\ Hilbert\index{Hilbert, D.} \cite{hilbert88}.
%We prove a stronger version of (\ref{eq:posR}) in \Cref{cor:nonnegRRx}.
For the \emph{Stieltjes moment problem}\index{Stieltjes moment problem|see{Theorem, Stieltjes}} (\Cref{thm:stieltjesMP}) he uses
\begin{equation}\label{eq:pos0infty1}
\pos([0,\infty)) = \left\{ f_1^2 + f_2^2 + x\cdot (g_1^2 + g_2^2) \,\middle|\, f_1,f_2,g_1,g_2\in\rset[x]\right\}
\end{equation}
with the reference to G.\ P\'olya\index{P\'olya, G.} and G.\ Szeg{\"o}\index{Szeg\"o, G.} (previous editions of \cite{polya64,polya70}).
In \cite[p.\ 82, ex.\ 45]{polya64} the representation (\ref{eq:pos0infty1}) is still included, while it was already known before, see e.g.\ \cite[p.\ 6, footnote]{shohat43}, that
\begin{equation}\label{eq:pos0infty2}
\pos([0,\infty)) = \left\{ f^2 + x\cdot g^2 \,\middle|\, f,g\in\rset[x]\right\}
\end{equation}
is sufficient.
Also in \cite[Prop.\ 3.2]{schmudMomentBook} the representation (\ref{eq:pos0infty1}) is used, not the simpler representation (\ref{eq:pos0infty2}).
%We prove a stronger version of (\ref{eq:pos0infty2}) in \Cref{cor:nonneg0inftyRx}.

For the $[-1,1]$-moment problem, Haviland uses
\begin{equation}\label{eq:pos-11}
\pos([-1,1]) = \left\{ f^2 + (1-x^2)\cdot g^2 \,\middle|\, f,g\in\rset[x]\right\}.
\end{equation}
For the \emph{Hausdorff moment problem}\index{Hausdorff moment problem|see{Theorem, Hausdorff}} (\Cref{thm:hausdorffMP}) he uses that any strictly positive polynomial on $[0,1]$ is a linear combination of
\begin{equation}\label{eq:lukacs}
x^m\cdot (1-x)^{p}
\end{equation}
with $m,p\in\nset_0$, $p\geq m$, and with non-negative coefficients.

Haviland gives this with the references to a previous edition of \cite{polya70}.
But this result is actually due to S.\ N.\ Bernstein\index{Bernstein, S.\ N.} \cite{bernstein12,bernstein15}.

\begin{bernthm}[\cite{bernstein12} for (i), \cite{bernstein15} for (ii); or see e.g.\ {\cite[p.\ 30]{achieser56}} or {\cite[Prop.\ 3.4]{schmudMomentBook}}]\label{thm:bernstein}\index{Theorem!Bernstein}\index{Bernstein!Theorem}
Let $f\in\cat([0,1],\rset)$ and let
\begin{equation}\label{eq:bernPoly}
B_{f,d}(x) := \sum_{k=0}^d \binom{d}{k}\cdot x^k\cdot (1-x)^{d-k}\cdot f\left(\frac{k}{d}\right)
\end{equation}
be the \emph{Bernstein polynomials}\index{polynomial!Bernstein}\index{Bernstein!polynomial} of $f$ of degree $d\in\nset$.
Then the following hold:
\begin{enumerate}[(i)]
\item The polynomials $B_{f,d}$ converge uniformly on $[0,1]$ to $f$, i.e.,
\[\|f-B_{f,d}\|_\infty \xrightarrow{d\to\infty} 0.\]

\item If additionally $f\in\rset[x]$ with $f>0$ on $[0,1]$, then there exist a constant $D=D(f)\in\nset$ and constants $c_{k,l}\geq 0$ for all $k,l=0,\dots,D$ such that
\[f(x) = \sum_{k,l=0}^D c_{k,l}\cdot x^k\cdot (1-x)^l.\]

\item The statements (i) and (ii) also hold on $[0,1]^n$ for any $n\in\nset$. Especially every $f\in\rset[x_1,\dots,x_n]$ with $f>0$ on $[0,1]^n$ is of the form
\[f(x) = \sum_{\alpha_1,\dots,\beta_n=0}^D c_{\alpha_1,\dots,\beta_n}\cdot x_1^{\alpha_1}\cdots x_n^{\alpha_n}\cdot (1-x_1)^{\beta_1}\cdots (1-x_n)^{\beta_n}\]
for some $D\in\nset$ and $c_{\alpha_1,\dots,\beta_n}\geq 0$.
\end{enumerate}
\end{bernthm}

The multidimensional statement (iii) follows from the classical one-dimensional cases (i) and (ii), see e.g.\ \cite{hildeb33}, \cite{schnabel68}, and \cite[p.\ 51]{lorentz86}.
For this and more on Bernstein polynomials see e.g.\ \cite{lorentz86}.
The experts in real algebraic geometry will of course recognize (iii) as a special case of \emph{Schmüdgen's Theorem}\index{Theorem!Schm\"{u}dgen}\index{Schm\"{u}dgen!Theorem} \cite[Cor.\ 3]{schmud91}. (iii) has been known long before the general case was first proved in \cite{schmud91}.

\Cref{thm:bernstein} only holds for $f>0$.
Allowing zeros at the interval end points is possible and gives the following ``if and only if''-statement.

\begin{cor}\label{cor:bernstein}
Let $f\in\rset[x]\setminus\{0\}$.
Then the following are equivalent:
\begin{enumerate}[(i)]
\item $f>0$ on $(0,1)$.

\item For some $D\in\nset$,
\[f(x) = \sum_{i=0}^D c_{k,l}\cdot x^l\cdot (1-x)^k\]
with $c_{k,l}\geq 0$ for all $k,l=0,\dots,D$ and $c_{k',l'}>0$ at least once.
\end{enumerate}
\end{cor}
\begin{proof}
See Problem \ref{prob:bernstein}.
\end{proof}

On $[-1,1]$ a strengthened version of \Cref{thm:bernstein} (ii) is attributed to F.\ Luk\'acs\index{Luk\'acs, F.} \cite{lukacs18} (\emph{Luk\'acs Theorem}).\index{Theorem!Luk\'acs|see{Luk\'acs--Markov}}\index{Luk\'acs Theorem|see{Luk\'acs--Markov Theorem}}
Note that Luk\'acs in \cite{lukacs18} reproves several results/formulas, which already appeared in a work by M.\ R.\ Radau\index{Radau, M.\ R.} \cite{radau80}, as also pointed out by L.\ Brickman\index{Brickman, L.} \cite[p.\ 196]{brickman59}.
Additionally, in \cite[p.\ 61, footnote 4]{kreinMarkovMomentProblem} M.\ G.\ Kre\u{\i}n\index{Kre\u{\i}n, M.\ G.} and A.\ A.\ Nudel'man\index{Nudel'man, A.\ A.} state that A.\ A.\ Markov\index{Markov, A.\ A.} proved a more precise version of Luk\'acs Theorem already in 1906 \cite{markov06},\footnote{We do not have access to \cite{markov06} and can therefore neither confirm nor decline this statement.} see also \cite{markov95}.
Kre\u{\i}n and Nudel'man called it \emph{Markov's Theorem}.\index{Markov's Theorem|see{Luk\'acs--Markov Theorem}}\index{Theorem!Markov|see{Luk\'acs--Markov}}
It is the following.
%We will get a stronger version in the \Cref{thm:luma2}.

\begin{lumathm}[\cite{markov06} or e.g.\ \cite{lukacs18}, {\cite[p.\ 61, Thm.\ 2.2]{kreinMarkovMomentProblem}}]\index{Luk\'acs--Markov Theorem}\index{Theorem!Luk\'acs--Markov}\label{thm:luma}
Let $-\infty < a < b < \infty$ and let $p\in\rset[x]$ be with $\deg p = n$ and $p\geq 0$ on $[a,b]$.
Then the following hold:
\begin{enumerate}[(i)]
\item If $\deg p = 2m$ for some $m\in\nset_0$, then $p$ is of the form
\[p(x) = f(x)^2 + (x-a)(b-x)\cdot g(x)^2\]
for some $f,g\in\rset[x]$ with $\deg f = m$ and $\deg g = m-1$.

\item If $\deg p = 2m+1$ for some $m\in\nset_0$, then $p$ is of the form
\[p(x) = (x-a)\cdot f(x)^2 + (b-x)\cdot g(x)^2\]
for some $f,g\in\rset[x]$ with $\deg f = \deg g = m$.
\end{enumerate}
\end{lumathm}

For case (i) note that the relation
\begin{equation}\label{eq:abProduct}
(x-a)(b-x) = \frac{1}{b-a}\left[ (x-a)^2(b-x) + (x-a)(b-x)^2 \right]
\end{equation}
implies
\begin{equation}\label{eq:posab}
\pos([a,b]) = \left\{f(x)^2 + (x-a)\cdot g(x)^2 + (b-x)\cdot h(x)^2 \,\middle|\, f,g,h\in\rset[x]\right\}.
\end{equation}
The special part about the \Cref{thm:luma} is the degree bound on the polynomials $f$ and $g$.
Equation (\ref{eq:abProduct}) destroys these degree bounds, since we have to go one degree higher.

%In the \Cref{thm:luma2} we will see how from \Cref{thm:karlinPosab} an even stronger version follows which describes the polynomials $f$ and $g$ more precisely and up to a certain point uniquely.
In \cite[p.\ 61, Thm.\ 2.2 and p.\ 373, Thm.\ 6.4]{kreinMarkovMomentProblem} the \Cref{thm:luma} is called \emph{Markov--Luk\'acs Theorem}\index{Markov--Luk\'acs Theorem|see{Luk\'acs--Markov Theorem}}\index{Theorem!Markov--Luk\'acs|see{Luk\'acs--Markov}} since Markov gave the more precise version much earlier than Luk\'acs.\footnote{Here, we use the alphabetical order, as usually used in naming theorems.}
In \cite{havila36} Haviland uses this result without any reference or attribution to either Luk\'acs or Markov.

For the two-dimensional Hausdorff moment problem, Haviland uses with a reference to \cite{hildeb33} that any polynomial $f\in\rset[x,y]$ which is strictly positive on $[0,1]^2$ is a linear combination of
\[x^m\cdot y^n\cdot (1-x)^{p}\cdot (1-y)^{q}\]
with $n,m,q,p\in\nset_0$ and non-negative coefficients.
This is actually \Cref{thm:bernstein} (iii).

T.\ H.\ Hildebrandt\index{Hildebrandt, T.\ H.} and I.\ J.\ Schoenberg\index{Schoenberg, I.\ J.} \cite{hildeb33} already solved the moment problem on $[0,1]^n$ for all $n\in\nset$ getting the same result as Haviland.
The idea of using $\pos(K)$-descriptions to solve the moment problem was therefore already used by Hildebrandt and Schoenberg in 1933 \cite{hildeb33}, before Haviland uses this in \cite{havila35} and generalized this in \cite{havila36} as suggested to him by Wintner.

With these broader historical remarks we see that of course more people are connected to Theorem \ref{thm:haviland}.
It might also be appropriate to call Theorem \ref{thm:haviland} the \emph{Haviland--Wintner}\index{Theorem!Haviland--Wintner}\index{Haviland--Wintner Theorem} or \emph{Haviland--Hildebrandt--Schoenberg--Wintner Theorem}.\index{Theorem!Haviland--Hildebrandt--Schoenberg--Wintner}\index{Haviland--Hildebrandt--Schoenberg--Wintner Theorem}
But as so often, the list of contributors is long (and maybe even longer) and hence the main contribution (the general proof) is rewarded by calling it just \emph{Haviland's Theorem}.

The last classical moment problem which we want to mention on the long list was solved by K.\ I.\ {\v{S}}venco \cite{svenco39}.\index{Svenco, K.\ I.@{\v{S}}venco, K.\ I.}

\begin{svethm}[\cite{svenco39}]\index{Svenco's Theorem@{\v{S}}venco's Theorem}\index{Theorem!Svenco@{\v{S}}venco}\index{moment!problem!Svenco@{\v{S}}venco|see{Theorem, {\v{S}}venco}}\label{thm:svenco}
Let $s=(s_i)_{i\in\nset_0}$ be a real sequence.
Then the following are equivalent:
\begin{enumerate}[(i)]
\item $s$ is a $(-\infty,0]\cup [1,\infty)$-moment sequence.

\item $L_s(p)\geq 0$ for all $p\in\pos((-\infty,0]\cup[1,\infty))$.

\end{enumerate}
\end{svethm}

The general case of \Cref{thm:svenco} on
\begin{equation}\label{eq:zaun}
\rset\setminus \bigcup_{i=1}^n (a_i,b_i)
\end{equation}
for any $n\in\nset$ and
\[a_1 < b_1 < \dots < a_n < b_n\]
was proved by V.\ A.\ Fil'\v{s}tinski\v{\i}\index{Fil'\v{s}tinski\v{\i}, V.\ A.} \cite{filsti64}.
All non-negative polynomials on (\ref{eq:zaun}) can be explicitly written down.
More precisely, all moment problems on closed and semi-algebraic sets $K\subseteq\rset$ follow nowadays easily from \Cref{thm:haviland} resp.\ the \Cref{thm:basicrepresentation} and some well established results from real algebraic geometry, see e.g.\ \cite[Prop.\ 2.7.3]{marshallPosPoly}.

\Cref{thm:haviland} was important to give the solutions of the classical moment problem, i.e., mostly one-dimensional cases.
After that it was no longer used and only became important again when descriptions of strictly positive and non-negative polynomials on $K\subseteq\rset^n$ with $n\geq 2$ be came available.
This process was started with \cite{schmud91} and real algebraic geometry was revived by it.

\section*{Problems}%%%%%%%%%%%%%%%%%%%%%
\addcontentsline{toc}{section}{Problems}

\begin{prob}\label{prob:adaptedLem}
Prove \Cref{lem:adapted}.
\end{prob}

\begin{prob}\label{prob:compactAdapted}
Let $\cX$ be a compact topological Hausdorff space and let $E\subseteq\cat(\cX,\rset)$ be a subspace such that there exists an $e\in E$ such that $e(x)>0$ for all $x\in\cX$. Show that $E$ is an adapted space.
\end{prob}

\begin{prob}\label{prob:adaptedPolynomials}
Let $n\in\nset$ and let $\cX\subseteq\rset^n$ be closed.
Show that $\rset[x_1,\dots,x_n]$ on $\cX$ is an adapted space.
\end{prob}

\begin{prob}\label{prob:adaptedPolynomials2}
Let $n\in\nset$, let $\cX\subseteq\rset^n$ be closed, and let $E\subseteq\rset[x_1,\dots,x_n]$ be an adapted space.
Show that if $E$ is finite dimensional, then $\cX$ is compact.
\end{prob}

\begin{prob}
Let $n\in\nset$.
\begin{enumerate}[\bfseries\qquad a)]
\item Let $\cX=\rset^n$ and let $E\subseteq\cC(\cX,\rset)$ be an adapted space.
Show that $E$ is infinite dimensional.

\item What is the minimal requirement on $\cX\subseteq\rset^n$ such that any adapted space $E\subseteq\cC(\cX,\rset)$ is infinite dimensional?
\end{enumerate}
\end{prob}

\begin{prob}\label{prob:adaptedCompact}
Prove \Cref{lem:adaptedCompact}.
\end{prob}

\begin{prob}\label{prob:stieltjes}
Prove \Cref{thm:stieltjesMP} with the \Cref{thm:basicrepresentation}.
\end{prob}

\begin{prob}\label{prob:hamburger}
Prove \Cref{thm:hamburgerMP} with the \Cref{thm:basicrepresentation}.
\end{prob}

\begin{prob}\label{prob:hausdorff}
Prove \Cref{thm:hausdorffMP} with the \Cref{thm:basicrepresentation}.
\end{prob}

\begin{prob}\label{prob:haviland}
Prove \Cref{thm:haviland} with the \Cref{thm:basicrepresentation}.
\end{prob}

\begin{prob}\label{prob:bernstein}
Use \Cref{thm:bernstein} (ii) to prove \Cref{cor:bernstein}.
\end{prob}

%%%%%%%%%%%%%%%%%%%%%%%%%%%%%%%%%%%%%%%%%
%%%%%%%%%%%%%%%%%%%%%%%%%%%%%%%%%%%%%%%%%
\part{Linear Operators on Polynomials}%%%
%%%%%%%%%%%%%%%%%%%%%%%%%%%%%%%%%%%%%%%%%
%%%%%%%%%%%%%%%%%%%%%%%%%%%%%%%%%%%%%%%%%
\label{part:linOpPoly}

%\motto{Pure mathematics is, in its way, the poetry of logical ideas.\\ \medskip
%\ \hspace{1cm} \normalfont{Albert Einstein \cite{einstein35noether}}\index{Einstein, A.}}

\chapter{Representations of Linear Operators on Polynomials}%%%%
\label{ch:linOpRepr}%%%%%%%%%%%%%%%%%%%%%%%%%%%%%%%%%%%%%%%%%%%%
%%%%%%%%%%%%%%%%%%%%%%%%%%%%%%%%%%%%%%%%%%%%%%%%%%%%%%%%%%%%%%%%

In this chapter we finally start with our investigation of linear operators
\[T:\rset[x_1,\dots,x_n]\to\rset[x_1,\dots,x_n].\]

\section{Representations of General Linear Operators on $\rset[x_1,\dots,x_n]$}

To work with linear operators on polynomials we give at first representations of them.

\begin{thm}[folklore, \emph{canonical representation}]\label{thm:canonicalRepresentation}
Let $n\in\nset$ and let
\[T:\rset[x_1,\dots,x_n]\to\rset[x_1,\dots,x_n]\]
be linear.
Then, for all $\alpha\in\nset_0^n$, there exist unique polynomials $q_\alpha\in\rset[x_1,\dots,x_n]$ such that
\begin{equation}\label{eq:canRepr}
T = \sum_{\alpha\in\nset_0^n} q_\alpha\cdot\partial^\alpha.
\end{equation}
\end{thm}

The representation (\ref{eq:canRepr}) is called \emph{canonical representation}\index{representation!canonical}\index{canonical!representation} of $T$.

\begin{proof}
We proceed by induction over $d=|\alpha|\in\nset_0$.

\underline{$d=0$:}
Set
\[q_0 := T1,\]
i.e., $q_0$ is unique and $T_0 := q_0$ represents $T$ on $\rset[x_1,\dots,x_n]_{\leq 0} = \rset$.

\underline{$d\to d+1$:}
Assume, for $d\in\nset_0$, we have unique $q_\alpha\in\rset[x_1,\dots,x_n]$ for all $\alpha\in\nset_0^n$ with $|\alpha|\leq d$ such that
\[T_d := \sum_{\alpha\in\nset_0^n:|\alpha|\leq d} q_\alpha\cdot\partial^\alpha\]
represents $T$ on $\rset[x_1,\dots,x_n]_{\leq d}$.
Let $\beta\in\nset_0^n$ with $|\beta|=d+1$.
Then
\[Tx^\beta = \sum_{\alpha\in\nset_0^n:|\alpha|\leq d} q_\alpha\cdot\partial^\alpha x^\beta + q_\beta\cdot\partial^\beta x^\beta\]
implies
\[q_\beta := \frac{1}{\beta!}\cdot \left(Tx^\beta - \sum_{\alpha\in\nset_0^n:|\alpha|\leq d} q_\alpha\cdot\partial^\alpha x^\beta \right),\]
i.e., $q_\beta$ is unique and
\[T_{d+1} := \sum_{\alpha\in\nset_0^n:|\alpha|\leq d+1} q_\alpha\cdot\partial^\alpha\]
represents $T$ on $\rset[x_1,\dots,x_n]_{\leq d+1}$ with unique $q_\alpha$ for all $\alpha\in\nset_0^n$ with $|\alpha|\leq d+1$.

In summary, (\ref{eq:canRepr}) represents $T$ on all $\rset[x_1,\dots,x_n]$ with unique $q_\alpha$ for all $\alpha\in\nset_0^n$.
\end{proof}

\begin{prop}[{\cite[Lem.\ 5.1]{didio25KPosPresGen}}]
Let $n\in\nset$ and let
\[T:\rset[x_1,\dots,x_n]\to\rset[x_1,\dots,x_n] \quad\text{with}\quad T = \sum_{\alpha\in\nset_0^n} q_\alpha\cdot\partial^\alpha\]
be linear with unique $q_\alpha\in\rset[x_1,\dots,x_n]$ for all $\alpha\in\nset_0^n$.
Then the following are equivalent:
\begin{enumerate}[\;\, (i)]
\item $T\rset[x_1,\dots,x_n]_{\leq d} \subseteq \rset[x_1,\dots,x_n]_{\leq d}$ for all $d\in\nset_0$.

\item $\deg q_\alpha\leq |\alpha|$ for all $\alpha\in\nset_0^n$.
\end{enumerate}
\end{prop}
\begin{proof}
\underline{(i) $\Rightarrow$ (ii):}
Let $p\in\rset[x_1,\dots,x_n]$.
Since $\deg q_\alpha\leq |\alpha|$,
\[\deg (q_\alpha\cdot\partial^\alpha p) \leq \deg p\]
for all $\alpha\in\nset_0^n$, i.e., (ii) is proved.

\underline{(ii) $\Rightarrow$ (i):}
Assume to the contrary that there is an $\alpha\in\nset_0^n$ with $\deg q_\alpha > |\alpha|$.
Take the smallest of these $\alpha$'s with respect to the lexicographic order.
Then
\[\deg (Tx^\alpha) = \deg (q_\alpha\cdot\partial^\alpha x^\alpha) = \deg q_\alpha > |\alpha|\]
which contradicts (ii).
\end{proof}

\begin{prop}[{\cite[Rem.\ 2.3]{didio25hadamardLanger}}]
Let $n\in\nset$ and let
\[T:\rset[x_1,\dots,x_n]\to\rset[x_1,\dots,x_n]\]
be linear.
Then, for all $i\in\nset_0$, there are linear functionals
\[l_i:\rset[x_1,\dots,x_n]\to\rset\]
and polynomials $p_i\in\rset[x_1,\dots,x_n]$ such that
\begin{equation}\label{eq:tensorRepr}
T = \sum_{i\in\nset_0} l_i\cdot p_i, \quad\text{i.e.,}\quad Tf = \sum_{i\in\nset_0} l_i(f)\cdot p_i
\end{equation}
for all $f\in\rset[x_1,\dots,x_n]$.
\end{prop}
\begin{proof}
For every $\alpha\in\nset_0^n$, let
\[l_\alpha:\rset[x_1,\dots,x_n]\to\rset\]
be the linear functional defined by
\[l_\alpha(x^\beta) := \begin{cases}
1 & \text{if}\ \beta=\alpha,\\ 0 & \text{if}\ \beta\neq\alpha\end{cases}\]
and linear extension to all $\rset[x_1,\dots,x_n]$.
Then
\[Tf = \sum_{\alpha\in\nset_0^n} l_\alpha(Tf)\cdot x^\alpha\]
for all $f\in\rset[x_1,\dots,x_n]$, i.e.,
\[T = \sum_{\alpha\in\nset_0^n} (l_\alpha\circ T) \cdot x^\alpha.\]
Since $\nset_0^n$ is countable, (\ref{eq:tensorRepr}) is proved.
\end{proof}

\section{Representation of Diagonal Operators on $\rset[x_1,\dots,x_n]$}

A special class of linear operators on $\rset[x_1,\dots,x_n]$ are diagonal operators.

\begin{dfn}
Let $n\in\nset$.
A linear operator
\[T:\rset[x_1,\dots,x_n]\to\rset[x_1,\dots,x_n]\]
is called \emph{diagonal},\index{diagonal!operator}\index{operator!diagonal} if
\[Tx^\alpha = t_\alpha x^\alpha \quad\text{with}\quad t_\alpha\in\rset\]
for all $\alpha\in\nset_0^n$.
The sequence $t = (t_\alpha)_{\alpha\in\nset_0^n}$ is called \emph{diagonal sequence}\index{sequence!diagonal}\index{diagonal!sequence} of $T$.
\end{dfn}

Clearly, a diagonal operator is uniquely determined by its diagonal sequences and every real sequence $t = (t_\alpha)_{\alpha\in\nset_0^n}$ gives a diagonal operator.

\begin{thm}[{\cite[Rem.\ 4.2]{didio25hadamardLanger}}]\label{thm:diagonalRepr}
Let $n\in\nset$ and let
\[T:\rset[x_1,\dots,x_n]\to\rset[x_1,\dots,x_n]\]
be a linear operator.
Then the following are equivalent:
\begin{enumerate}[(i)]
\item $T$ is a diagonal operator, i.e., $Tx^\alpha = t_\alpha x^\alpha$ with unique $t_\alpha\in\rset$ for all $\alpha\in\nset_0^n$.

\item The operator $T$ is of the form
\begin{equation}\label{eq:diagRepr1}
T = \sum_{\alpha\in\nset_0^n} \frac{c_\alpha}{\alpha!}\cdot x^\alpha\cdot \partial^\alpha
\end{equation}
for a unique real sequence $c = (c_\alpha)_{\alpha\in\nset_0^n}$.
\end{enumerate}
If one of the equivalent statements (i) or (ii) holds, then the diagonal sequence $t=(t_\alpha)_{\alpha\in\nset_0^n}$ and the sequence $c=(c_\alpha)_{\alpha\in\nset_0^n}$ of coefficients $c_\alpha$ fulfill the relations
\begin{equation}\label{eq:talphacalphaRelations}
t_\alpha = \sum_{\beta\in\nset_0^n:\ \beta\preceq\alpha} \binom{\alpha}{\beta}\cdot c_\beta \qquad\text{and}\qquad c_\alpha = \sum_{\beta\in\nset_0^n:\ \beta\preceq\alpha} (-1)^{|\alpha-\beta|}\cdot \binom{\alpha}{\beta}\cdot t_\beta
\end{equation}
for all $\alpha\in\nset_0^n$.
\end{thm}
\begin{proof}
\underline{(i) $\Rightarrow$ (ii):}
At first we prove (\ref{eq:diagRepr1}) by induction over $d = |\alpha|$.

$d=0$:
Set
\[c_0 := T1 = t_0.\]
Then $T$ is represented by $T_0 := c_0 = t_0$ on $\rset[x_1,\dots,x_n]_{\leq 0} = \rset$.

$d\to d+1$:
Assume for all $\alpha\in\nset_0^n$ with $|\alpha|\leq d$ there are unique $c_\alpha$ such that $T$ is represented by
\[\sum_{\alpha\in\nset_0^n: |\alpha|\leq d} \frac{c_\alpha}{\alpha!}\cdot x^\alpha\cdot\partial^\alpha\]
on $\rset[x_1,\dots,x_n]_{\leq d}$.
Let $\beta\in\nset_0^n$ with $|\beta|=d+1$.
Then
\[t_\beta x^\beta = Tx^\beta = \sum_{\alpha\in\nset_0^n: |\alpha|\leq d} \frac{c_\alpha}{\alpha!}\cdot x^\alpha\cdot\partial^\alpha x^\beta + \frac{c_\beta}{\beta!}\cdot x^\beta\cdot\partial^\beta x^\beta\]
implies
\[c_\beta x^\beta = Tx^\beta - \sum_{\alpha\in\nset_0^n: |\alpha|\leq d} \frac{c_\alpha}{\alpha!}\cdot x^\alpha\cdot\partial^\alpha x^\beta\]
with unique $c_\beta\in\rset$.

In summary, (\ref{eq:diagRepr1}) represents $T$ on $\rset[x_1,\dots,x_n]$.

\underline{(ii) $\Rightarrow$ (i):}
By (\ref{eq:diagRepr1}),
\[Tx^\alpha = \sum_{\beta\in\nset_0^n} \frac{c_\beta}{\beta!}\cdot x^\beta\cdot\partial^\beta x^\alpha = \sum_{\beta\in\nset_0^n: \beta\preceq\alpha} \binom{\alpha}{\beta}\cdot x^\alpha = t_\alpha x^\alpha\]
with $t_\alpha\in\rset$, i.e., we proved (i).

We already proved the first equation of (\ref{eq:talphacalphaRelations}) in the step ``(ii) $\Rightarrow$ (i)''.
The second equation in (\ref{eq:talphacalphaRelations}) is given as Problem \ref{prob:calphaEquation}.
\end{proof}

\begin{exm}[{\cite[Exm.\ 4.3]{didio25hadamardLanger}}]
Let $n\in\nset$ and let
\[T:\rset[x_1,\dots,x_n]\to\rset[x_1,\dots,x_n],\quad f\mapsto Tf = f(0)\]
be the evaluation operator at $x=0$.
Then
\[Tx^\alpha = \delta_{0,\alpha}\cdot x^\alpha\]
for all $\alpha\in\nset_0^n$, i.e., $T$ is a diagonal operator and
\[T = \sum_{\alpha\in\nset_0^n} \frac{(-1)^{|\alpha|}}{\alpha!}\cdot x^\alpha\cdot\partial^\alpha,\]
i.e.,
\[t = (\delta_{0,\alpha})_{\alpha\in\nset_0^n} \qquad\text{and}\qquad c = \big((-1)^{|\alpha|}\big)_{\alpha\in\nset_0^n}. \tag*{$\circ$}\]
%\exmsymbol
\end{exm}

\begin{exm}[{\cite[Exm.\ 4.3]{didio25hadamardLanger}}]
Let $n\in\nset$ and let
\[\id:\rset[x_1,\dots,x_n]\to\rset[x_1,\dots,x_n],\quad f\mapsto \id f = f\]
be the identity map on $\rset[x_1,\dots,x_n]$.
Then
\[Tx^\alpha = x^\alpha \qquad\text{and}\qquad T = \sum_{\alpha\in\nset_0^n} \frac{\delta_{0,\alpha}}{\alpha!}\cdot x^\alpha\cdot\partial^\alpha,\]
i.e., $t= (1)_{\alpha\in\nset_0^n}$ and $c = (\delta_{0,\alpha})_{\alpha\in\nset_0^n}$.
\exmsymbol
\end{exm}

\section*{Problems}%%%%%%%%%%%%%%%%%%%%%
\addcontentsline{toc}{section}{Problems}

\begin{prob}
Let $n\in\nset$ and let
\[T = \sum_{\alpha\in\nset_0^n} q_\alpha\cdot\partial^\alpha\]
with $q_\alpha\in\rset[x_1,\dots,x_n]$ for all $\alpha\in\nset_0^n$.
Prove or disprove the following statements:
\begin{enumerate}[\bfseries\qquad a)]
\item If only finitely many $q_\alpha$ are non-zero, then $\dim T\rset[x_1,\dots,x_n] < \infty$.

\item If $\dim T\rset[x_1,\dots,x_n] < \infty$, then only finitely many $q_\alpha$ are non-zero.
\end{enumerate}
\end{prob}

\begin{prob}\label{prob:calphaEquation}
Prove
\[c_\alpha = \sum_{\beta\in\nset_0^n:\ \beta\preceq\alpha} (-1)^{|\alpha-\beta|}\cdot \binom{\alpha}{\beta}\cdot t_\beta\]
in (\ref{eq:talphacalphaRelations}) of \Cref{thm:diagonalRepr}.
\end{prob}

\begin{prob}\label{prob:scaling}
Let $n\in\nset$ and $y=(y_1,\dots,y_n)\in\rset^n$.
Define the \emph{scaling operator}\index{operator!scaling}\index{scaling!operator}
\[S_y:\rset[x_1,\dots,x_n]\to\rset[x_1,\dots,x_n]\]
by
\[(S_y f)(x) = f(x_1 y_1, \dots,x_n y_n)\]
for all $f\in\rset[x_1,\dots,x_n]$.
\begin{enumerate}[\bfseries\qquad a)]
\item Show that $S_y$ is a diagonal operator.

\item What is the diagonal sequences $t=(t_\alpha)_{\alpha\in\nset_0^n}$ and the sequence $c = (c_\alpha)_{\alpha\in\nset_0^n}$ of coefficients in (\ref{eq:diagRepr1})?
\end{enumerate}
\end{prob}

\begin{prob}
Let $\kset$ be a field.
\begin{enumerate}[\bfseries\qquad a)]
\item For $\kset = \cset$, how do the canonical representations of $T$ look like in the general case (\Cref{thm:canonicalRepresentation}) and the diagonal case (\Cref{thm:diagonalRepr})?

\item For a field $\kset$ with finite characteristic ($\mathrm{char}\, \kset < \infty$), how do the canonical representations of $T$ look like in the general case (\Cref{thm:canonicalRepresentation}) and the diagonal case (\Cref{thm:diagonalRepr})?
\end{enumerate}
\end{prob}

\chapter{The Topological Vector Spaces $\cset^{\nset_0^n}$ and $\cset[x_1,\dots,x_n]$}%%%
\label{ch:topVecSpaces}%%%%%%%%%%%%%%%%%%%%%%%%%%%%%%%%%%%%%%%%%%%%%%%%%%%%%%%%%%%%%%%%%%
%%%%%%%%%%%%%%%%%%%%%%%%%%%%%%%%%%%%%%%%%%%%%%%%%%%%%%%%%%%%%%%%%%%%%%%%%%%%%%%%%%%%%%%%%

We have seen in the previous chapter that we have to deal with infinite sums of operators on $\rset[x_1,\dots,x_n]$.
We also have to work with infinite sums and therefore limits in $\rset[x_1,\dots,x_n]$ and $\rset^{\nset_0^n}$.
Hence, we look at suitable topologies on $\rset[x_1,\dots,x_n]$ and $\rset^{\nset_0^n}$.
For that we can go without any problem to $\cset[x_1,\dots,x_n]$ and $\cset^{\nset_0^n}$.
We recommend \cite{treves67,koethe69TopVecSp1,koethe79TopVecSp2,pietsch72,schaef99} for further reading, especially \cite{treves67}.

\section{The Fr\'echet Space $\cset^{\nset_0^n}$}%%%
%%%%%%%%%%%%%%%%%%%%%%%%%%%%%%%%%%%%%%%%%%%%%%%%%%%%

\begin{dfn}
A \emph{Fr\'echet space}\index{space!Fr\'echet}\index{Fr\'echet!space} is a topological vector space which is
\begin{enumerate}[(i)]
\item metrizable (in particular, it is Hausdorff),

\item complete, and

\item locally convex.
\end{enumerate}
\end{dfn}

\begin{exm}[{\cite[p.\ 91, Exm. III]{treves67}}]\label{exm:frechet}
Let $n\in\nset$.
The space
\[\cset^{\nset_0^n} := \big\{ s = (s_\alpha)_{\alpha\in\nset_0^n} \,\big|\, s_\alpha\in\cset\ \text{for all}\ \alpha\in\nset_0^n\big\}\]
of all complex sequences indexed by $\nset_0^n$ equipped with the semi-norms
\[\big|(s_\alpha)_{\alpha\in\nset_0^n}\big|_d := \sup_{\alpha\in\nset_0^n: |\alpha|\leq d} |s_\alpha|\]
for all $d\in\nset_0$ is a Fr\'echet space.
To see this, let
\[\iota:\nset_0^n\to\nset_0\]
be a bijection.
Then a metric is given by
\[m(s,t) := \sum_{\alpha\in\nset_0^n} \frac{1}{2^{\iota(\alpha)}}\cdot\frac{|s_\alpha-t_\alpha|}{1 + |s_\alpha - t_\alpha|}\]
for any $s=(s_\alpha)_{\alpha\in\nset_0^n}$ and $t=(t_\alpha)_{\alpha\in\nset_0^n}\in\cset^{\nset_0^n}$.
To understand the topology of $\cset^{\nset_0^n}$ it is sufficient to understand the convergence
\[s^{(k)}\xrightarrow{k\to\infty} s\]
for sequences
\[s=(s_\alpha)_{\alpha\in\nset_0^n},\ s^{(k)}= (s_{k,\alpha})_{\alpha\in\nset_0^n}\in\cset^{\nset_0^n}.\]
From the semi-norms $|\,\cdot\,|_d$ we find
\begin{align*}
s^{(k)} \xrightarrow{k\to\infty} s\quad &\Leftrightarrow\quad \big| s - s^{(k)}\big|_d \xrightarrow{k\to\infty} 0\ \text{for all}\ d\in\nset_0\\
&\Leftrightarrow\quad s_{k,\alpha}\xrightarrow{k\to\infty} s_\alpha\ \text{for all}\ \alpha\in\nset_0^n,
\end{align*}
i.e., the topology of $\cset^{\nset_0^n}$ is generated by the coordinate-wise convergence.
\exmsymbol
\end{exm}

\begin{rem}\label{rem:frechet}
Let $n\in\nset$. Then
\[\cset^{\nset_0^n} \cong \cset[[x_1,\dots,x_n]]. \tag*{$\circ$}\]
%\exmsymbol
\end{rem}

\begin{exm}
Every Banach space is a Fr\'echet space.
\exmsymbol
\end{exm}

\begin{rem}
$\cset^{\nset_0^n}$ is a Fr\'echet space that is not a Banach space.
\exmsymbol
\end{rem}

\begin{lem}[{\cite[Cor.\ 3.7 (i)]{didio24posPresConst}}]\label{lem:momConeClosedFrechet}
Let $n\in\nset$ and let $K\subseteq\rset^n$ be closed.
Then the cone
\[\cS(\rset[x_1,\dots,x_n]:K\to\rset,\{x^\alpha\}_{\alpha\in\nset_0^n})\]
of $K$-moment sequences is closed in the Fr\'echet topology of $\cset^{\nset_0^n}$.
\end{lem}
\begin{proof}
Let $s^{(k)},s\in\cS$ with $s^{(k)}\xrightarrow{k\to\infty} s$ in $\cset^{\nset_0^n}$.
Then
\[\underbrace{L_{s^{(k)}}(p)}_{\geq 0} \quad\xrightarrow{k\to\infty}\quad L_s(p) \geq 0\]
for all $p\in\pos(K)$, since $\deg p < \infty$.
Hence, by \Cref{thm:haviland}, $s\in\cS$.
\end{proof}

\section{The LF-Space $\cset[x_1,\dots,x_n]$}%%%
%%%%%%%%%%%%%%%%%%%%%%%%%%%%%%%%%%%%%%%%%%%%%%%%

\begin{dfn}\label{dfn:LF}
Let $E$ be a complex vector space such that
\[ E = \bigcup_{n\in\nset} E_n\]
with Fr\'echet spaces $E_n$ for all $n\in\nset$ and the natural injections $E_n \hookrightarrow E_{n+1}$, i.e., $E_n\subseteq E_{n+1}$, is an isomorphism.
If $E$ is equipped with the topology such that a convex $V$ is a neighborhood of zero in $E$ if and only if $V\cap E_n$ is a neighborhood of zero in $E_n$ for all $n\in\nset$, then $E$ is called a \emph{LF-space}.\index{LF-space}\index{space!LF-}
The set $\{E_n\}_{n\in\nset}$ is called a \emph{defining sequence}\index{sequence!defining}\index{defining!sequence} of $E$.
\end{dfn}

The LF stands for \emph{limit of Fr\'echet spaces}.

\begin{exm}\label{exm:polyLF}
Let $n\in\nset$ and set
\[E_d := \cset[x_1,\dots,x_n]_{\leq d}\]
for $d\in\nset_0$.
Since $E_d$ are finite dimensional, the $E_d$ are Fr\'echet spaces.
Since
\[\cset[x_1,\dots,x_n] = \bigcup_{d\in\nset_0} E_d,\]
we can equip $\cset[x_1,\dots,x_n]$ with the topology by \Cref{dfn:LF}.
Then $\cset[x_1,\dots,x_n]$ is a LF-space.
\exmsymbol
\end{exm}

We will see a simpler version how to understand the topology of $\cset[x_1,\dots,x_n]$.

\begin{prop}[{\cite[Prop.\ 13.1]{treves67}}]
Let $E$ be an LF-space, let $\{E_n\}_{n\in\nset}$ be a defining sequence of $E$, let $F$ be an arbitrary locally convex topological vector space, and let
\[l:E\to F\]
be linear.
Then the following are equivalent:
\begin{enumerate}[(i)]
\item $l:E\to F$ is continuous.

\item For all $n\in\nset$,
\[l|_{E_n}:E_n\to F\]
is continuous.
\end{enumerate}
\end{prop}

\begin{exm}\label{exm:dualPoly}
Let $s\in\cset^{\nset_0^n}$.
Then the Riesz functional
\[L_s:\cset[x_1,\dots,x_n]\to\cset\]
is continuous, since
\[E_d = \cset[x_1,\dots,x_n]_{\leq d}\]
are finite dimensional and every linear functional on a finite dimensional space are continuous.

On the other side, every linear functional
\[l:\cset[x_1,\dots,x_n]\to\cset\]
is uniquely determined by
\[s_\alpha := l(x^\alpha)\]
for all $\alpha\in\nset_0^n$, i.e.,
\[L_{\cdot}: \cset^{\nset_0^n}\to \big(\cset[x_1,\dots,x_n]\big)^*,\quad s\mapsto L_s\]
is an isomorphism where $E^*$ is the dual of $E$, i.e., the set of all continuous linear functionals on a topological vector space $E$.
Hence, the Riesz functional $L_s$ is the dual pairing map
\[\langle\,\cdot\,,\,\cdot\,\rangle: E^*\times E,\quad (s,p)\mapsto \langle s,p\rangle:= L_s(p)\]
between
\[E^* = \cset^{\nset_0^n} \qquad\text{and}\qquad E = \cset[x_1,\dots,x_n]. \tag*{$\circ$}\]
%\exmsymbol
\end{exm}

\begin{thm}[{\cite[Thm.\ 13.1]{treves67}}]
Any LF-space is complete.
\end{thm}

In the following continuation of \Cref{exm:polyLF} we give a simpler interpretation of the topology of the LF-space $\cset[x_1,\dots,x_n]$.

\begin{exm}[\Cref{exm:polyLF} continued]
Let $n\in\nset$ and let
\[p = \sum_{\alpha\in\nset_0^n} p_\alpha\cdot x^\alpha,\quad p_k = \sum_{\alpha\in\nset_0^n} p_{k,\alpha}\cdot x^\alpha\quad\in\cset[x_1,\dots,x_n]\]
for all $k\in\nset$.
Then
\[p_k \xrightarrow{k\to\infty} p\quad \text{in}\ \cset[x_1,\dots,x_n]\]
if and only if
\begin{equation}\label{eq:polyConvergence}
\langle s, p_k\rangle \xrightarrow{k\to\infty} \langle s, p\rangle \quad\text{for all}\ s\in\cset^{\nset_0^n}
\end{equation}
by \Cref{exm:dualPoly}.
For
\[\delta_\alpha := (\delta_{\alpha,\beta})_{\beta\in\nset_0^n}\in\cset^{\nset_0^n}\]
with $\alpha\in\nset_0^n$, (\ref{eq:polyConvergence}) implies
\[p_{k,\alpha} = \langle \delta_\alpha, p_k\rangle \xrightarrow{k\to\infty} \langle \delta_\alpha,p\rangle = p_\alpha\]
for all $\alpha\in\nset_0^n$, i.e., we have in $\cset[x_1,\dots,x_n]$ coordinate-wise convergence.

We now prove that (\ref{eq:polyConvergence}) also implies
\[\sup_{k\in\nset} \deg p_k < \infty.\]
To prove this assume
\[\deg p_k \xrightarrow{k\to\infty}\infty,\]
i.e., w.l.o.g.\ $\deg p_k = k$.
For each $k\in\nset$, let $\alpha_k\in\nset_0^n$ be such that
\[p_{k,\alpha} = \langle p_k,\delta_{\alpha_k}\rangle \neq 0.\]
Hence, there exists a $s\in\cset^{\nset_0^n}$ such that
\[\langle s,p_k\rangle = k \to\infty.\]
But this contradicts (\ref{eq:polyConvergence}), i.e., we have that
\[p_k \xrightarrow{k\to\infty} p\quad \text{in}\ \cset[x_1,\dots,x_n]\]
if and only if
\[\sup_{k\in\nset}\deg p_k \leq D < \infty \quad\text{and}\quad p_k\to p\ \text{in}\ \cset[x_1,\dots,x_n]_{\leq D}.\]

In summary, we proved that to understand the LF-topology of $\cset[x_1,\dots,x_n]$ it is sufficient to understand the convergence in it.
And similar to $\cset^{\nset_0^n}$, in $\cset[x_1,\dots,x_n]$ we have the coordinate-wise convergence with the additional uniform degree boundedness of the $p_k$.
\exmsymbol
\end{exm}

\begin{lem}
Let $n\in\nset$ and let $K\subseteq\rset^n$ be closed.
Then
\[\pos(K)\]
is closed in the LF-topology of $\cset[x_1,\dots,x_n]$.
\end{lem}
\begin{proof}
Let $p_k,p\in\rset[x_1,\dots,x_n]$ with
\[p_k \quad\xrightarrow{k\to\infty}\quad p\quad \text{in}\ \cset[x_1,\dots,x_n].\]
Then
\[\sup_{k\in\nset} \deg p_k \leq D <\infty\]
for some $D\in\nset$,
\[p_k\to p \quad\text{in}\ \cset[x_1,\dots,x_n],\]
and hence
\[\underbrace{p_k(x)}_{\geq 0} \quad\xrightarrow{k\to\infty}\quad p(x)\geq 0\]
for all $x\in K$, i.e., $p\in\pos(K)$.
\end{proof}

\section*{Problems}%%%%%%%%%%%%%%%%%%%%%
\addcontentsline{toc}{section}{Problems}

\begin{prob}[{\cite[Lem.\ 2.8]{didio25KPosPresGen}} or \Cref{lem:polySharpPartialFrechetSpace}]\label{prob:sharpFrechet}
Show that
\begin{multline*}
\cset[x_1,\dots,x_n]^\sharp[[\partial_1,\dots,\partial_n]] :=\\ \left\{ \sum_{\alpha\in\nset_0^n} p_\alpha\cdot \partial^\alpha \,\middle|\, p_\alpha\in\cset[x_1,\dots,x_n]_{\leq |\alpha|},\alpha\in\nset_0^n\right\}.
\end{multline*}
is a Fr\'echet space with the coordinate-wise convergence.
\end{prob}

\begin{prob}[{\cite[Rem.\ 2.9]{didio25KPosPresGen}}]\label{prob:formalformalseries}
Show that
\begin{multline*}
\cset[[x_1,\dots,x_n]][[\partial_1,\dots,\partial_n]] :=\\ \left\{ \sum_{\alpha\in\nset_0^n} p_\alpha\cdot \partial^\alpha \,\middle|\, p_\alpha\in\cset[[x_1,\dots,x_n]],\alpha\in\nset_0^n \right\}
\end{multline*}
is a Fr\'echet space with the coordinate-wise convergence.
\end{prob}

\begin{prob}
Let
\begin{multline*}
\cT := \{T:\cset[x_1,\dots,x_n]\to\cset[x_1,\dots,x_n]\ \text{linear}\}\\
= \left\{ \sum_{\alpha\in\nset_0^n} q_\alpha\cdot\partial^\alpha \,\middle|\, q_\alpha\in\cset[x_1,\dots,x_n]\ \text{for all}\ \alpha\in\nset_0^n \right\}.
\end{multline*}
\begin{enumerate}[\bfseries\qquad a)]
\item Exists a topology such that $\cT$ is a Fr\'echet space?

\item Exists a topology such that $\cT$ is a LF-space?
\end{enumerate}
\end{prob}

\begin{prob}
Let $n\in\nset$ and let
\[S: \cset^{\nset_0^n}\to\cset^{\nset_0^n}\]
be linear.
\begin{enumerate}[\bfseries\qquad a)]
\item Show that, for every $\alpha\in\nset_0^n$, there exists a polynomial $p_\alpha\in\cset[x_1,\dots,x_n]$ such that the $\alpha$-th coordinate $(Ss)_\alpha$ of $Ss$ with $s\in\cset^{\nset_0^n}$ is of the form
\[(Ss)_\alpha = \langle p_\alpha,s\rangle = L_s(p_\alpha).\]

\item Take in $\cset^{\nset_0^n}$ the standard basis
\[\cB := \{\delta_\alpha\}_{\alpha\in\nset_0^n} \qquad\text{with}\qquad \delta_\alpha := (\delta_{\alpha,\beta})_{\beta\in\nset_0^n}.\]
Show that, with respect to the basis $\cB$, the operator $S$ is represented by an infinite dimensional matrix $\tilde{S}$ such that in every row of $\tilde{S}$ there are only finitely many non-zero entries.

\item How many non-zero entries can $\tilde{S}$ have in each column?
\end{enumerate}
\textit{Hint:} Since $|\nset_0^n| = |\nset_0|$, it is sufficient to show (a) and (b) only for $n=1$.
\end{prob}

\begin{prob}
Let $n\in\nset$, let $\cset^{\nset_0^n}$ be with the Fr\'echet topology, and let $\cset[x_1,\dots,x_n]$ be with the LF-topology.
The dual pairing is given by
\[\langle\,\cdot\,,\,\cdot\,\rangle: \cset^{\nset_0^n}\times\cset[x_1,\dots,x_n]\to\cset,\quad (s,p)\mapsto \langle s,p\rangle := L_s(p).\]
Let
$S:\cset^{\nset_0^n}\to\cset^{\nset_0^n}$
be linear.
\begin{enumerate}[\bfseries\qquad a)]
\item How is the dual $S^*:\cset[x_1,\dots,x_n]\to\cset[x_1,\dots,x_n]$ defined?

\item Show $S^{**} = S$.
\end{enumerate}
\end{prob}

%%%%%%%%%%%%%%%%%%%%%%%%%%%%%%%%%%%%
%%%%%%%%%%%%%%%%%%%%%%%%%%%%%%%%%%%%
\part{$K$-Positivity Preserver}%%%%%
\label{part:KposPres}%%%%%%%%%%%%%%%
%%%%%%%%%%%%%%%%%%%%%%%%%%%%%%%%%%%%
%%%%%%%%%%%%%%%%%%%%%%%%%%%%%%%%%%%%

\chapter{$K$-Positivity Preserver}%%%
\label{ch:kPosPres}%%%%%%%%%%%%%%%%%%
%%%%%%%%%%%%%%%%%%%%%%%%%%%%%%%%%%%%%

We collected enough about linear operators
\[T:\rset[x_1,\dots,x_n]\to\rset[x_1,\dots,x_n]\]
and about moments to characterize all $T$ with
\[T\pos(K)\subseteq\pos(K)\]
for closed $K\subseteq\rset^n$.
The results in this chapter appeared in \cite{guterman08,borcea11,didio24posPresConst} and especially \cite{didio25KPosPresGen}.

\section{Definition}%%%
%%%%%%%%%%%%%%%%%%%%%%%

For a closed set $K\subseteq\rset^n$, we remind the reader of
\[\pos(K) := \big\{f\in\rset[x_1,\dots,x_n] \,\big|\, f\geq 0\ \text{on}\ K\big\}.\]
For the special case $K=\emptyset$, see Problem \ref{prob:emptyK}.
Hence, we here only look at the non-trivial case $K\neq\emptyset$.

\begin{dfn}
Let $n\in\nset$, let $K\subseteq\rset^n$ be closed and non-empty, and let
\[T:\rset[x_1,\dots,x_n]\to\rset[x_1,\dots,x_n]\]
be linear.
If
\[T\pos(K)\subseteq\pos(K),\]
then we call $T$ a \emph{$K$-positivity preserver}.\index{preserver!$K$-positivity}\index{positivity!$K$-!preserver}
If $K=\rset^n$, then, for short, we call $T$ a \emph{positivity preserver}.\index{preserver!positivity}\index{positivity!preservers}
\end{dfn}

At first we see in the following example, that, for every $K\subseteq\rset^n$ with $K\neq\emptyset$, there are non-trivial $K$-positivity preserver.

\begin{exm}\label{exm:simpleKposPres}
Let $n\in\nset$, let $K\subseteq\rset^n$ be closed and non-empty, let $p\in\pos(K)$, and let
\[L:\rset[x_1,\dots,x_n]\to\rset\]
be a $K$-moment functional.
Then
\[T:\rset[x_1,\dots,x_n]\to\rset[x_1,\dots,x_n],\quad f\mapsto Tf := L(f)\cdot p\]
is a $K$-positivity preserver.
\exmsymbol
\end{exm}

\begin{dfn}\label{dfn:Ty}
Let $q_\alpha\in\rset[x_1,\dots,x_n]$ for all $\alpha\in\nset_0^n$ and let
\[T = \sum_{\alpha\in\nset_0^n} q_\alpha\cdot\partial^\alpha\]
be a linear operator.
For $y\in\rset^n$, we define the linear map
\[T_y:\rset[x_1,\dots,x_n]\to\rset[x_1,\dots,x_n]\]
by
\begin{equation}\label{eq:Ty}
T_y := \sum_{\alpha\in\nset_0^n} q_\alpha(y)\cdot\partial^\alpha.
\end{equation}
\end{dfn}

Clearly,
\[(Tf)(x) \;=\; \sum_{\alpha\in\nset_0^n} q_\alpha(x)\cdot(\partial^\alpha f)(x) \;=\; (T_x f)(x)\]
for all $f\in\rset[x_1,\dots,x_n]$ and $x\in\rset^n$.

\section{Characterization of $K$-Positivity Preserver}%%%
%%%%%%%%%%%%%%%%%%%%%%%%%%%%%%%%%%%%%%%%%%%%%%%%%%%%%%%%%

\begin{lem}[{\cite[Lem.\ 3.3]{didio25KPosPresGen}}]\label{lem:KposMomSeq}
Let $q_\alpha\in\rset[x_1,\dots,x_n]$ for all $\alpha\in\nset_0^n$ and let $K\subseteq\rset^n$ be closed and non-empty.
If
\[T = \sum_{\alpha\in\nset_0^n} q_\alpha\cdot\partial^\alpha\]
is a $K$-positivity preserver, then
\[(\alpha!\cdot q_\alpha(y))_{\alpha\in\nset_0^n}\]
is a $(K-y)$-moment sequence for all $y\in K$.
\end{lem}
\begin{proof}
Let $y\in K$ and define
\[L_y:\rset[x_1,\dots,x_n]\to\rset,\quad f\mapsto L_y(f) := (T_y f)(y).\]
Since $T$ is a $K$-positivity preserver,
\[L_y(f) = (T_y f)(y) \geq 0\]
for all $f\in\pos(K)$.
Hence, $L_y$ is a $K$-moment functional by \Cref{thm:haviland}, i.e., there exists a (Radon) measure $\mu_y$ with $\supp\mu_y\subseteq K$ such that
\[\sum_{\alpha\in\nset_0^n} q_\alpha(y)\cdot f^{(\alpha)}(y) = (T_y f)(y) = L_y(f) = \int_K f(x)~\diff\mu_y(x)\]
for all $f\in \rset[x_1,\dots,x_n]$.
Using the Taylor formula
\[f(x) = \sum_{\alpha\in\nset_0^n} \frac{f^{(\alpha)}(y)}{\alpha!}\cdot (x-y)^\alpha\]
we obtain
\begin{equation}\label{eq:proofTaylor}
\sum_{\alpha\in\nset_0^n} q_\alpha(y)\cdot f^{(\alpha)}(y) = \sum_{\alpha\in\nset_0^n} \frac{f^{(\alpha)}(y)}{\alpha!}\cdot \int (x-y)^\alpha~\diff\mu_y(x)
\end{equation}
for all $f\in\rset[x_1,\dots,x_n]$.
For $\beta\in\nset_0^n$, set
\[f_\beta(x) := (x-y)^\beta.\]
Then
\[f^{(\alpha)}_{\beta}(y) = \beta!\cdot \delta_{\alpha,\beta}.\]
Hence, by \Cref{eq:proofTaylor},
\[\beta!\cdot q_\beta(y) = \int_K (x-y)^\beta~\diff\mu_y(x) = \int_{K-y} z^\beta~\diff\mu_y(z+y),\]
i.e.,
\[(\alpha!\cdot q_\alpha(y))_{\alpha\in\nset_0^n}\]
is a $(K-y)$-moment sequence with representing measure $\nu_y(\,\cdot\,) := \mu_y(\,\cdot\,+y)$.
\end{proof}

\begin{lem}[{\cite[Lem.\ 3.4]{didio25KPosPresGen}}]\label{lem:K-yMomSeqMeasure}
Let
\[T = \sum_{\alpha\in\nset_0^n} q_\alpha\cdot \partial^\alpha\]
with $q_\alpha\in\rset[x_1,\dots,x_n]$ for $\alpha\in\nset_0^n$ and let $K\subseteq\rset^n$ be closed and non-empty.
If $y\in K$ is such that
\[(\alpha!\cdot q_\alpha(y))_{\alpha\in\nset_0^n}\]
is a moment sequence with representing measure $\mu_y$, then
\[(T_y f)(y) = \int f(x+y)~\diff\mu_y(x)\]
for all $f\in\rset[x_1,\dots,x_n]$.
\end{lem}
\begin{proof}
Since $\mu_y$ is a representing measure of $(\alpha!\cdot q_\alpha(y))_{\alpha\in\nset_0^n}$,
\[\alpha!\cdot q_\alpha(y) = \int x^\alpha~\diff\mu_y(x)\]
for all $\alpha\in\nset_0^n$.
From Taylor's formula
\[f(x+y) = \sum_{\alpha\in\nset_0^n} \frac{f^{(\alpha)}(y)}{\alpha!}\cdot x^\alpha\]
we obtain
\begin{multline*}
\int f(x+y)~\diff\mu_y(x) = \sum_{\alpha\in\nset_0^n} \frac{f^{(\alpha)}(y)}{\alpha!}\cdot \int x^\alpha~\diff\mu_y(x)\\
= \sum_{\alpha\in\nset_0^n} f^{(\alpha)}(y)\cdot q_\alpha(y) = (T_y f)(y).\qedhere
\end{multline*}
\end{proof}

By combining \Cref{lem:KposMomSeq} and \Cref{lem:K-yMomSeqMeasure}, we get the following.

\begin{thm}[{\cite[Main Thm.\ 3.5]{didio25KPosPresGen}}]\label{thm:kPosPresCara}
Let $n\in\nset$ and let $K\subseteq\rset^n$ be closed and non-empty.
Let
\[T:\rset[x_1,\dots,x_n]\to\rset[x_1,\dots,x_n] \quad\text{with}\quad T = \sum_{\alpha\in\nset_0^n} q_\alpha\cdot\partial^\alpha\]
and $q_\alpha\in\rset[x_1,\dots,x_n]$ for all $\alpha\in\nset_0^n$ be linear.
Then the following are equivalent:
\begin{enumerate}[(i)]
\item $T$ is a $K$-positivity preserver.

\item For all $y\in K$, the sequence $(\alpha!\cdot q_\alpha(y))_{\alpha\in\nset_0^n}$ is a $(K-y)$-moment sequence.
\end{enumerate}
If one of the equivalent conditions (i) or (ii) holds, then, for any $y\in K$,
\[(T_y f)(y) = \int f(x+y)~\diff\mu_y(x)\]
for all $f\in\rset[x_1,\dots,x_n]$, where $\mu_y$ is a representing measure of the $(K-y)$-moment sequence $(\alpha!\cdot q_\alpha(y))_{\alpha\in\nset_0^n}$.
\end{thm}
\begin{proof}
(i) $\Rightarrow$ (ii):
This is \Cref{lem:KposMomSeq}.

(ii) $\Rightarrow$ (i): 
Let $y\in K$ and suppose $(\alpha!\cdot q_\alpha(y))_{\alpha\in\nset_0^n}$ is a $(K-y)$-moment sequence having representing measure $\mu_y$ with $\supp\mu_y\subseteq K-y$.
Let $f\in\pos(K)$.
Then, by \Cref{lem:K-yMomSeqMeasure},
\[(Tf)(y)=(T_y f)(y)=\int_{K-y} f(x+y)~\diff\mu_y(x)=\int_K f(z)~\diff\mu_y(z-y)\geq 0,\]
since $f\geq 0$ on $K$ and $\supp\mu_y(\,\cdot\,-y)\subseteq K$.
Since $f\in\pos(K)$ and $y\in K$ were arbitrary, we have $T\pos(K)\subseteq\pos(K)$.
Hence, (i) is proved.
\end{proof}

The special case of diagonal $\rset^n$-positivity preserver can (and must) be proved separately.

\begin{thm}[{\cite[Cor.\ 4.3]{borcea11}}]
Let $n\in\nset$ and
\[T:\rset[x_1,\dots,x_n]\to\rset[x_1,\dots,x_n]\quad\text{with}\quad Tx^\alpha = t_\alpha x^\alpha\]
for $t=(t_\alpha)_{\alpha\in\nset_0^n}$ and $t_\alpha\in\rset$ for all $\alpha\in\nset_0^n$.
Then the following are equivalent:
\begin{enumerate}[(i)]
\item $T$ is a postivity preserver.

\item $t=(t_\alpha)_{\alpha\in\nset_0^n}$ is a moment sequence.
\end{enumerate}
If one of the statements (i) or (ii) holds and $\mu$ is a representing measure of $t$, then
\begin{equation}\label{eq:integralTdiagonal}
(Tf)(x) = \int f(x_1 y_1,\dots, x_n y_n)~\diff\mu(y_1,\dots,y_n)
\end{equation}
for all $f\in\rset[x_1,\dots,x_n]$ and $x\in\rset^n$.
\end{thm}
\begin{proof}
(i) $\Rightarrow$ (ii):
Define
\[L: \rset[x_1,\dots,x_n]\to\rset, \quad f\mapsto L(f) = (Tf)(\one)\]
Hence,
\[L(x^\alpha) = (Tx^\alpha)(\one) = (t_\alpha x^\alpha)(\one) = t_\alpha\]
for all $\alpha\in\nset_0^n$ and
\[L(f) = (Tf)(\one) \geq 0\]
for all $f\in\pos(\rset^n)$, i.e., by \Cref{thm:haviland}, $L$ is a moment functional and $t$ is a moment sequence.
Let $\mu$ be a representing measure of $t$.
Then
\[Tx^\alpha = t_\alpha x^\alpha = \int y^\alpha~\diff\mu(y) \cdot x^\alpha = \int (x_1 y_1,\dots,x_n y_n)^\alpha~\diff\mu(y)\]
for all $\alpha\in\nset_0^n$.
By linearity of $T$, (\ref{eq:integralTdiagonal}) is proved.

(ii) $\Rightarrow$ (i):
Let $t$ be a moment sequences with representing measure $\mu$ and (\ref{eq:integralTdiagonal}) holds.
Let $f\in\pos(\rset^n)$.
Then
\[(Tf)(x) = \int \underbrace{f(x_1 y_1,\dots, x_n y_n)}_{\geq 0}~\diff\mu(x) \geq 0\]
for all $x\in\rset^n$, i.e., $Tf\in\pos(\rset^n)$ and $T$ is a positivity preserver.
\end{proof}

\begin{cor}[{\cite[Thm.\ 5.2]{didio25hadamardLanger}}]\label{cor:odot}
Let $n\in\nset$ and let
\[s=(s_\alpha)_{\alpha\in\nset_0^n}\quad\text{and}\quad t=(t_\alpha)_{\alpha\in\nset_0^n}\]
be two moment sequences with representing measures $\nu$ and $\mu$.
Then the \emph{Hadamard product}\index{Hadamard!product}\index{product!Hadamard}
\[s\odot t := (s_\alpha\cdot t_\alpha)_{\alpha\in\nset_0^n}\]
of $s$ and $t$ is a moment sequence with representing measure $\nu\odot\mu$, where $\nu\odot\mu$ is defined by
\[\nu\odot\mu = (\nu\times\mu)\circ m^{-1},\]
i.e.,
\[(\nu\odot\mu)(A) = (\nu\times\mu)(m^{-1}(A))\]
for all Borel sets $A\in\fB(\rset^n)$ and
\[m:\rset^n\times\rset^n\to\rset^n,\quad (x,y)\mapsto m(x,y) := (x_1 y_1,\dots,x_n y_n).\]
\end{cor}
\begin{proof}
See Problem \ref{prob:odot}.
\end{proof}

For $K=\rset^n$, the characterization of diagonal positivity preserver is solved.
For $K\subsetneq\rset^n$ it is open.

\begin{open}
Let $n\in\nset$, $K\subseteq\rset^n$ be closed and non-empty, and let
\[T:\rset[x_1,\dots,x_n]\to\rset[x_1,\dots,x_n] \quad\text{with}\quad Tx^\alpha = t_\alpha x^\alpha\]
for all $\alpha\in\nset_0^n$.
What are necessary and sufficient conditions for
\[t = (t_\alpha)_{\alpha\in\nset_0^n}\in\rset^{\nset_0^n}\]
such that $T$ is a $K$-positivity preserver?
\end{open}

\section{Properties of $K$-Positivity Preservers}%%%
%%%%%%%%%%%%%%%%%%%%%%%%%%%%%%%%%%%%%%%%%%%%%%%%%%%%

We have already seen in \Cref{exm:simpleKposPres} non-trivial $K$-positivity preserver.
From known $K$-positivity preserver we gain new $K$-positivity preserver by the following method.

\begin{exm}[{\cite[Exm.\ 3.1]{didio25KPosPresGen}}]
Let $n\in\nset$, $p=(p_1,\dots,p_n)\in\rset[x_1,\dots,x_n]^n$, and $s = (s_\alpha)_{\alpha\in\nset_0^n}$ be a moment sequence.
Then
\[T = \sum_{\alpha\in\nset_0^n} \frac{p^{\alpha}\cdot s_\alpha}{\alpha!}\cdot\partial^\alpha\]
is a positivity preserver with non-constant coefficients.
\exmsymbol
\end{exm}
\begin{proof}
For any $y\in\rset^n$, we have that $p_y := (p^\alpha(y))_{\alpha\in\nset_0^n}$ is a moment sequence with representing measure $\delta_{p(y)}$. 
Hence, by \Cref{cor:odot}, we have for every $y\in\rset^n$ that
\[p_y\odot s = (p^\alpha(y)\cdot s_\alpha)_{\alpha\in\nset_0^n}\]
is a moment sequence.
By \Cref{thm:kPosPresCara}, we have that $T$ is a positivity preserver.
\end{proof}

From the previous example we see that the degree of
\[q_\alpha = \frac{p_\alpha\cdot s_\alpha}{\alpha!}\]
can be much larger than $|\alpha|$.
The following result shows that the degree of the polynomial coefficients $q_\alpha$ of a positivity preserver $T$ can even grow arbitrarily large.
It is sufficient to show this for the case $n=1$.

\begin{prop}[{\cite[Prop.\ 3.2]{didio25KPosPresGen}}]
Let $(r_i)_{i\in\nset_0}$ be a sequence in $\nset$.
Then there exist a sequence $(c_i)_{i\in\nset_0}$ in $(0,\infty)$ and a sequence $(k_i)_{i\in\nset_0}$ in $\nset$ such that
\[T = \sum_{i\in\nset_0} \frac{p_i}{i!}\cdot\partial_x^i\]
is a positivity preserver with
\[p_{2i}(x) = c_i + x^{2k_i r_i} \qquad\text{and}\qquad p_{2i+1}(x) = 0\]
for all $i\in\nset_0$.
\end{prop}
\begin{proof}
By \Cref{thm:kPosPresCara}, it is sufficient to show that $p(y) := (p_i(y))_{i\in\nset_0}$ is a $\rset$-moment sequence for all $y\in\rset$.
Since we are in the one-dimensional case, this is equivalent to $\cH_i(p(y))\succeq 0$ for all Hankel matrices $\cH_i(p(y)) := (p_{j+l}(y))_{j,l=0}^i$, $i\in\nset_0$, see e.g.\ \cite[Thm.\ 3.8 and Prop.\ 3.11]{schmudMomentBook} or \Cref{thm:hamburgerMP}.
But for this it is sufficient to ensure that
\begin{equation}\label{eq:hankeldet}
\det \cH_i(p(y))>0
\end{equation}
for all $i\in\nset_0$ and $y\in\rset$.
We construct the sequences $(c_i)_{i\in\nset_0}$ and $(k_i)_{i\in\nset_0}$ from (\ref{eq:hankeldet}) by induction:

\underline{$i=0$:}
We have
\[\det \cH_0(p(y)) = p_0(y) = c_0 + y^{2k_0 r_0} \geq 1\]
for $c_0=1$ and $k_0=1$.

\underline{$i\to i+1$:}
Assume we have $(c_j)_{j=0}^i$ and $(k_j)_{j=0}^i$ such that
\[\det \cH_j(p(y))\geq 1\]
for $j=0,\dots,i$ and $y\in\rset$.
Since $p_{2i+1}=0$ for $i\in\nset_0$, we get
\[\det\cH_{i+1}(p(y))=p_{2i+2}(y)\cdot\det\cH_i(p(y))+q_i(p_0(y),p_2(y),\dots,p_{2i}(y))\]
for some $q_i\in\rset[x_1,\dots,x_n]$ by expanding the determinant $\det\cH_{i+1}(p(y))$.
Hence, we can choose $c_{i+1}\gg 1$ and $k_{i+1}\gg 1$ such that
\[c_{i+1} \geq 1 + \max_{x\in [-2,2]} |q_i(p_0(x),p_2(x),\dots,p_{2i}(x))|\]
and
\[x^{2k_{i+1} r_{i+1}} \geq 1 + |q_i(p_0(x),p_2(x),\dots,p_{2i}(x)|\]
for all $x\in (-\infty,-2]\cup [2,\infty)$ since $\det\cH_i(p(y))\geq 1$.
\end{proof}

\begin{dfn}\label{dfn:polySharpPartial}
Let $n\in\nset$.
We define
\[\cset[x_1,\dots,x_n]^\sharp[[\partial_1,\dots,\partial_n]] := \left\{ \sum_{\alpha\in\nset_0^n} p_\alpha\cdot \partial^\alpha \,\middle|\, p_\alpha\in\cset[x_1,\dots,x_n]_{\leq |\alpha|},\alpha\in\nset_0^n\right\}.\]
\end{dfn}

The following was already stated as Problem \ref{prob:sharpFrechet}.

\begin{lem}[{\cite[Lem.\ 2.8]{didio25KPosPresGen}}]\label{lem:polySharpPartialFrechetSpace}
Let $n\in\nset$.
The vector space
\[\cset[x_1,\dots,x_n]^\sharp[[\partial_1,\dots,\partial_n]]\]
is a Fr\'echet space with the coordinate-wise convergence.
\end{lem}
\begin{proof}
We have
\begin{multline*}
\cset[x_1,\dots,x_n]^\sharp[[\partial_1,\dots,\partial_n]] =\\ \left\{\sum_{\alpha\in\nset_0^n} \sum_{\beta\in\nset_0^n:|\beta|\leq |\alpha|} c_{\alpha,\beta}\cdot x^\beta\cdot\partial^\alpha \,\middle|\, c_{\alpha,\beta}\in\cset\ \text{for}\ \alpha,\beta\in\nset_0^n\ \text{with}\ |\beta|\leq |\alpha|\right\},
\end{multline*}
i.e., $\cset[x_1,\dots,x_n]^\sharp[[\partial_1,\dots,\partial_n]]$ is as a vector space isomorphic to the vector space of sequences
\[\cS := \big\{(c_{\alpha,\beta})_{\alpha,\beta\in\nset_0^n: |\beta|\leq |\alpha|} \,\big|\, c_{\alpha,\beta}\in\cset\ \text{for}\ \alpha,\beta\in\nset_0^n\ \text{with}\ |\beta|\leq |\alpha|\big\}.\]
But
\[\cS \cong\cset[[x_1,\dots,x_n]] \cong \cset^{\nset_0^n} \cong \cset^{\nset_0}\]
is a Fr\'echet space.
Hence, so is $\cset[x_1,\dots,x_n]^\sharp[[\partial_1,\dots,\partial_n]]$ when endowed with the coordinate-wise convergence.
\end{proof}

\begin{rem}[{\cite[Rem.\ 2.9]{didio25KPosPresGen}} and Problem \ref{prob:formalformalseries}]
By the same argument as in the proof of \Cref{lem:polySharpPartialFrechetSpace} we have that
\[\cset[[x_1,\dots,x_n]][[\partial_1,\dots,\partial_n]] := \left\{ \sum_{\alpha\in\nset_0^n} p_\alpha\cdot \partial^\alpha \,\middle|\, p_\alpha\in\cset[[x_1,\dots,x_n]],\alpha\in\nset_0^n \right\}\]
is a Fr\'echet space with the coordinate-wise convergence.
\exmsymbol
\end{rem}

\begin{lem}[{\cite[Lem.\ 3.6 (ii)]{didio25KPosPresGen}}]\label{lem:closedPosPres}
Let $K\subseteq\rset^n$ be closed.
Then the set of $K$-positivity preservers $T$ such that
\[T\rset[x_1,\dots,x_n]_{\leq d}\subseteq\rset[x_1,\dots,x_n]_{\leq d} \qquad\text{for all}\ d\in\nset_0\]
is closed in the Fr\'echet topology of $\cset[x_1,\dots,x_n]^\sharp [[\partial_1,\dots,\partial_n]]$.
\end{lem}
\begin{proof}
Let
\[T_k = \sum_{\alpha\in\nset_0^n} q_{k,\alpha}\cdot\partial^\alpha\]
be $K$-positivity preservers for all $k\in\nset$ with
\begin{equation}\label{eq:proofTkConvergence}
T_k\to T_0\in\cset[x_1,\dots,x_n]^\sharp [[\partial_1,\dots,\partial_n]].
\end{equation}
Then, for all $y\in K$,
\[s_k(y) = (s_{k,\alpha})_{\alpha\in\nset_0^n} := (\alpha!\cdot q_{k,\alpha}(y))_{\alpha\in\nset_0^n}\]
is a $(K-y)$-moment sequence by \Cref{thm:kPosPresCara}.
By (\ref{eq:proofTkConvergence}),
\[s_{k,\alpha}(y)\to s_{0,\alpha}\]
for all $\alpha\in\nset_0^n$ and, by \Cref{lem:momConeClosedFrechet}, $s_0 := (s_{0,\alpha})_{\alpha\in\nset_0^n}$ is a $(K-y)$-moment sequence.
Since $y\in K$ was arbitrary, it follows from \Cref{thm:kPosPresCara} that $T_0$ is a $K$-positivity preserver.
\end{proof}

\section{Positivity Preserver with Constant Coefficients and the Convolution of Sequences}%%%
%%%%%%%%%%%%%%%%%%%%%%%%%%%%%%%%%%%%%%%%%%%%%%%%%%%%%%%%%%%%%%%%%%%%%%%%%%%%%%%%%%%%%%%%%%%%%

We have characterized general $K$-positivity preserver in \Cref{thm:kPosPresCara}, i.e., with polynomial coefficients $q_\alpha\in\rset[x_1,\dots,x_n]$.
To learn more, we look at the simple cases of constant coefficients $q_\alpha\in\rset$.

\begin{exm}\label{exm:translation}
Let $n=1$ and $t\in\rset$.
Then
\[e^{t\partial_x} = \sum_{k\in\nset_0} \frac{t^k}{k!}\cdot\partial_x^k\]
is positivity preserver by \Cref{thm:kPosPresCara}, since
\[(t^k)_{k\in\nset_0}\]
is a moment sequence with representing measure $\delta_t$.
\exmsymbol
\end{exm}

\begin{exm}\label{exm:heat}
Let $n=1$.
Then
\[\exp(t\cdot\partial_x^2) = \sum_{j\in\nset_0} \frac{t^j\cdot \partial_x^{2j}}{j!}\]
is a positivity preserver for all $t\geq 0$.
\exmsymbol
\end{exm}
\begin{proof}
See Problem \ref{prob:10t}.
\end{proof}

The following is a linear operator $T$ which is not a positivity preserver.

\begin{exm}[{\cite[Exm.\ 1.4]{didio24posPresConst}}]\label{exm:notpos}
Let $k\geq 3$ and $a\in\rset\setminus\{0\}$.
Then
\[\exp(a\partial_x^k) := \sum_{j\in\nset_0} \frac{a^j\cdot\partial_x^{j\cdot k}}{j!}\]
is not a positivity preserver, since
\[q_{2k} = \frac{a^2}{2} \neq 0 \qquad\text{but}\qquad q_{2k+2} = 0,\]
i.e., $(j!\cdot q_j)_{j\in\nset_0}$ is not a moment sequence.
\exmsymbol
\end{exm}

Since linear operators on $\rset[x_1,\dots,x_n]$ are characterized by the (polynomials) coefficients $q_\alpha$ (\Cref{thm:canonicalRepresentation}) we only need to look at
\[(q_\alpha)_{\alpha\in\nset_0^n}.\]
By \Cref{thm:kPosPresCara}, $K$-positivity preservers are characterized by
\[(\alpha!\cdot q_\alpha(y))_{\alpha\in\nset_0^n}\]
for fixed $y\in K$.
Hence, looking at constant coefficients is a natural step and we include in the following definition also the $\alpha!$ coefficients.

\begin{dfn}\label{dfn:Ds}
Let $s = (s_\alpha)_{\alpha\in\nset_0^n}$ be a real sequence.
We define
\[D:\rset^{\nset_0^n}\to\rset[[\partial_1,\dots,\partial_n]],\quad s\mapsto D(s) := \sum_{\alpha\in\nset_0^n} \frac{s_\alpha}{\alpha!}\cdot\partial^\alpha.\]
\end{dfn}

\begin{dfn}
Let $s = (s_\alpha)_{\alpha\in\nset_0^n}$ and $t = (t_\alpha)_{\alpha\in\nset_0^n}$ be two real sequences.
We define \textit{the convolution}
\[s*t = (u_\alpha)_{\alpha\in\nset_0^n}\]
of $s$ and $t$ by
\begin{equation}\label{eq:dfnConvolutionSequences}
u_\alpha := \sum_{\beta\preceq\alpha} \binom{\alpha}{\beta}\cdot s_\beta\cdot t_{\alpha-\beta}.
\end{equation}
\end{dfn}

The following lemma immediately follows from the previous two definitions.

\begin{lem}\label{lem:convolutionSt}
Let $s = (s_\alpha)_{\alpha\in\nset_0^n}$ and $t = (t_\alpha)_{\alpha\in\nset_0^n}$ be two real sequences.
Then
\[D(s) D(t) = D(t) D(s) = D(s*t).\]
\end{lem}
\begin{proof}
See Problem \ref{prob:convolutionST}.
\end{proof}

\begin{dfn}[see e.g.\ {\cite[Sect.\ 3.9]{bogachevMeasureTheory}}]\label{dfn:addConv}
Let $n\in\nset$, let $\mu$ and $\nu$ be measures on $\rset^n$, and let
\[a:\rset^n\times\rset^n\to\rset^n,\quad (x,y)\mapsto x+y.\]
We define the \emph{additive convolution} $\mu*\nu$ by
\[\mu * \nu := (\mu\times\nu)\circ a^{-1}.\]
We define
\[\mu^{*0} := \delta_0 \quad\text{and}\quad \mu^{*k} := \underbrace{\mu*\dots *\mu}_{k\text{-times}}\]
for all $k\in\nset$.
\end{dfn}

\begin{thm}[{\cite[Cor.\ 3.4]{didio24posPresConst}}]\label{thm:convolutionSt}
Let $n\in\nset$, let $s\in\rset^{\nset_0^n}$ be a moment sequence with representing measure $\mu$, and $t\in\rset^{\nset_0^n}$ be a moment sequence with representing measure $\nu$.
Then $s*t$ is represented by the measure $\mu*\nu$.
\end{thm}
\begin{proof}
See Problem \ref{prob:convolutionST2}.
\end{proof}

\section*{Problems}%%%%%%%%%%%%%%%%%%%%%
\addcontentsline{toc}{section}{Problems}

\begin{prob}\label{prob:odot}
Prove \Cref{cor:odot}.
\end{prob}

\begin{prob}\label{prob:odot}
Let $n\in\nset$ and let $\nu=\delta_x$ and $\mu=\delta_y$ with $x,y\in\rset^n$ in \Cref{cor:odot}.
What is the measure $\nu\odot\mu$?
\end{prob}

\begin{prob}\label{prob:translation}
In \Cref{exm:translation}, show that
\[(e^{t\partial_x}p)(x) = p(x+t)\]
for all $p\in\rset[x_1,\dots,x_n]$ and $t\in\rset$.
\end{prob}

\begin{prob}\label{prob:10t}
Show that \Cref{exm:heat} is a positivity preserver.
What is a representing measure of the constant coefficient sequence?\\
\textit{Hint:} For a fixed variance $\sigma
$, what moments does the Gaussian distributions have?
\end{prob}

\begin{prob}\label{prob:convolutionST}
Prove \Cref{lem:convolutionSt}.
\end{prob}

\begin{prob}\label{prob:convolutionST2}
Prove \Cref{thm:convolutionSt}.
\end{prob}

\begin{prob}\label{prob:emptyK}
What happens to the results in this \Cref{ch:kPosPres}, when $K=\emptyset$?
\end{prob}

%\advanced %%%%%%%%%%%%%%%%%%%%%%%%%%%%%%%%%%%%%%%%%%%%%%%%%
%\chapter{$K$-Positivity Preserver on Matrix Polynomials}%%%
%\label{ch:matrix}%%%%%%%%%%%%%%%%%%%%%%%%%%%%%%%%%%%%%%%%%%
%%%%%%%%%%%%%%%%%%%%%%%%%%%%%%%%%%%%%%%%%%%%%%%%%%%%%%%%%%%

%\cite{didio25matrix}

%%%%%%%%%%%%%%%%%%%%%%%%%%%%%%%%%%%%%%%%%%%%%%%%%%%%%%%%%%%%%%%
%%%%%%%%%%%%%%%%%%%%%%%%%%%%%%%%%%%%%%%%%%%%%%%%%%%%%%%%%%%%%%%
\part{Generators of $K$-Positivity Preserving Semi-Groups}%%%%%
\label{part:generators}%%%%%%%%%%%%%%%%%%%%%%%%%%%%%%%%%%%%%%%%
%%%%%%%%%%%%%%%%%%%%%%%%%%%%%%%%%%%%%%%%%%%%%%%%%%%%%%%%%%%%%%%
%%%%%%%%%%%%%%%%%%%%%%%%%%%%%%%%%%%%%%%%%%%%%%%%%%%%%%%%%%%%%%%

\chapter{Finite and Infinite Dimensional Lie Groups}%%%
\label{ch:lieGroups}%%%%%%%%%%%%%%%%%%%%%%%%%%%%%%%%%%%
%%%%%%%%%%%%%%%%%%%%%%%%%%%%%%%%%%%%%%%%%%%%%%%%%%%%%%%

Lie groups are an essential structure in mathematics.
They are widely studied and taught.
For literature see e.g.\ \cite{warner83} and \cite{hall04}.

Unfortunately, this only holds for the usual, i.e., finite dimensional, Lie groups.
In our study we need infinite dimensional versions, since we are working on the infinite dimensional space $\rset[x_1,\dots,x_n]$.
We already identified $\cset[x_1,\dots,x_n]$ as a LF-space and $\cset^{\nset_0^n}$ as a Fr\'echet space.
These are the topologies we need also for the infinite dimensional versions of Lie groups.
We therefore repeat the finite dimensional results which are required for our study and also introduce (regular) Fr\'echet Lie groups.
For more literature on the topic, see e.g.\ \cite{leslie67}, \cite{omori74}, \cite{omori97}, \cite{schmed23}.

\section{Finite Dimensional Lie Groups}

\begin{dfn}
Let $n\in\nset$.
A \emph{Lie group} $G$\index{Lie!group}\index{group!Lie} is a differentiable manifold which is also endowed with a group structure such that the map
\[G\times G\to G,\quad (\sigma,\tau)\mapsto \sigma\tau^{-1}\]
is $C^\infty$.
\end{dfn}

\begin{exms}
Let $n\in\nset$.
\begin{enumerate}[\bfseries (a)]
\item The Euclidean space $\rset^n$ with the vector addition is a Lie group.

\item $\cset^* := \cset\setminus\{0\}$ with the multiplication is a Lie group.

\item The product $G\times H$ of two Lie groups $G$ and $H$ is itself are Lie group with the product manifold structure and the direct product group structure:
\[(\sigma_1,\tau_1)\cdot (\sigma_2,\tau_2) = (\sigma_1\sigma_2,\tau_1\tau_2).\]

\item The manifold $\Gl(n,\rset)$\label{GlnR} of all non-singular $n\times n$-matrices with real entries are a Lie group under matrix multiplication.
\exmsymbol
\end{enumerate}
\end{exms}

\begin{dfn}
A Lie algebra $\fg$ over $\rset$ is a real vector space $\fg$ together with a bilinear operator
\[ [\,\cdot\,,\,\cdot\,]: \fg\times\fg\to\fg,\label{bracket}\]
called the \emph{bracket},\index{bracket} such that
\begin{enumerate}[(i)]
\item $[x,y] = -[y,x]$ and

\item $[[x,y],z] + [[y,z],x] + [[z,x],y] = 0$
\end{enumerate}
for all $x,y,z\in\fg$.
\end{dfn}

\begin{rem}
We see that Lie groups are finite dimensional.
However, our Lie algebra definition covers also the infinite dimensional cases.
\exmsymbol
\end{rem}

\begin{exms}\label{exms:lieAlgebras}
Let $n\in\nset$.
\begin{enumerate}[\bfseries\quad (a)]
\item Any (finite or infinite dimensional) vector space becomes a Lie algebra if all brackets are set to zero.
Such a Lie algebra is called \emph{abelian}.\index{Lie!algebra!abelian}\index{abelian!algebra!Lie}

\item The vector space $\gl(n,\rset)$ of all $n\times n$-matrices form a Lie algebra, if we set
\[[A,B] := AB - BA\]
for all $A,B\in\gl(n,\rset)$.\label{glnr}

\item A $2$-dimensional vector space with basis $x$, $y$ becomes a Lie algebra if we set
\[[x,x] = [y,y] = 0 \qquad\text{and}\qquad [x,y] = y\]
and extend it bilinearly with $[y,x]= -y$.

\item The vector space $\rset^3$ with the bilinear operation $x\times y$ of the vector cross product is a Lie algebra.
\exmsymbol
\end{enumerate}
\end{exms}

\begin{prop}[see e.g.\ {\cite[Prop.\ 3.7]{warner83}}]
Let $G$ be a Lie group.
Then the tangent space $T_e G$\label{tangentspace} of $G$ at the identity $e\in G$ is a finite dimensional Lie algebra with
\[\dim G = \dim T_e.\]
\end{prop}

\begin{dfn}
Let $G$ be a Lie group.
We define the \emph{Lie algebra $\fg$ of the Lie group $G$} to be $\fg := T_e G$ the tangent space of $G$ at the identity $e\in G$.
\end{dfn}

\begin{rem}
Alternatively, one can define the Lie algebra $\fg$ of the Lie group $G$ to be the (Lie) algebra of left invariant vector fields on $G$, see e.g.\ \cite{warner83}.
\exmsymbol
\end{rem}

\begin{rem}
An extremely strong theorem which was proved by Igor Dmitrievich Ado,\index{Ado!I.\ D.} the so called \emph{Ado's Theorem},\index{Theorem!Ado}\index{Ado!Theorem} states that every finite dimensional Lie algebra has a faithful $(1:1)$ representation in $\gl(n,\rset)$, i.e., for any Lie algebra $\fg$ there exists a sufficient large $n$ such that $\fg$ is isomorphic to a sub-Lie algebra of $\gl(n,\rset)$.
\exmsymbol
\end{rem}

\begin{thm}[see e.g.\ {\cite[Thm.\ 3.19]{warner83}}]\label{thm:lieGroupAlegbra1:1}
Let $G$ be a Lie group with Lie algebra $\fg$.
Then there is a one-to-one correspondence between connected Lie subgroups of $G$ and subalgebras of $\fg$.
\end{thm}

\section{Matrix Lie Groups and the Matrix Exponential Function}

While abstract Lie groups come in every color, the simpler examples are matrix Lie groups.
These are even sufficient for our study.

\begin{dfn}\label{dfn:matrixExp}
Let $n\in\nset$.
We define the \emph{matrix exponential function}
\[\exp:\rset^{n\times n}\to\rset^{n\times n},\quad A\mapsto \exp A := \sum_{k\in\nset_0} \frac{A^k}{k!}.\]
\end{dfn}

\begin{rem}
For short, we also write
\[e^A := \exp A. \tag*{$\circ$}\]
%\exmsymbol
\end{rem}

\begin{lem}
The matrix exponential is well-defined, i.e.,
\[\exp A\in\rset^{n\times n}\]
for all $A\in\rset^{n\times n}$.
\end{lem}
\begin{proof}
Let $\|\cdot\|$ be the spectral norm on $\rset^{n\times n}$:
\[\|A\| := \max_{x\in\rset^n\setminus\{0\}} \frac{\|Ax\|}{\|x\|}.\]
Clearly,
\[\|A+B\| \leq \|A\| + \|B\| \qquad\text{and}\qquad  \|AB\| \leq \|A\|\cdot \|B\|\]
for all $A,B\in\rset^{n\times n}$.
Hence,
\[\|\exp A\| \leq \sum_{k\in\nset_0} \frac{\|A\|^k}{k!} = e^{\|A\|} < \infty\]
for all $A\in\rset^{n\times n}$.
Hence, in the $\exp$ definition we have bounded dominance and therefore convergence in $\rset^{n\times n}$.
\end{proof}

\begin{prop}\label{prop:matrixExp}
Let $n\in\nset$.
Then the following hold:
\begin{enumerate}[(i)]
\item $\exp 0 = \id$.

\item $(\exp X)^T = \exp (X^T)$ for all $X\in\gl(n,\rset)$.

\item For all $X,Y\in\gl(n,\rset)$ with $[X,Y]=0$,
\[\exp(X+Y) = \exp X\cdot \exp Y.\]

\item For all $X\in\gl(n,\rset)$ and $a,b\in\rset$,
\[\exp ((a+b)X) = \exp(aX)\cdot \exp(bX).\]

\item For all $X\in\gl(n,\rset)$, $\exp X$ is invertiable and
\[(\exp X)^{-1} = \exp(-X).\]

\item For all $X\in\gl(n,\rset)$ and $C\in\Gl(n,\rset)$,
\[\exp (CXC^{-1}) = C\cdot \exp X\cdot C^{-1}.\]

\item Let
\[\tr X := X_{1,1} + X_{2,2} + \dots + X_{n,n}\]
be the trace\index{trace} $\tr X$ of $X=(X_{i,j})_{i,j=1}^n \in\gl(n,\rset)$.
Then
\[\det (\exp X) = e^{\tr X}\]
for all $X\in\gl(n,\rset)$.

\item For all $X\in\gl(n,\rset)$ and $Y\in\gl(m,\rset)$,
\[\exp (X\oplus Y) = \exp X \oplus \exp Y.\]

\item For all $X\in\gl(n,\rset)$,
\[\exp X = \lim_{k\to\infty} \left(\id + \frac{X}{k} \right)^k = \lim_{k\to\infty} \left(\id - \frac{X}{k} \right)^{-k}.\]
\end{enumerate}
\end{prop}
\begin{proof}
See Problem \ref{prob:matrixExp}.
\end{proof}

\begin{thm}
Let $n\in\nset$.
Then the following hold:
\begin{enumerate}[(i)]
\item For $X\in\gl(n,\rset)$, the function
\[\rset\to\Gl(n,\rset),\quad t\mapsto \exp(tX)\]
is $C^\infty$ in $t\in\rset$ with
\[\frac{\diff}{\diff t} \exp(tX)\Big|_{t=0} = X.\]

\item $\gl(n,\rset)$ is the Lie algebra of the Lie group $\Gl(n,\rset)$, i.e.,
\[\gl(n,\rset) = T_\id \Gl(n,\rset).\]
\end{enumerate}
\end{thm}
\begin{proof}
(i):
Since in the definition of $\exp X$ the sum is absolute convergent with, we can interchange differentiation and the infinite sum:
\[\frac{\diff}{\diff t} \exp (tX) = \frac{\diff}{\diff t} \sum_{k\in\nset_0} \frac{t^k\cdot X^k}{k!} = \sum_{k\in\nset} \frac{t^{k-1}\cdot X^k}{(k-1)!} = X\cdot \exp(tX).\]
This proves the statement, since $\exp 0 = \id$.

(ii):
Let $X\in\gl(n,\rset)$.
Then, by (i) and \Cref{prop:matrixExp},
\[g_X:\rset\to\Gl(n,\rset),\quad t\mapsto \exp(tX)\]
is a smooth curve in $\Gl(n,\rset)$ with
\[g(0) = \id \quad\text{and}\quad g'(0) = X,\]
i.e., $\gl(n,\rset)\subseteq T_\id\Gl(n,\rset)$.
But since
\[\dim \gl(n,\rset) = n^2 = \dim\Gl(n,\rset),\]
we have $\gl(n,\rset) = T_\id\Gl(n,\rset)$.
\end{proof}

\begin{exm}
Let $n\in\nset$ and let
\[\Sl(n,\rset) := \{X\in\Gl(n,\rset) \,|\, \det X = 1\}\label{SlnR}\]
be the \emph{special linear group}.\index{group!linear!special}
Clearly, $\Sl(n,\rset)$ is a proper subgroup of $\Gl(n,\rset)$ with
\[\dim\Sl(n,\rset) = n^2 - 1.\]
It is also connected.
We want to calculate the corresponding Lie subalgebra $\fsl(n,\rset)$ of the Lie subgroup $\Sl(n,\rset)$.
By \Cref{prop:matrixExp} (vii), we have $\det \exp X = e^{\tr X}$, i.e.,
\[X\in\fsl(n,\rset) \quad\Rightarrow\quad \tr X = 0.\]
But since
\[\dim \fsl(n,\rset) = n^2 - 1 = \dim \{X\in\gl(n,\rset) \,|\, \tr X = 0\},\]
we have
\[\fsl(n,\rset) = \{X\in\gl(n,\rset) \,|\, \tr X = 0\}.\label{slnR} \tag*{$\circ$}\]
\end{exm}

\begin{thm}\label{thm:expProperties}
Let $n\in\nset$, let
\[\|X\| := \sqrt{\sum_{i,j=1}^n x_{i,j}^2}\]
be the Hilbert--Schmidt norm of $X=(x_{i,j})_{i,j=1}^n\in\gl(n,\rset)$, and let
\[\log X := \sum_{k\in\nset} (-1)^{k+1}\cdot\frac{(X-\id)^k}{k}.\]
Then the following hold:
\begin{enumerate}[(i)]
\item For all $X\in\gl(n,\rset)$ with $\|X-\id\|< 1$, $\log X$ is well-defined and
\[\exp(\log X) = X.\]

\item If $X = \id + N$ for some nilpotent $N\in\gl(n,\rset)$, i.e., $N^n=0$, then
\[\log X = \sum_{k=1}^{n-1} (-1)^{k+1}\cdot\frac{N^k}{k} \quad\text{and}\quad \exp(\log X) = X.\]

\item If $X\in\gl(n,\rset)$ is such that $\|X\|< \log 2$ and $\|\id - \exp X\| < 1$, then
\[\log (\exp X) = X.\]
\end{enumerate}
\end{thm}
\begin{proof}
Follows by straight forward calculations from the absolute convergence of the $\log$-series within its radius of convergence similar to the case $n=1$ in analysis.
\end{proof}

\begin{cor}\label{cor:invertable}
Let $n\in\nset$ and $X\in\Gl(n,\rset)$.
Then there exists a $A\in\gl(n,\rset)$ such that $X = \exp A$, i.e., every invertable $X$ is of the form $\exp A$.
\end{cor}
\begin{proof}
See Problem \ref{prob:invertable}.
\end{proof}

\section{Regular Fr\'echet Lie Groups}

Already Hideki Omori\index{Omori!Hideki} stated the following, see \cite[pp.\ III-IV]{omori74}:
\begin{quote}\itshape
[G]eneral Fr\'echet manifolds are very difficult to treat.
For instance, there are some difficulties in the definition of tangent bundles, hence in the definition of the concept of $C^\infty$-mappings.
Of course, there is neither an implicit function theorem nor a Frobenius theorem in general.
Thus, it is difficult to give a theory of general Fr\'echet Lie groups.
\end{quote}
A more detailed study is given by Omori in \cite{omori97} and the theory successfully evolved since then, see e.g.\ \cite{omori74,kac85,omori97,stras02,wurz04,schmed23} and references therein.
% glockner17: does it exists?
We will give here only the basic definitions which will be needed for our study.

\begin{dfn}[{\cite[pp.\ 9, Def.\ 3.1]{omori97}}]
A pair $(G,\fg)$ consisting of a metric group $G$ and a complete locally convex topological /vector space $\fg$ is called a \emph{topological group of exponential type}\index{group!topological of exponential type} if there is a continuous mapping
\[\exp:\fg\to G\]
such that the following hold:
\begin{enumerate}[(i)]
\item For every $X\in\fg$, $\{\exp(sX)\}_{s\in\rset}$ is a one-parameter subgroup of $G$.

\item For $X,Y\in\fg$, $X=Y$ if and only if
\[\exp(sX) = \exp(sY)\]
for every $s\in\rset$.

\item For a sequence $(X_n)_{n\in\nset}$ in $\fg$,
\[\lim_{n\to\infty} X_n = X\in\fg\]
if and only if
\[\lim_{n\to\infty} \exp(sX_n) = \exp(sX)\]
uniformly in $s$ on every $[a,b]$ with $-\infty < a < b < \infty$.

\item There exists a continuous mapping $\Ad:G\times\fg\to\fg$ such that
\[h\cdot \exp(sX)\cdot h^{-1} = \exp(s\Ad(h)X)\]
for every, $s\in\rset$, $h\in G$, and $X\in\fg$.
\end{enumerate}
\end{dfn}

\begin{dfn}[{\cite[pp.\ 10]{omori97}}]\label{dfn:C1}
Let $(G,\fg)$ be a topological group of exponential type.
A continuous mapping
\[c:\rset\to G\]
is called \emph{differentiable}\index{differentiable} at $t_0\in\rset$ if and only if
\[\lim_{n\to\infty} \left( c(t_0+ s/n)\cdot c(t_0)^{-1} \right)^n\]
converges uniformly in $s$ on each compact interval $[a,b]$ with $-\infty < a < b< \infty$ to a one-parameter subgroup
\[\exp(sX(t_0)), \qquad \text{where}\ X(t_0)\in\fg.\]
$X(t_0)$ is called the \emph{derivative}\index{derivative} of $c$ at $t_0$ and is denoted by
\[X(t_0) = \dot{c}(t_0) = \left.\frac{\diff}{\diff t} c(t) \right|_{t=t_0}.\]
$c$ is called a $C^1$-curve, if $c$ is differentiable at every $t_0\in\rset$ and $\dot{c}$ is continuous with respect to $t\in\rset$.
By $C^1(G,\fg)$ we denote the space of all $C^1$-curves in $G$.
\end{dfn}

\begin{dfn}[see e.g.\ {\cite[p.\ 63, Dfn.\ 1.1]{omori97}}]\label{dfn:frechetLieGroup}
We call $(G,\,\cdot\,)$ a (\emph{regular}) \emph{Fr\'echet Lie group}\index{group!regular Fr\'echet Lie group} if the following conditions are fulfilled:
\begin{enumerate}[(i)]
\item $G$ is an infinite dimensional smooth Fr\'echet manifold.

\item $(G,\,\cdot\,)$ is a group.

\item The map $G\times G\to G$, $(A,B)\mapsto A\cdot B^{-1}$ is smooth.

\item The \emph{Fr\'echet Lie algebra}\index{algebra!Lie!Fr\'echet} $\fg$ of $G$ is isomorphic to the tangent space $T_e G$ of $G$ at the unit element $e\in G$.

\item $\exp:\fg\to G$ is a smooth mapping such that
\[\left.\frac{\diff}{\diff t} \exp(t u)\right|_{t=0} = u\]
holds for all $u\in\fg$.

\item The space $C^1(G,\fg)$ of $C^1$-curves in $G$ coincides with the set of all $C^1$-curves in $G$ under the Fr\'echet topology.
\end{enumerate}
\end{dfn}

We give here only the definition of a regular Fr\'echet Lie group.
We will get more familiar with this structure in the next two chapter, when we investigate liner operators
\[T:\rset[x_1,\dots,x_n]\to\rset[x_1,\dots,x_n]\]
with constant coefficients.
It will turn out that these are a (commutative) regular Fr\'echet Lie group.

For more on infinite dimensional manifolds, differential calculus, Lie groups, and Lie algebras see e.g.\ \cite{leslie67,omori97,schmed23}.

\section*{Problems}%%%%%%%%%%%%%%%%%%%%%
\addcontentsline{toc}{section}{Problems}

\begin{prob}\label{prob:lieAlgebras}
Check that the \Cref{exms:lieAlgebras} are indeed Lie algebras.
\end{prob}

%\begin{prob}\label{prob:maxtrixExp}
%Why is the matrix exponential map in \Cref{dfn:matrixExp} well-defined?

%\noindent
%\emph{Hint:} Use the \emph{Hilbert--Schmidt norm} $\|A\| := \sqrt{\sum_{i,j=1}^n a_{i,j}^2}$.
%\end{prob}

\begin{prob}\label{prob:matrixExp}
Prove \Cref{prop:matrixExp}.
\end{prob}

\begin{prob}
Use the Jordan decomposition to give a method to calculate $\exp X$ for a real $n\times n$-matrix $X$.
\end{prob}

\begin{prob}
Let $n\in\nset$ and let
\[\mathrm{O}(n,\rset) := \left\{X\in\Gl(n,\rset) \,\middle|\, X^T = X^{-1}\right\}\label{OnR}\]
be the \emph{orthogonal group}.\index{group!orthogonal}
\begin{enumerate}[\bfseries\qquad a)]
\item Is $\mathrm{O}(n,\rset)$ a Lie group?

\item Is $\mathrm{O}(n,\rset)$ connected?

\item Calculate $\fo(n,\rset) = T_\id \mathrm{O}(n,\rset)$.\label{onR}
\end{enumerate}
\end{prob}

\begin{prob}
Let $n\in\nset$ and let
\[\mathrm{SO}(n,\rset) := \left\{X\in\Gl(n,\rset) \,\middle|\, X^T = X^{-1}\ \text{and}\ \det X = 1\right\}\label{SOnR}\]
be the \emph{special orthogonal group}.\index{group!special orthogonal}
\begin{enumerate}[\bfseries\qquad a)]
\item Is $\mathrm{SO}(n,\rset)$ a Lie group?

\item Calculate the Lie algebra $\so(n,\rset)$\label{sonR} of $\mathrm{SO}(n,\rset)$.
\end{enumerate}
\end{prob}

\begin{prob}\label{prob:invertable}
Prove \Cref{cor:invertable}.
\end{prob}

\begin{prob}
Let $a\in\rset$ and set
\[A := \begin{pmatrix}
0 & -a\\ a & 0
\end{pmatrix}.\]
\begin{enumerate}[\bfseries\qquad a)]
\item Calculate $\exp A$.

\item Is $\exp:\gl(n,\rset)\to\Gl(n,\rset)$ bijective?
\end{enumerate}
\end{prob}

\chapter{The Regular Fr\'echet Lie Group $\fD_c$ of Linear Operators with Constant Coefficients}%%%
\label{ch:constCoeff}%%%%%%%%%%%%%%%%%%%%%%%%%%%%%%%%%%%
%%%%%%%%%%%%%%%%%%%%%%%%%%%%%%%%%%%%%%%%%%%%%%%%%%%%%%%%

We are now going to investigate generators of positivity preserving semi-groups.
We start with constant coefficients.
This was first published in \cite{didio24posPresConst}.

\section{The Definition of $\fD_c$ and $\fd_c$}

\begin{dfn}\label{dfn:fDcfdc}
Let $n\in\nset$.
We define
\[\fD_c := \left\{\; T=\sum_{\alpha\in\nset_0^n} q_\alpha\cdot\partial^\alpha\in\rset[[\partial_1,\dots,\partial_n]] \;\middle|\; q_0=1\;\right\} \]
and
\[\fd_c := \left\{\; A=\sum_{\alpha\in\nset_0^n} a_\alpha\cdot\partial^\alpha\in\rset[[\partial_1,\dots,\partial_n]] \;\middle|\; a_0=0\;\right\}.\]
\end{dfn}

\section{The Lie Group $\fD_{c,d}$}

\begin{dfn}\label{dfn:fDcd}
Let $n\in\nset$ and $d\in\nset_0$.
We define
\[\fD_{c,d} := \fD_c|_{\rset[x_1,\dots,x_n]_{\leq d}}
\quad\text{and}\quad
\fd_{c,d} := \fd_c|_{\rset[x_1,\dots,x_n]_{\leq d}}.\]
\end{dfn}

\begin{rem}[{\cite[p.\ 888]{didio24posPresConst}}]
Since
\[A\rset[x_1,\dots,x_n]_{\leq d}\subseteq\rset[x_1,\dots,x_n]_{\leq d}\]
for all $d\in\nset_0$ and $A\in\fD_c$, the set $\fD_{c,d}$ is well-defined.
From \Cref{dfn:fDcd} we see that $\fD_{c,d}$ consists only of operators of the form
\[\sum_{\alpha\in\nset_0^n: |\alpha|\leq d} c_\alpha\cdot\partial^\alpha\]
with $c_\alpha\in\rset$ and $c_0=1$, since on $\rset[x_1,\dots,x_n]_{\leq d}$ every operator $\partial^\beta$ with $|\beta|>d$ fulfills
\[\partial^\beta p = 0\]
for all $p\in\rset[x_1,\dots,x_n]_{\leq d}$, i.e.,
\[\partial^\beta = 0\]
on $\rset[x_1,\dots,x_n]_{\leq d}$.
\exmsymbol
\end{rem}

\begin{rem}[{\cite[Rem.\ 2.7]{didio24posPresConst}}]
We can also define $\fD_{c,d}$ by
\[\fD_c/\langle \partial^\alpha\, |\, |\alpha|=d+1\rangle.\]
Both definitions are almost identical.
However, \Cref{dfn:fDcd} has the following advantage.
In $\fD/\langle \partial^\alpha\, |\, |\alpha|=d+1\rangle$ we have the problem that we are working with equivalence classes and hence we can not calculate $A+B$ for $A\in\fD_d$ and $B\in\fD_{e}$ for $d\neq e$.
With \Cref{dfn:fDcd} we can calculate $A+B$ for $A\in\fD_d$ and $B\in\fD_e$ with $d\neq e$ since $A+B$ is defined on
\[\dom(A+B) = \dom A \cap \dom B = \rset[x_1,\dots,x_n]_{\leq\min\{d,e\}}\]
as usual for (unbounded) operators \cite{schmudUnbound}.
\Cref{dfn:fDcd} can then even be used to calculate $A+B$ for $A$ on $\rset[x_1,\dots,x_n]_{\leq d}$ and $B$ on $\rset[x_1,\dots,x_m]_{\leq e}$ for $n\neq m$ and $d\neq e$ on
\[\dom(A+B)=\rset[x_1,\dots,x_{\min\{n,m\}}]_{\leq\min\{d,e\}}.\tag*{$\circ$}\]
\end{rem}

\begin{exm}[{\cite[Exm.\ 2.8]{didio24posPresConst}}]\label{exm:n1d3}
Let $n=1$ and $d=3$.
Then
\[\fD_3 = \left\{ 1 + c_1\partial_x + c_2\partial_x^2 + c_3\partial_x^3
% \sum_{i=0}^3 c_i\cdot \partial_x^i 
\;\middle|\; c_1,c_2,c_3\in\rset,\ c_0>0 \right\}\quad \text{on}\quad \rset[x]_{\leq 3} .\]
Let
\[A = 1 + a_1\partial_x + a_2\partial_x^2 + a_3\partial_x^3 \quad\text{and}\quad B = 1 + b_1\partial_x + b_2\partial_x^2 + b_3\partial_x^3\]
be in $\fD_3$.
Then $AB = BA$ and
\begin{equation}\label{eq:d3mult}
\begin{split}
AB &= (1 + a_1\partial_x + a_2\partial_x^2 + a_3\partial_x^3)\cdot (1 + b_1\partial_x + b_2\partial_x^2 + b_3\partial_x^3)\\
&= 1 + (a_1 + b_1)\partial_x + (a_2 + a_1 b_1 + b_2)\partial_x^2 + (a_3 + a_2 b_1 + a_1 b_2 + b_3)\partial_x^3,
\end{split}
\end{equation}
since derivatives $\partial^i$ with $i\geq 4$ are the zero operators on $\rset[x]_{\leq 3}$.
Hence, $(\fD_3,\,\cdot\,)$ is a commutative semi-group with neutral element $\one = 1$.

We will now see that $\fD_3$ is even a commutative group.
For that it is sufficient to find for any $A\in\fD_3$ a $B\in\fD_3$ with $AB= \one$.
By (\ref{eq:d3mult}), $AB = \one$ is equivalent to
\begin{align*}
0 &= a_1 + b_1 &&\Rightarrow\quad b_1 = -a_1\\
0 &= a_2 + a_1 b_1 + b_2 &&\Rightarrow\quad b_2 = -a_2 + a_1^2\\
0 &= a_3 + a_2 b_1 + a_1 b_2 + b_3 &&\Rightarrow\quad b_3 = -a_3 + 2a_2 a_1 - a_1^3,
\end{align*}
i.e., every $A\in\fD_3$ has the unique inverse
\[A^{-1} = 1 - a_1\partial_x + (-a_2 + a_1^2)\partial_x^2 + (-a_3 + 2a_2 a_1 - a_1^3)\partial_x^3 \quad\in\fD_3.\]
Hence, $(\fD_3,\,\cdot\,)$ is a commutative group.
\exmsymbol
\end{exm}

We have seen in the previous example that $(\fD_d,\,\cdot\,)$ for $n=1$ and $d=3$ is a commutative group.
This holds for all $n\in\nset$ and $d\in\nset_0$.

\begin{lem}[{\cite[Lem.\ 2.9]{didio24posPresConst}}]\label{lem:ddgroup}
Let $n\in\nset$ and $d\in\nset_0$.
Then $(\fD_{c,d},\,\cdot\,)$ is a commutative group.
\end{lem}
\begin{proof}
Let
\[A = \sum_{\alpha: |\alpha|\leq d} a_\alpha\partial^\alpha\quad \text{and}\quad B = \sum_{\beta: |\beta|\leq d} b_\beta\partial^\beta \quad\in\fD_{c,d},\]
i.e., $a_0 = b_0 =1$.
Then
\[AB = C = \sum_{\gamma\in\nset_0^n: |\gamma|\leq d} c_\gamma\cdot\partial^\gamma\]
with
\begin{equation}\label{eq:inverseSystem}
c_\gamma = \sum_{\alpha,\beta\in\nset_0^n: \alpha + \beta = \gamma} a_\alpha b_\beta.
\end{equation}
Let $\alpha = (\alpha_1,\dots,\alpha_n)\succeq\beta = (\beta_1,\dots,\beta_n)$ on $\nset_0^n$ if and only if $\alpha_i \geq \beta_i$ for all $i=1,\dots,n$.
Then (\ref{eq:inverseSystem}) can be solved by induction on $|\gamma|$.
For $|\gamma|=0$, we have
\[c_0 = a_0\cdot b_0 \quad\text{and}\quad a_0 = c_0 = 1,\]
i.e.,
\[b_0 = 1.\]
So assume (\ref{eq:inverseSystem}) is solved for all $c_\gamma$ with $|\gamma|\leq k-1$ for some $k=0,1,\dots,d-1$.
Then, for any $\gamma\in\nset_0$ with $|\gamma|=k$, we have
\begin{equation}\label{eq:inverseA}
b_\gamma = a_0 b_\gamma = -\sum_{\alpha\in\nset_0^{n}\setminus\{0\}: \gamma\succeq\alpha} a_\alpha\cdot b_{\gamma-\alpha},
\end{equation}
i.e., the system (\ref{eq:inverseSystem}) of equations has a unique solution gained by induction.
Hence, for every $A\in\fD_{c,d}$ there exists a unique $B\in\fD_{c,d}$ with $AB = BA = \one$.
\end{proof}

From \Cref{lem:ddgroup} we have seen that $(\fD_{c,d},\,\cdot\,)$ for any $n\in\nset$ and $d\in\nset_0$ is a commutative group.
Let
\begin{equation}\label{eq:iota}
\iota_d: \{1\}\times\rset^{\binom{n+d}{n}-1}\to\fD_d,\quad (a_\alpha)_{\alpha\in\nset_0^n: |\alpha|\leq d} \mapsto \sum_{\alpha\in\nset_0^n: |\alpha|\leq d} a_\alpha\cdot\partial^\alpha
\end{equation}
be an affine linear map.
Then $\iota_d$ in (\ref{eq:iota}) is a diffeomorphism and it is a coordinate map for $\fD_{c,d}$.
The smooth manifold
\[\{1\}\times\rset^{\binom{n+d}{n}-1}\]
inherits the group structure of $\fD_{c,d}$ through $\iota_d$, i.e.,
\begin{equation}
(\fD_d,\,\cdot\,)\; \overset{\iota_d}{\cong}\; \left(\{1\}\times\rset^{\binom{n+d}{n}-1},\odot\right)
\end{equation}
Hence, the map $\iota_d$ shows the following.

\begin{thm}[{\cite[Thm.\ 2.10]{didio24posPresConst}}]
Let $n\in\nset$ and $d\in\nset_0$.
Then $(\fD_{c,d},\,\cdot\,)$ is a Lie group.
\end{thm}
\begin{proof}
The map $\iota_d$ in (\ref{eq:iota}) is a diffeomorphism between $\fD_{c,d}$ and $\{1\}\times\rset^{\binom{n+d}{n}-1}$.
Hence, $\fD_{c,d}$ is a differentiable manifold which possesses the group structure $(\fD_{c,d},\,\cdot\,)$.
By (\ref{eq:inverseSystem}) and (\ref{eq:inverseA}), the map
\[\fD_{c,d}\times \fD_{c,d}\to\fD_{c,d},\quad (A,B)\mapsto AB^{-1}\]
is smooth.
Hence, $(\fD_{c,d},\,\cdot\,)$ is a commutative Lie group.
\end{proof}

\section{The Lie Algebra $\fd_{c,d}$ of $\fD_{c,d}$}

Since every $A\in\fD_{c,d}$ is a linear map
\[A:\rset[x_1,\dots,x_n]_{\leq d}\to\rset[x_1,\dots,x_n]_{\leq d}\]
between finite-dimensional vector spaces $\rset[x_1,\dots,x_n]_{\leq d}$, we can choose a basis of $\rset[x_1,\dots,x_n]_{\leq d}$ and get a matrix representation $\tilde{A}$ of $A$.
Take the monomial basis of $\rset[x_1,\dots,x_n]_{\leq d}$.
Then $\tilde{A}$ is an upper triangular matrix with diagonal entries $1$.

\begin{exm}[\Cref{exm:n1d3} continued, {\cite[Exm.\ 2.11]{didio24posPresConst}}]\label{exm:n1d3cont}
Let $n=1$ and $d=3$.
Then every
\[A = 1 + a_1\partial_x + a_2\partial_x^2 + a_3\partial_x^3 \quad\in\fD_{c,3}\]
has with the monomial basis $\{1,x,x^2,x^3\}$ of $\rset[x]_{\leq 3}$ the matrix representation
\[\tilde{A} = \begin{pmatrix}
1 & a_1 & 2a_2 & 6a_3\\
0 & 1 & 2a_1 & 6a_2\\
0 & 0 & 1 & 3a_1\\
0 & 0 & 0 & 1
\end{pmatrix}\]
and we therefore set
\[\tilde{\fD}_{c,3} := \left\{ \begin{pmatrix}
1 & a_1 & 2a_2 & 6a_3\\
0 & 1 & 2a_1 & 6a_2\\
0 & 0 & 1 & 3a_1\\
0 & 0 & 0 & 1
\end{pmatrix} \,\middle|\, a_1,a_2,a_3\in\rset \right\}.\]
Hence,
\[(\tilde{A}-\id)^4 = 0\]
as a matrix and also
\[(A-\one)^4 = 0\]
as an operator on $\rset[x]_{\leq 3}$.
From \Cref{thm:expProperties} we find that the matrix valued exponential map
\[\exp:\mathrm{gl}(4,\cset)\to\mathrm{Gl}(4,\cset),\quad \tilde{A}\mapsto \exp(\tilde{A}) := \sum_{k\in\nset_0} \frac{\tilde{A}^k}{k!}\]
is surjective and the logarithm
\[\tilde{A}\mapsto \log \tilde{A} := -\sum_{k\in\nset} \frac{(\id-\tilde{A})^k}{k}\]
is well-defined for all $\id + N\in\mathrm{Gl}(4,\cset)$ with $N$ nilpotent, see \Cref{thm:expProperties} (ii).
Since $\tilde{\fD}_{c,3}\subseteq \mathrm{Gl}(4,\cset)$ with $(\id-\tilde{A})^4 = 0$ for all $\tilde{A}\in\tilde{\fD}_{c,3}$, we have
\begin{equation}\label{eq:logA}
\log:\tilde{\fD}_3\to\mathrm{gl}(4,\cset),\quad \tilde{A}\mapsto\log \tilde{A} = -\sum_{k=1}^3 \frac{(\id-\tilde{A})^k}{k}.
\end{equation}
Since also $(A-\one)^4 = 0$ for all $A\in\fD_{c,3}$, we can use (\ref{eq:logA}) also for the differential operators in $\fD_{c,3}$:
\[\log:\fD_{c,3}\to\left\{d_0 + d_1\partial_x + d_2\partial_x^2 + d_3\partial_x^3 \,\middle|\, d_0,\dots,d_3\in\rset\right\},\quad A\mapsto -\sum_{k=1}^3 \frac{(\one-A)^k}{k}.\]
To determine the image $\log\fD_{c,3}$ recall that also $\log$ is an injective map by \Cref{thm:expProperties} (ii) and hence $\log\fD_{c,3}$ is $3$-dimensional with $d_0 = 0$, i.e., we have
\[\log\fD_{c,3} = \left\{d_1\partial_x + d_2\partial_x^2 + d_3\partial_x^3 \,\middle|\, d_1,d_2,d_3\in\rset\right\} =: \fd_{c,3}.\]
In summary, since $A^4 = 0$ for all $A\in\fd_{c,3}$, we have that
\begin{equation}\label{eq:df4exp}
\exp:\fd_{c,3}\to\fD_{c,3},\quad A\mapsto \sum_{k=0}^3 \frac{A^k}{k!}
\end{equation}
is surjective with inverse
\begin{equation}\label{eq:df4log}
\log:\fD_{c,3}\to\fd_{c,3},\quad A\mapsto -\sum_{k=1}^3 \frac{(\one-A)^k}{k}.
\end{equation}
Therefore, $\fd_{c,3}$ is the Lie algebra of $\fD_{c,3}$ and $\exp$ in (\ref{eq:df4exp}) is the exponential map between the Lie algebra $\fd_{c,3}$ and its Lie group $\fD_{c,3}$ with inverse $\log$ in (\ref{eq:df4log}).
\exmsymbol
\end{exm}

The previous example of the Lie algebra $\fd_{c,3}$ of the Lie group $\fD_{c,3}$ holds for all $n\in\nset$ and $d\in\nset_0$.
It is clear that $(\fd_{c,d},\,\cdot\,,+)$ is an algebra on $\rset[x_1,\dots,x_n]_{\leq d}$.
The algebra $\fd_{c,d}$ is a Lie algebra with the brackets
\[ [\,\cdot\,,\,\cdot\,]: \fd_{c,d}\times\fd_{c,d},\quad (A,B) \mapsto [A,B] := AB- BA,\]
which is identical zero since $\fd_{c,d}$ is commutative.
we have the following.

\begin{thm}[{\cite[Thm.\ 2.13]{didio24posPresConst}}]\label{thm:liealgebrad}
Let $n\in\nset$ and $d\in\nset_0$.
Then $(\fd_{c,d},\,\cdot\,,+)$ is the Lie algebra of the Lie group $(\fD_{c,d},\,\cdot\,)$ with exponential map
\[\exp: \fd_{c,d}\to\fD_{c,d},\quad A\mapsto \sum_{k=0}^d \frac{A^k}{k!}\]
and its inverse
\[\log:\fD_{c,d}\to\fd_{c,d},\quad A\mapsto -\sum_{k=1}^d \frac{(\one-A)^k}{k}.\]
\end{thm}
\begin{proof}
Follows from \Cref{thm:expProperties} similar to \Cref{exm:n1d3cont}.
\end{proof}

\section{The Regular Fr\'echet Lie Group $\fD_c$ and its regular Fr\'echet Lie Algebra $\fd_c$}

We applied (finite dimensional) Lie group arguments for the case of bounded degree.
Without the degree bounds we get the following.

\begin{thm}[{\cite[Thm.\ 2.15]{didio24posPresConst}}]\label{thm:liealgebra}
Let $n\in\nset$.
Then the following hold:
\begin{enumerate}[(i)]
\item $(\fd_c,\,\cdot\,,+)$ is a commutative algebra.

\item $(\fD_c,\,\cdot\,)$ is a commutative group.

\item The map
\[\exp:\fd_c\to\fD_c,\quad A\mapsto \sum_{k\in\nset_0} \frac{A^k}{k!}\]
is bijective.

\item The map
\[\log:\fD_c\to\fd_c,\quad A\mapsto -\sum_{k\in\nset} \frac{(\one-A)^k}{k}\]
is bijective.

\item The maps
\[\exp:\fd_c\to\fD_c \qquad\text{and}\qquad \log:\fD_c\to\fd_c\]
are inverse to each other.
\end{enumerate}
\end{thm}
\begin{proof}
(i): That is clear.

(ii): That
\[A\cdot B = B\cdot A\]
for all $A,B\in\fD_c$ is clear.
The inverse of $A\in\fD_c$ is uniquely determined by solving (\ref{eq:inverseSystem}) to get (\ref{eq:inverseA}) for all $\gamma\in\nset_0^n$.
This is a formal power series argument with coordinate-wise convergence (i.e., in the Fr\'echet topology, see \Cref{exm:frechet}).

(iii): At first we show that
\[\exp:\fd_c\to\fD_c\]
is well-defined.
To see this, note that for any $A\in\fd_c$ we have
\[A^k = \sum_{\alpha\in\nset_0^n: |\alpha|\geq k} c_\alpha\cdot\partial^\alpha,\]
i.e., $A^k$ contains no differential operators of order $\leq k-1$.
Hence, the sum
\[\sum_{k=0}^K \frac{A^k}{k!} = \sum_{\alpha\in\nset_0^n} c_{K,\alpha}\cdot\partial^\alpha \]
converges coefficient-wise to
\[\exp A =\sum_{k\in\nset_0} \frac{A^k}{k!} = \sum_{\alpha\in\nset_0^n} c_\alpha\cdot\partial^\alpha,\]
i.e., in the Fr\'echet topology of
\[\fD_c\subsetneq\rset[[\partial_1,\dots,\partial_n]] \cong \rset[[x_1,\dots,x_n]],\]
see \Cref{exm:frechet}.
In other words, the coefficients $c_\alpha$ depend only on $A^k$ for $k=0,\dots,|\alpha|$.
We therefore have
\[c_{K,\alpha} = c_\alpha\]
for all $K>|\alpha|$, and hence $\exp A\in\fD_c$ is well-defined.
With that we have
\[\exp\fd_c\subseteq\fD_c.\]
For equality we give the inverse map in (v).

(iv): To show that
\[\log:\fD_c\to\fd_c\]
is well-defined the same argument as in (iii) holds for $(\one-A)^k$ with $A\in\fD_c$.
It shows that
\[\log A\in\fd_c\]
for all $A\in\fD_c$ is well-defined and we have
\[\log\fD_c\subseteq\fd_c.\]

(v): To prove that $\exp$ and $\log$ are inverse to each other we remark
\[\rset[x_1,\dots,x_n] = \bigcup_{d\in\nset_0} \rset[x_1,\dots,x_n]_{\leq d}.\]
For $d\in\nset_0$, define
\[\exp_d A := \sum_{k=0}^d \frac{A^k}{k!} \quad\text{and}\quad \log_d A := -\sum_{k=1}^d \frac{(\one-A)^k}{k}.\]
Then, for every $p\in\rset[x_1,\dots,x_n]$ with $d = \deg p$,
\[\exp(\log A)p = \exp_d(\log_d A)p = Ap\]
for all $A\in\fd_c$ by \Cref{thm:liealgebrad}, i.e.,
\[\exp(\log A) = A\]
for all $A\in\fD_c$.
Similarly,
\[\log(\exp A)p = Ap\]
for all $A\in\fd_c$.
This also shows the remaining assertions
\[\exp(\fd_c) = \fD_c \qquad\text{and}\qquad \log \fD_c = \fd_c\]
from (iii) and (iv).
\end{proof}

\begin{cor}[{\cite[Cor.\ 2.16]{didio24posPresConst}}]\label{cor:lieprops}
Let $n\in\nset$ and let $\fd_c,\fD_c\subseteq\rset[[\partial_1,\dots,\partial_n]]$ be Fr\'echet spaces (equipped with the coordinate-wise convergence).
Then the following hold:
\begin{enumerate}[(i)]
\item $\fD_c\times\fD_c\to\fD_c$, $(A,B)\mapsto AB^{-1}$ is smooth.

\item $\exp:\fd_c\to\fD_c$ is smooth and
\[\left.\frac{\diff}{\diff t}\right|_{t=0} \exp(t u) = u\]
for all $u\in\fd_c$.

\item $\log:\fd_c\to\fD_c$ is smooth.
\end{enumerate}
\end{cor}
\begin{proof}
(i): Let
\[A = \sum_{\alpha\in\nset_0^n} a_\alpha\partial^\alpha \qquad\text{and}\qquad B = \sum_{\alpha\in\nset_0^n} b_\alpha\partial^\alpha\]
with $a_0 = b_0 = 1$.
From (\ref{eq:inverseSystem}) we see that the multiplication is smooth since every coordinate $c_\gamma$ of the product
\[AB = \sum_{\alpha\in\nset_0^n} c_\alpha\partial^\alpha\]
is a polynomial in $a_\alpha$ and $b_\alpha$ with $|\alpha|\leq |\gamma|$.
The inverse
\[B^{-1} = \sum_{\alpha\in\nset_0^n} d_\alpha\partial^\alpha\]
is smooth because of (\ref{eq:inverseA}), i.e., also the coefficients $d_\gamma$ of the inverse depend polynomially on the coefficients $b_\alpha$ with $|\alpha|\leq |\gamma|$.

(ii): In the proof of \Cref{thm:liealgebra} (iii) we have already seen that the coefficients $c_\gamma$ of
\[\exp A = \sum_{\alpha\in\nset_0^n} c_\alpha\partial^\alpha\]
depend polynomially on the coefficients $a_\alpha$ of
\[A = \sum_{\alpha\in\nset_0^n\setminus\{0\}} a_\alpha\partial^\alpha\]
with $|\alpha|\leq |\gamma|$. 

The condition
\[\left.\frac{\diff}{\diff t}\right|_{t=0} \exp(t u) = u\]
then follows by direct calculations.

(iii): Follows like (ii) from \Cref{thm:liealgebra} (iv).
\end{proof}

It is easy to see that $\fd_c$ and $\fD_c$ are both infinite dimensional smooth (Fr\'echet) manifolds.
Hence, summing everything up we have the following.

\begin{thm}[{\cite[Thm.\ 2.17]{didio24posPresConst}}]\label{thm:mainFrechetLieGroups}
Let $n\in\nset$.
Then $(\fD_c,\,\cdot\,)$ as a Fr\'echet space is a commutative regular Fr\'echet Lie group with the commutative Fr\'echet Lie algebra $(\fd_c,\,\cdot\,,+)$.
The exponential map
\[\exp:\fd_c\to\fD_c,\quad A\mapsto \sum_{k\in\nset_0} \frac{A^k}{k!}\]
is smooth and bijective with the smooth and bijective inverse
\[\log:\fD_c\to\fd_c,\quad A\mapsto -\sum_{k\in\nset} \frac{(\one-A)^k}{k}.\]
\end{thm}
\begin{proof}
We have that $\fD_c$ is an infinite dimensional smooth manifold, $\fD_c$ is a Fr\'echet space (with the coefficient-wise convergence topology, see \Cref{exm:frechet}) and, by \Cref{thm:liealgebra} (ii), we also have that $(\fD_c,\,\cdot\,)$ is a commutative group.
By \Cref{cor:lieprops} (i), we have that
\[(A,B)\mapsto AB^{-1}\]
is continuous in the Fr\'echet topology.
Hence, $(\fD_c,\,\cdot\,)$ is an infinite dimensional commutative Fr\'echet Lie group.

The properties about $\exp$ and $\log$ are \Cref{thm:liealgebra} (iii) -- (v).

We now prove the regularity condition (vi) in \Cref{dfn:frechetLieGroup}.
Let
$F:\rset\to\fD_c$
be a $C^1$-differentiable function, i.e.,
\begin{equation}\label{eq:diffDef}
\lim_{n\to\infty} \left( F\!\left(t+\frac{s}{n}\right)\cdot F(t)^{-1} \right)^n
\end{equation}
converges uniformly on each compact interval to a one-parameter subgroup
\[\exp (s f(t))\]
where $f:\rset\to\fd$ is the derivative $\dot{F}(t)$ of $F(t)$ at $t\in\rset$, see \cite[p.\ 10]{omori97}.
But we can take the logarithm of $F$
\[\tilde{f}(t) := \log F(t)\]
for all $t\in\rset$ to see that $\tilde{f}$ is $C^1$ since $\log$ is smooth by \Cref{cor:lieprops} (iii).
By \Cref{cor:lieprops} (ii), we have $f = \tilde{f}$.
Hence, $C^1(\fD_c,\fd_c)$ coincides with the set of all $C^1$-curves in $\fD_c$ under the Fr\'echet topology of $\fD_c$.
\end{proof}

In the previous proof we can also replace (\ref{eq:diffDef}) by the fact that a function
\[F:\rset\to\rset[[\partial_1,\dots,\partial_n]],\quad t\mapsto F(t) := \sum_{\alpha\in\nset_0^n} F_\alpha(t)\cdot\partial^\alpha\]
is $C^m$ for some $m\in\nset_0$ if and only if every coordinate function $F_\alpha$ is $C^m$.

\section*{Problems}%%%%%%%%%%%%%%%%%%%%%
\addcontentsline{toc}{section}{Problems}

\begin{prob}\label{prob:}
Let $n=1$, $d=3$, and let
\[A = 2\partial_x - \partial_x^2 + \partial_x^3 \quad\in\fd_{c,3}.\]
Calculate $\exp A\in\fD_{c,3}$.
\end{prob}

\begin{prob}\label{prob:}
Let $n=1$, $d=3$, and let
\[T = 1 - \partial_x + \partial_x^2 + \partial_x^3 \quad\in\fD_{c,3}.\]
Calculate $\log T\in\fd_{c,3}$.
\end{prob}

\chapter{Generators of $\rset^n$-Positivity Preserving Semi-Groups with Constant Coefficients}
\label{ch:generatorsConstant}%%%%%%%%%%%%%%%%%%%%%%%%%%%%%%%%%%%%%%%%%%%%%%%%%%%%%%
%%%%%%%%%%%%%%%%%%%%%%%%%%%%%%%%%%%%%%%%%%%%%%%%%%%%%%%%%%%%%%%%%%%%%%%%%%%%%%%%%%%

In the previous chapter we established the rules how to calculate and deal with
\[\exp A\]
when
\[A:\rset[x_1,\dots,x_n]\to\rset[x_1,\dots,x_n]\]
has constant coefficients.
We can therefore now look at maps $A$ such that
\[(\exp (t\cdot A))_{t\geq 0}\]
is a positivity preserving semi-group.
The results in this chapter are published in \cite{didio24posPresConst}.

\section{Generators with Constant Coefficients and Finite Degree}

\begin{dfn}\label{dfn:fd+}
Let $n\in\nset$.
We define the set
\[\fD_{c,+} := \left\{A\in\fD_c \,\middle|\, A\ \text{is a positivity preserver}\right\}\label{fDc+}\]
of all \emph{positivity preservers with constant coefficients} and we define the set
\[\fd_{c,+}:=\left\{A\in\fd_c\,\middle|\,\exp(tA)\in\fD_{c,+}\ \text{for all}\ t\geq 0\right\}\label{fdc+}\]
of all \emph{generators of positivity preservers with constant coefficients}.\index{generator!positivity perservers!with constant coefficient}
\end{dfn}

In \Cref{exm:translation} (and Problem \ref{prob:translation}) we have already seen that
\[(e^{t\partial_x} f)(x) = f(x+t)\]
for all $x,t\in\rset$.
Hence,
\[e^{t\partial_x}\]
is a positivity preserver with constant coefficients for all $t\in\rset$ and $\partial_x\in\fd_{c,+}$.

In \Cref{exm:heat} we have seen that
\[e^{t\partial_x^2}\]
is a positivity preserver with constant coefficients for all $t\geq 0$.
Hence, $\partial_x^2\in\fd_{c,+}$.

In \Cref{exm:notpos} we have also seen that $\partial_x^k\not\in\fd_{c,+}$ for all $k\geq 3$.

The following result shows that the cases in \Cref{exm:translation} and \Cref{exm:heat} are the only generators of positivity preserves of finite rank.

\begin{lem}[{\cite[Lem.\ 4.4]{didio24posPresConst}}]
Let
\[A = \sum_{j=1}^k a_j\partial_x^j \quad\in\fd_{c,+}.\]
Then $k\leq 2$.
\end{lem}
\begin{proof}
Let $k\geq 3$ and $a_k = 1$.
By \Cref{exm:notpos}, we have that $\exp(\partial_x^k)\not\in\fD_{c,+}$, i.e., it is not a positivity preserver.
Hence, there exists a $f_0\in\rset[x]$ with $f_0\geq 0$ and $x_0\in\rset$ such that $[\exp(\partial_x^k) f_0](x_0) = -1$.

Assume to the contrary that $A\in\fd_{c,+}$.
By scaling $x$ and $A$ we have that
\begin{equation}\label{eq:aScaling}
A_\lambda := \sum_{j=1}^k \lambda^{k-j} a_j \partial_x^j\quad\in\fd_+
\end{equation}
holds for all $\lambda > 0$.
By \Cref{thm:mainFrechetLieGroups},
\[[\exp(A_\lambda) f_0](x_0)\]
is continuous in $\lambda$.
Since $(\exp(A_0) f_0)(x_0)=-1$, there exists a $\lambda_0 > 0$ such that
\[[\exp(A_{\lambda_0}) f_0](x_0) < 0,\]
i.e., $A\not\in\fd_{c,+}$ and therefore we have $k\leq 2$.
\end{proof}

It is easy to see that the previous result also holds for $n\geq 2$.
To see this let $\alpha\in\nset_0^n$ with $|\alpha|\geq 3$.
Then from \Cref{thm:kPosPresCara} it follows that
\[\exp(\partial^\alpha)\not\in\fD_{c,+}.\]
Choosing the same scaling argument (\ref{eq:aScaling}) we find $\partial^\alpha\not\in\fd_{c,+}$.

\section{Infinitely Divisible Measures}

\begin{dfn}
Let $n\in\nset$.
If there exists a measure $\nu$ such that
\[\mu = \nu^{*k},\]
then a measure $\mu$ on $\rset^n$ is called \emph{divisible by $k\in\nset$}.\index{measure!divisible}
If it is divisible by any $k\in\nset$, then a measure $\mu$ on $\rset^n$ is called \emph{infinitely divisible}.\index{measure!infinitely divisible}
\end{dfn}

Infinitely divisible measures are fully characterized by the \emph{L\'evy--Khinchin formula}.\index{formula!L\'evy--Khinchin}\index{L\'evy--Khinchin formula}

\begin{thm}[L\'evy--Khinchin, see e.g.\ {\cite[Cor.\ 15.8]{kallenberg02}} or {\cite[Satz 16.17]{klenkewtheorie}}]\label{thm:leviKhinchinFormula}
Let $n\in\nset$ and let $\mu$ be a measure on $\rset^n$.
Then the following are equivalent:
\begin{enumerate}[(i)]
\item $\mu$ is infinitely divisible.

\item There exist a vector $b\in\rset^n$, a symmetric matrix $\Sigma\in\rset^{n\times n}$ with $\Sigma\succeq 0$, and a measure $\nu$ on $\rset^n$ with $\nu(\{0\}) = 0$ such that
\[\log \int e^{itx}~\diff\mu(x) = itb - \frac{1}{2} t^T\Sigma t + \int (e^{itx} - 1 - itx \cdot\chi_{\|x\|_2<1})~\diff\mu(x)\]
for the characteristic function of $\mu$.
\end{enumerate}
\end{thm}

\section{Generators of $\rset^n$-Positivity Preserving Semi-Groups with Constant Coefficients}

For $\fD_{c,+}$ and $\fd_{c,+}$ the following holds.

\begin{cor}[{\cite[Cor.\ 4.5]{didio24posPresConst}}]\label{cor:fDplusfdplusProps}
Let $n\in\nset$.
Then the following hold:
\begin{enumerate}[(i)]
\item $\fD_{c,+}$ is a closed and convex set.

\item $\fd_{c,+}$ is a non-trivial, closed, and convex cone.
\end{enumerate}
\end{cor}

\begin{cor}[{\cite[Cor.\ 4.6]{didio24posPresConst}}]\label{cor:posPDE}
Let $n\in\nset$ and $A\in\fd_c$.
Then the following are equivalent:
\begin{enumerate}[(i)]
\item $A\in\fd_{c,+}$.
\item The unique solution $p_t$ of
\begin{equation}\label{eq:timePDE}
\partial_t p = Ap
\end{equation}
for any initial value $p_0\in\pos(\rset^n)$ fulfills $p_t\in\pos(\rset^n)$ for all $t\geq 0$.
\end{enumerate}
\end{cor}
\begin{proof}
Since $p_t = \exp(tA)p_0$ is the unique solution of the time evolution (\ref{eq:timePDE}), we have that (i) $\Leftrightarrow$ $\exp(tA)$ is a positivity preserver for all $t\geq 0$ $\Leftrightarrow$ (ii).
\end{proof}

While we have
\[\fd_{c,+} \subseteq \log\fD_{c,+},\]
equality does not hold as we will see in \Cref{cor:generatorsNotLog}.
The existence of a positivity preserver is equivalent to the existence of an infinitely divisible representing measure as the following result shows.

\begin{thm}[{\cite[Main Thm.\ 4.7]{didio24posPresConst}}]\label{thm:infinitelyDivisibleMeasure}
Let $n\in\nset$.
The following are equivalent:
\begin{enumerate}[(i)]
\item $A\in\fd_{c,+}$.

\item $\exp A$ has an infinitely divisible representing measure.

\item $\exp(tA)$ has an infinitely divisible representing measure for some $t>0$.

\item $\exp(tA)$ has an infinitely divisible representing measure for all $t>0$.
\end{enumerate}
\end{thm}
\begin{proof}
(i) $\Rightarrow$ (ii):
Let $A\in\fd_{c,+}$, i.e.,
\[\exp (t A)\in\fD_{c,+}\]
has a representing measure $\mu_t$ for all $t\in [0,\infty)$.
Set
\[\nu_k := (\mu_{1/k!})^{*k!}.\]
Then $\nu_k$ is a representing measure of $\exp A $ for all $k\in\nset$.
Since $\rset[x_1,\dots,x_n]$ is an adapted space, $(\nu_k)_{k\in\nset}$ is vaguely compact by \cite[Thm.\ 1.19]{schmudMomentBook} and there exists a subsequence $(k_i)_{i\in\nset}$ such that $\nu_{k_i}\to\nu$ and $\nu$ is a representing measure of $\exp A$.

It remains to show that $\nu$ is infinitely divisible, i.e., for every $l\in\nset$ there exists a measure $\omega_l$ with $\omega_l^{*l} = \nu$.

Let $l\in\nset$.
For $i\geq l$ we define
\[\omega_{l,i}:= (\mu_{1/k_i!})^{*k_i!/l}\]
i.e., $\omega_{l,i}$ is a representing measure of $\exp(A/l)$.
Again, since $\rset[x_1,\dots,x_n]$ is an adapted space by \cite[Thm.\ 1.19]{schmudMomentBook} there exists a subsequence $(i_j)_{j\in\nset}$ such that $\omega_{l,i_j}$ converges to some $\omega_l$, i.e., $\omega_{l,i_j}\xrightarrow{j\to\infty}\omega_l$.
Hence,
\[(\omega_l)^{*l} = \lim_{j\to\infty} (\omega_{l,i_j})^{*l} = \lim_{j\to\infty} \nu_{k_{i_j}} = \nu,\]
i.e., $\nu$ is divisible by all $l\in\nset$ and hence $\nu$ is an infinitely divisible representing measure of $\exp A$.

(ii) $\Rightarrow$ (i):
Let $\mu_1$ be an infinitely divisible representing measure of $\exp A$.
Then
\[\mu_q := \mu_1^{*q}\]
exists for all $q\in\qset\cap[0,\infty)$ and it is a representing measure of $\exp(qA)$, i.e.,
\[\exp(qA)\in\fD_{c,+}\]
for all $q\in [0,\infty)\cap\qset$.
Since, by \Cref{thm:mainFrechetLieGroups},
\[\exp:\fd_c\to\fD_c\]
is continuous and by \Cref{cor:fDplusfdplusProps} (i) $\fD_{c,+}$ is closed, we have that
\[\exp(qA)\in\fD_{c,+}\]
for all $q\geq 0$.
Hence, we have $A\in\fd_{c,+}$.

(iv) $\Rightarrow$ (iii): Clear.

(iii) $\Rightarrow$ (i):
By ``(ii) $\Leftrightarrow$ (i)'' we have that
\[qtA\in\fd_{c,+}\]
for $t>0$ and all $q\in [0,\infty)\cap\qset$.
Since $\fd_{c,+}$ is closed by \Cref{cor:fDplusfdplusProps} (ii), we have
\[q_itA\to A\in\fd_{c,+}\]
for $q_i\in\qset$ with $q_i\to t^{-1}$ as $i\to\infty$.

(i) $\Rightarrow$ (iv):
Since $A\in\fd_+$ and $\fd_{c,+}$ is a closed convex cone by \Cref{cor:fDplusfdplusProps} (ii), we have that
\[tA\in\fd_{c,+}\]
for all $t>0$ and hence by ``(i) $\Leftrightarrow$ (ii)'' we have that $\exp(tA)$ has an infinitely divisible representing measure for all $t>0$.
\end{proof}

\begin{exm}[\Cref{exm:translation} continued,[{\cite[Exm.\ 4.8]{didio24posPresConst}}]
Let $n=1$ and $a\in\rset$.
Then
\[\exp(a\partial_x) = \sum_{k\in\nset_0} \frac{a^k}{k!}\cdot\partial_x^k\]
is represented by $\mu = \delta_a$ since $\delta_a$ is the representing measure of the moment sequence $(a^k)_{k\in\nset_0}$.
For any $r>0$, we have
\[\delta_{a/r}^{*r} = \delta_a,\]
i.e., $\delta_a$ is infinitely divisible.
In fact, $\delta_a$ are the only compactly supported infinitely divisible measures, see e.g.\ \cite[p.\ 316]{klenkewtheorie}.
Hence, by \Cref{thm:infinitelyDivisibleMeasure}, $\partial_x\in\fd_{c,+}$.
\exmsymbol
\end{exm}

\begin{exm}[{\cite[Exm.\ 4.9]{didio24posPresConst}}]\label{exm:nonPosGen}
Let
\[A\in\fD_{c,+}\]
be the positivity preserver represented by the measure
\[\diff\mu = \chi_{[0,1]^n}~\diff\lambda\]
where $\lambda$ is the $n$-dimensional Lebesgue measure and $\chi_{[0,1]^n}$ is the characteristic function of $[0,1]^n$.
Since $\supp\mu$ is compact, $\mu$ is unique.

It is known that the only infinitely divisible measures with compact support are $\delta_x$ for $x\in\rset^n$, see e.g.\ \cite[p.\ 316]{klenkewtheorie}.
Therefore, we have that $\mu$ is not infinitely divisible and hence $\log A\not\in\fd_{c,+}$.
\exmsymbol
\end{exm}

The previous example implies that the inclusion
\[\fd_+\subseteq\log\fD_+\]
is proper.

\begin{cor}[{\cite[Cor.\ 4.10]{didio24posPresConst}}]\label{cor:generatorsNotLog}
Let $n\in\nset$.
Then
\[\fd_{c,+} \subsetneq \log\fD_{c,+}.\]
\end{cor}
\begin{proof}
We have $\log\fD_{c,+}\setminus\fd_{c,+}\neq\emptyset$ by \Cref{exm:nonPosGen}.
\end{proof}

We have seen in \Cref{thm:infinitelyDivisibleMeasure} the one-to-one correspondence between a positivity preserver
\[e^A\]
with
\[A\in\fD_{c,+}\]
having an infinitely divisible representing measure and
\[A\in\fD_{c,+}\]
being a generator.
The infinitely divisible measures are fully characterized by the L\'evy--Khinchin formula, see \Cref{thm:leviKhinchinFormula}.
The L\'evi--Khinchin formula is used in the following result to fully characterize the generators $\fd_{c,+}$ of the positivity preservers $\fD_{c,+}$.

\begin{thm}[{\cite[Main Thm.\ 4.11]{didio24posPresConst}}]\label{thm:mainPosGenerators}
Let $n\in\nset$.
Then the following are equivalent:
\begin{enumerate}[(i)]
\item $\displaystyle A = \sum_{\alpha\in\nset_0^n\setminus\{0\}} \frac{a_\alpha}{\alpha!}\cdot\partial^\alpha \in\fd_{c,+}$.

\item There exists a symmetric matrix $\Sigma = (\sigma_{i,j})_{i,j=1}^n\in\rset^n$ with $\Sigma\succeq 0$, a vector $b = (b_1,\dots,b_n)^T\in\rset^n$, and a measure $\nu$ on $\rset^n$ with
\[\nu(\{0\}) = 0 \qquad\text{and}\qquad \int_{\rset^n} |x^\alpha|~\diff\nu(x)<\infty\]
for all $\alpha\in\nset_0^n$ with $|\alpha|\geq 2$ such that
\begin{align*}
a_{e_i} &= b_i + \int_{\|x\|_2\geq 1} x_i~\diff\nu(x) && \text{for all}\ i=1,\dots,n,\\
a_{e_i+e_j} &= \sigma_{i,j} + \int_{\rset^n} x^{e_i + e_j}~\diff\nu(x) && \text{for all}\ i,j=1,\dots,n,
\intertext{and}
a_\alpha &= \int_{\rset^n} x^\alpha~\diff\nu(x) &&\text{for all}\ \alpha\in\nset_0^n\ \text{with}\ |\alpha|\geq 3.
\end{align*}
\end{enumerate}
\end{thm}
\begin{proof}
By \Cref{thm:infinitelyDivisibleMeasure} ``(i) $\Leftrightarrow$ (ii)'', we have that (i) $A\in\fd_{c,+}$ if and only if $\exp A$ has an infinitely divisible representing measure $\mu$, i.e., by \Cref{thm:kPosPresCara}, we have
\begin{equation}\label{eq:levi1}
\exp A = \sum_{\alpha\in\nset_0^n} \frac{1}{\alpha!}\cdot \int_{\rset^n} x^\alpha~\diff\mu(x)\cdot\partial^\alpha.
\end{equation}
By \Cref{thm:mainFrechetLieGroups}, we can take the logarithm and hence (\ref{eq:levi1}) is equivalent to
\begin{equation}\label{eq:levi2}
A = \sum_{\alpha\in\nset_0^n\setminus\{0\}} \frac{a_\alpha}{\alpha!}\cdot\partial^\alpha = \log \left(\sum_{\alpha\in\nset_0^n} \frac{1}{\alpha!}\cdot \int_{\rset^n} x^\alpha~\diff\mu(x)\cdot\partial^\alpha\right).
\end{equation}
With the isomorphism
\[\cset[[\partial_1,\dots,\partial_n]]\to\cset[[t_1,\dots,t_n]],\quad \partial_1\mapsto i t_1,\ \dots,\ \partial_n\mapsto it_n\]
we have that (\ref{eq:levi2}) is equivalent to
\begin{align}
\sum_{\alpha\in\nset_0^n\setminus\{0\}} \frac{a_\alpha}{\alpha!}\cdot(it)^\alpha
&= \log \left(\sum_{\alpha\in\nset_0^n} \frac{1}{\alpha!}\cdot \int_{\rset^n} x^\alpha~\diff\mu(x)\cdot (it)^\alpha\right).\label{eq:levi3}
\intertext{But the right hand side of (\ref{eq:levi3}) is now the characteristic function}
&= \log \int e^{i tx}~\diff\mu(x)\notag
\intertext{of $\mu$.
Hence, by the L\'evy--Khinchin formula (see \Cref{thm:leviKhinchinFormula}), we have}
&= i bt -\frac{1}{2}t^T\Sigma t + \int (e^{itx} - 1 - itx \cdot\chi_{\|x\|_2<1})~\diff\nu(x).\label{eq:levi4}
\end{align}
After a power series expansion of
\[e^{itx}\]
in the Fr\'echet topology of
\[\cset[[x_1,\dots,x_n]],\]
see \Cref{exm:frechet}, and a comparison of coefficients we have that (\ref{eq:levi4}) is equivalent to (ii) which ends the proof.
\end{proof}

From the previous result we see that the difference between
\[\fd_{c,+}\]
and a moment sequence is that the representing (L\'evy) measure $\nu$ in (\ref{eq:levi4}) can have a singularity of order $\leq 2$ at the origin.

\section*{Problems}%%%%%%%%%%%%%%%%%%%%%
\addcontentsline{toc}{section}{Problems}

\begin{prob}
Show that
\[e^{t\cdot\partial_x^2}\]
is not a positivity preserver for any $t<0$.
\end{prob}

\begin{prob}\label{prob:}
Let $n\in\nset$ and $a,b\in\rset^n$.
Show
\[\delta_a * \delta_b = \delta_{a+b}.\]
\end{prob}

\chapter{The Set $\fd$}%%%
\label{ch:fd}%%%%%%%%%%%%%
%%%%%%%%%%%%%%%%%%%%%%%%%%

We have seen several cases of linear maps
\[A:\rset[x_1,\dots,x_n]\to\rset[x_1,\dots,x_n]\]
such that for every $t\geq 0$ the maps
\[\exp(t\cdot A):\rset[x_1,\dots,x_n]\to\rset[x_1,\dots,x_n]\]
is well-defined, i.e.,
\[\deg \exp(t\cdot A)p < \infty\]
for all $p\in\rset[x_1,\dots,x_n]$ and $t\geq 0$.
For $A$ with constant coefficients and for $A$ with
\[A\rset[x_1,\dots,x_n]_{\leq d}\subseteq\rset[x_1,\dots,x_n]_{\leq d}\]
this was the case.
The question we want to attack now is to describe all linear operators $A$ such that $\exp(t\cdot A)$ is well-defined, i.e., maps into $\rset[x_1,\dots,x_n]$.
The results presented in this chapter are from \cite{didio25PosToSoS}.

\section{The Set $\fd$ and some Properties of $\fd$}
%%%%%%%%%%%%%%%%%%%%%%%%%%%%%%%%%%%%%%%%%%%%%%%%%%%%

\begin{dfn}
Let
\[A:\rset[x_1,\dots,x_n]\to\rset[x_1,\dots,x_n]\]
be linear and $f_0\in\rset[x_1,\dots,x_n]$.
The unique solution of
\[\partial_t f = Af\]
with initial values $t_0=0$ and $f(\,\cdot\,,0) = f_0$ is
\[f(\,\cdot\,,t) = e^{tA} f_0.\]
$(e^{tA})_{t\in\rset}$ is called \emph{well-defined}, if
\[e^{tA}:\rset[x_1,\dots,x_n]\to\rset[x_1,\dots,x_n]\]
for all $t\in\rset$ and
\[f(x,t) = (e^{tA}f_0)(x) = \sum_{\alpha\in\nset_0^n} c_\alpha(t)\cdot x^\alpha\]
is analytic in $t$, i.e., $c_\alpha$ is analytic for all $\alpha\in\nset_0^n$.
\end{dfn}

\begin{dfn}\label{dfn:fd}
Let $n\in\nset$.
We define
\begin{multline*}
\fd := \Big\{A:\rset[x_1,\dots,x_n]\to\rset[x_1,\dots,x_n]\ \text{linear} \,\Big|\\ e^{tA}:\rset[x_1,\dots,x_n]\to\rset[x_1,\dots,x_n]\
\text{well-defined for all}\ t\in\rset\Big\}.
\end{multline*}
\end{dfn}

From well-definedness of
\[(e^{tA})_{t\in\rset}\]
it is clear, that it is sufficient to require for each $f_0\in\rset[x_1,\dots,x_n]$ there exists a $\varepsilon = \varepsilon(f_0) > 0$ such that
\[(e^{tA}f_0)_{t\in [0,\varepsilon(f_0))}\]
is well-defined, i.e., the coefficients are analytic.

Clearly,
\[\fd = \rset\cdot\fd,\]
i.e., $\fd$ is a cone and
\[\fd = -\fd.\]
The following theorem gives a characterization of $\fd$.

\begin{thm}[{\cite[Thm.\ 4.4]{didio25PosToSoS}}]\label{thm:fdCharac}
Let $n\in\nset$ and let
\[A:\rset[x_1,\dots,x_n]\to\rset[x_1,\dots,x_n]\]
be linear.
Then the following are equivalent:
\begin{enumerate}[(i)]
\item $A\in\fd$.

\item $\displaystyle \sup_{k\in\nset_0} \deg A^k x^\alpha < \infty$ for all $\alpha\in\nset_0^n$.

\item For all $i\in\nset_0$, there exist subspaces $V_i\subseteq\rset[x_1,\dots,x_n]$ with
\begin{enumerate}[(a)]
\item $\dim V_i < \infty$,

\item $\displaystyle \bigcup_{i\in\nset_0} V_i = \rset[x_1,\dots,x_n]$, and

\item $AV_i \subseteq V_i$ for all $i\in\nset_0$.
\end{enumerate}

\end{enumerate}
\end{thm}
\begin{proof}
The implications ``(ii) $\Rightarrow$ (i)'' and ``(iii) $\Rightarrow$ (i)'' are clear.

(i) $\Rightarrow$ (ii):
Since $e^{tA}$ is well-defined, for each $\alpha\in\nset_0^n$, there exists a degree $D=D(\alpha)\in\nset_0^n$ such that
\[e^{tA}x^\alpha\ \subseteq\ \rset[x_1,\dots,x_n]_{\leq D}\]
for all $t\in\rset$, see Problem \ref{prob:subDegree}.
Hence,
\[A^k x^\alpha\ =\ \partial^k_t e^{tA}x^\alpha \Big|_{t=0}\ \in\ \rset[x_1,\dots,x_n]_{\leq D}\]
for all $k\in\nset_0$, which proves (ii).

[(i) $\Leftrightarrow$ (ii)] $\Rightarrow$ (iii):
Let $\alpha\in\nset_0^n$.
By (i) and (ii), there exists a $D=D(\alpha)\in\nset_0$ with
\[e^{tA}x^\alpha,\ A^k x^\alpha\ \in\ \rset[x_1,\dots,x_n]_{\leq D}\]
for all $k\in\nset_0$ and all $t\in\rset$.
Set
\[V_{\alpha,0} := \rset\cdot x^\alpha.\]
By (ii),
\[AV_{\alpha,0}\quad \subseteq\quad \rset[x_1,\dots,x_n]_{\leq D}.\]
For all $i\in\nset_0$, define
\[V_{\alpha,i+1} := V_{\alpha,i} + AV_{\alpha,i},\]
i.e.,
\[V_{\alpha,k}=\rset\cdot x^\alpha +\rset\cdot Ax^\alpha +\dots+\rset\cdot A^k x^\alpha\]
for all $k\in\nset_0$.
By (ii) and the definition of $V_{\alpha,i+1}$,
\[V_{\alpha,i} \quad\subseteq\quad V_{\alpha,i+1} \quad\subseteq\quad \rset[x_1,\dots,x_n]_{\leq D}\]
and hence
\begin{equation}\label{eq:increasingViSequence}
V_{\alpha,0} \quad\subseteq\quad V_{\alpha,1} \quad\subseteq\quad \dots \quad\subseteq\quad \rset[x_1,\dots,x_n]_{\leq D}.
\end{equation}
Since (\ref{eq:increasingViSequence}) is an increasing sequence of finite dimensional vector spaces $V_{\alpha,i}$ bounded from above by the finite dimensional vector space $\rset[x_1,\dots,x_n]_{\leq D}$, there exists an index $I(\alpha)\in\nset_0$ such that
\[V_{\alpha,I(\alpha)} = V_{\alpha,I(\alpha)+1},\]
i.e.,
\[\dim V_{\alpha,I(\alpha)} < \infty \qquad\text{and}\qquad AV_{\alpha,I(\alpha)} \subseteq V_{\alpha,I(\alpha)}.\]
Since $\alpha\in\nset_0^n$ was arbitrary and $x^\alpha\in V_{\alpha,I(\alpha)}$,
\[\rset[x_1,\dots,x_n] = \bigcup_{\alpha\in\nset_0^n} V_{\alpha,I(\alpha)}.\]
Since $\nset_0^n$ is countable, we proved (iii) (a) -- (c).
\end{proof}

\Cref{thm:fdCharac} tells us that
\[e^{tA}\]
is only the matrix exponential function, since for an operator $A\in\fd$ we only need to know $A$ on every finite dimensional invariant subspace $V_i$, i.e., its calculation is easy as soon as the $V_i$ for $A$ are known.
An algorithm for the calculation of these $V_i$ is explicitly given in the proof of \Cref{thm:fdCharac}.
We therefore have the following corollary.

\begin{cor}[{\cite[Thm.\ 4.4]{didio25PosToSoS}}]\label{cor:taylor}
Let $n\in\nset$ and $A\in\fd$.
Then
\[e^{tA} = \sum_{k\in\nset_0} \frac{t^k\cdot A^k}{k!} = \lim_{k\to\infty} \left(\id + \frac{A}{k} \right)^k = \lim_{k\to\infty} \left(\id - \frac{A}{k} \right)^{-k}\]
for all $t\in\rset$.
\end{cor}
\begin{proof}
See Problem \ref{prob:expd}.
\end{proof}

While $\fd$ is a cone, it is not convex.

\begin{prop}\label{lem:fdNotConvex}
Let $n\in\nset$.
$\fd$ is a cone and it is not convex.
\end{prop}

The cone properties follows immediately from \Cref{dfn:fd}.
The non-convexness of $\fd$ is shown by the following examples.
It is sufficient to show it for $n=1$.

\begin{exm}[{\cite[Exm.\ 4.2]{didio25PosToSoS}}]\label{exm:ABshiftEvenOdd}
Let $n=1$.
Define the linear operators
\[A:\rset[x]\to\rset[x] \qquad\text{and}\qquad B:\rset[x]\to\rset[x]\]
by
\[Ax^k := \begin{cases}
x^k & \text{for}\ k=2m\ \text{with}\ m\in\nset_0,\\
x^{k+1} & \text{for}\ k=2m-1\ \text{with}\ m\in\nset
\end{cases}\]
and
\[Bx^k := \begin{cases}
x^{k+1} & \text{for}\ k=2m\ \text{with}\ m\in\nset_0,\\
x^k & \text{for}\ k=2m-1\ \text{with}\ m\in\nset
\end{cases}\]
for all $k\in\nset_0$ with linear extension to all $\rset[x]$.
Then
\[A\rset[x]_{\leq 2m}\ \subseteq\ \rset[x]_{\leq 2m} \qquad\text{and}\qquad B\rset[x]_{\leq 2m+1}\ \subseteq\ \rset[x]_{\leq 2m+1}\]
for all $m\in\nset_0$.
Since
\[A|_{\rset[x]_{\leq 2m}} \qquad\text{and}\qquad B|_{\rset[x]_{\leq 2m+1}}\]
are linear operators on finite dimensional spaces, i.e., matrices,
\[e^{tA} \qquad\text{and}\qquad e^{tB}\]
are well-defined for every $p\in\rset[x]$.
Hence, they are well-defined as linear maps
\[\rset[x]\to\rset[x]\]
and therefore $A,B\in\fd$.

However,
\[(A+B)x^k\quad =\quad \begin{cases}
x^k + x^{k+1} & \text{for}\ k=2m\ \text{with}\ m\in\nset_0,\\
x^{k+1} + x^k & \text{for}\ k=2m-1\ \text{with}\ m\in\nset
\end{cases}\quad =\quad x^k + x^{k+1}\]
for all $k\in\nset_0$, i.e.,
\[\deg \left(\sum_{i=0}^d \frac{t^i}{i!}\cdot (A+B)^i 1\right) = d \quad\xrightarrow{d\to\infty}\quad\infty\]
for all $t\neq 0$.
Hence, by \Cref{thm:fdCharac} (ii),
\[e^{(A+B)}\]
is not a map
\[\rset[x]\to\rset[x]\]
and $A+B\notin\fd$.
\exmsymbol
\end{exm}

We have seen in \Cref{exm:ABshiftEvenOdd} that $\fd$ is not closed under addition.
The following example shows that $\fd$ is also not closed under multiplication.

\begin{exm}[\Cref{exm:ABshiftEvenOdd} continued, {\cite[Exm.\ 4.6]{didio25PosToSoS}}]\label{exm:ABshiftEvenOdd2}
Let
\[A,B:\rset[x]\to\rset[x]\]
be given as in \Cref{exm:ABshiftEvenOdd}.
Then
\begin{equation}\label{eq:AB}
AB x^{2m} = x^{2m+2},\qquad ABx^{2m+1} = x^{2m+2},
\end{equation}
and
\begin{equation}\label{eq:BA}
BAx^{2m} = x^{2m+1},\qquad BA x^{2m+1} = x^{2m+3},
\end{equation}
for all $m\in\nset_0$, i.e., by \Cref{thm:fdCharac} (ii), $AB\notin\fd$ as well as $BA\notin\fd$.
\exmsymbol
\end{exm}

If two operators $A,B\in\fd$ possess a common family
\[\{V_i\}_{i\in\nset_0}\]
of invariant and finite dimensional subspace $V_i$ as provided by \Cref{thm:fdCharac} (iii), then
\[A+B\in\fd, \quad AB\in\fd,\quad\text{and}\quad BA\in\fd\]
and the Lie bracket $[\,\cdot\,,\,\cdot\,]$ is given by
\[[A,B] := AB - BA\ \in\fd.\]

In general, $\fd$ is not closed under this Lie bracket.
The operators $A,B\in\fd$ in \Cref{exm:ABshiftEvenOdd} provide an example of $[A,B]\notin\fd$.

\begin{exm}[\Cref{exm:ABshiftEvenOdd} and \ref{exm:ABshiftEvenOdd2} continued, {\cite[Exm.\ 4.7]{didio25PosToSoS}}]\label{exm:ABshiftEvenOdd3}
Let
\[A:\rset[x]\to\rset[x] \qquad\text{and}\qquad B:\rset[x]\to\rset[x]\]
be defined as in \Cref{exm:ABshiftEvenOdd}.
By (\ref{eq:AB}) and (\ref{eq:BA}) in \Cref{exm:ABshiftEvenOdd2},
\[[A,B]x^{2m} = x^{2m+2} - x^{2m+1} \qquad\text{and}\qquad [A,B]x^{2m+1} = x^{2m+2} - x^{2m+3}\]
for all $m\in\nset_0$, i.e.,
\[[A,B]\notin\fd\]
by \Cref{thm:fdCharac} (ii).
\exmsymbol
\end{exm}

In summary, for $A,B\in\fd$, in order to have
\[A+B\in\fd,\quad AB\in\fd,\quad BA\in\fd, \quad\text{and}\quad [A,B]\in\fd\]
it is sufficient that the operators $A$ and $B$ possess a common family
\[\{V_i\}_{i\in\nset_0}\]
of invariant finite dimensional subspaces $V_i$ in \Cref{thm:fdCharac} (iii).

\section{The Finite Dimensional Invariant Subspaces $V_i$ of $A\in\fd$}

The \Cref{exm:ABshiftEvenOdd} and \Cref{thm:fdCharac} (iii) showed that if for a linear operator
\[A:\rset[x_1,\dots,x_n]\to\rset[x_1,\dots,x_n]\]
there exist
\[d_0 < d_1 < d_2 < \dots\]
in $\nset_0$ with
\begin{equation}\label{eq:invariantDegree}
A\rset[x_1,\dots,x_n]_{\leq d_i}\ \subseteq\ \rset[x_1,\dots,x_n]_{\leq d_i}
\end{equation}
for all $i\in\nset_0$, then $A\in\fd$.
This was the property used in \cite{didio24posPresConst,didio25KPosPresGen}.

The next examples shows that (\ref{eq:invariantDegree}) is sufficient, but (\ref{eq:invariantDegree}) is not necessary as seen from \Cref{thm:fdCharac} (ii) and (iii).

\begin{exm}[{\cite[Exm.\ 4.3]{didio25PosToSoS}}]
Let $n=1$ and let
\[(p_i)_{i\in\nset}= (2,3,5,7,11,\dots)\]
be the list of all prime numbers.
Define the linear map
\[A:\rset[x]\to\rset[x]\]
by
\[Ax^k := \begin{cases}
x^{p_{m+1}} & \text{for}\ k=2m\ \text{with}\ m\in\nset,\\
0 & \text{else}.
\end{cases}\]
Then
\[A\rset[x]_{\leq d}\ \not\subseteq\ \rset[x]_{\leq d}\]
for all $d\in\nset$ with $d\geq 8$.
However,
\[A1 = 0, \quad Ax = 0, \quad Ax^2 = x^{p_2} = x^3,\quad Ax^3 = 0,\quad Ax^4 = x^{p_3} = x^5,\quad \dots\]
and therefore
\[A^2 x^k = 0\]
for all $k\in\nset_0$, i.e.,
\[e^{tA}:\rset[x]\to\rset[x]\]
is well-defined for all $t\in\rset$ and $A\in\fd$ by \Cref{thm:fdCharac} (ii).
\exmsymbol
\end{exm}

In general, in \Cref{thm:fdCharac} we can only have
\[V_i \subseteq V_{i+1}\]
for the invariant subspaces $V_i$ of an operator $A\in\fd$.
This shows the next example.

\begin{exm}[{\cite{didio25PosToSoS}}]\label{exm:jordan}
Let $n=1$ and $\partial_x\in\fd$.
Then every finite dimensional invariant subspace
\[V\subseteq\rset[x]\]
of $\partial_x$ is of the form
\[\rset[x]_{\leq d}\]
for some $d\in\nset_0$.

To see this, let $V$ be an invariant subspace of $\partial_x$ with
\[\dim V < \infty.\]
Let $p\in V$ be of highest degree, i.e.,
\[\deg q \leq \deg p =: d\]
for all $q\in V$ and hence
\[V\subseteq\rset[x]_{\leq d}.\]
Then
\[\partial_x^k p\in V \qquad\text{and}\qquad \deg (\partial_x^k p) = d - k\]
for all $k=1,\dots,d$.
Since
\[\deg (\partial_x^k p) = d-k,\]
the
\[\{\partial_x^k p\}_{k=0}^d\]
are linearly independent and span $\rset[x]_{\leq d}$, i.e.,
\[\rset[x]_{\leq d}\subseteq V.\]
In summary,
\[V=\rset[x]_{\leq d}\]
for some $d\in\nset_0$.
\exmsymbol
\end{exm}

The operator $\partial_x$ is therefore a Jordan block of infinite size and we can in general not hope for $\rset[x_1,\dots,x_n]$ to be split into disjoint invariant subspaces $V_i$ to get
\[\rset[x_1,\dots,x_n] = V_1 \oplus V_2 \oplus \dots.\]
Even when we allow
\[\dim V_i = \infty\]
in the decomposition, then \Cref{exm:jordan} shows that for $\partial_x$ only the trivial decomposition
\[\rset[x] = V_1\]
exists.

Additionally, for the invariant subspaces $V_i$ of $A\in\fd$ we have the following.

\begin{cor}[{\cite[Cor.\ 4.5]{didio25PosToSoS}}]
Let $n\in\nset$, let $A\in\fd$, and let
\[V\subseteq\rset[x_1,\dots,x_n]\]
be a finite or infinite dimensional subspace.
Then the following are equivalent:
\begin{enumerate}[(i)]
\item $AV\subseteq V$.

\item $e^{tA}V \subseteq V$ for all $t\in\rset$.
\end{enumerate}
\end{cor}
\begin{proof}
The direction ``(i) $\Rightarrow$ (ii)'' follows from
\[e^{tA} = \sum_{k\in\nset_0} \frac{t^k\cdot A^k}{k!}\]
and the direction ``(ii) $\Rightarrow$ (i)'' follows from
\[A = \partial_t e^{tA}\big|_{t=0}.\qedhere\]
\end{proof}

\section{Lie Algebras in $\fd$}

\begin{cor}[{\cite[Cor.\ 4.8]{didio25PosToSoS}}]\label{cor:fdV}
Let $n\in\nset$ and let $\cV := \{V_i\}_{i\in\nset_0}$ be a family of subspaces
\[V_i\subseteq\rset[x_1,\dots,x_n]\]
such that
\[\dim V_i < \infty \qquad\text{and}\qquad \rset[x_1,\dots,x_n]=\bigcup_{i\in\nset_0} V_i.\]
Then
\[\fd_\cV := \big\{A\in\fd \,\big|\, AV_i\subseteq V_i\ \text{for all}\ i\in\nset_0\big\}\label{fdV}\]
is a regular Fr\'echet Lie algebra in $\fd$ with Lie bracket
\[[\,\cdot\,,\,\cdot\,]:\fd_\cV\times\fd_\cV\to\fd_\cV,\quad (A,B)\mapsto [A,B] := AB - BA.\]
The set
\[G_\cV := \exp(\fd_\cV)\]
is the corresponding regular Fr\'echet Lie group.
\end{cor}
\begin{proof}
Follows immediately from \Cref{thm:fdCharac} (iii).
\end{proof}

The question is, are all regular Fr\'echet Lie algebras in $\fd$ contained in some $\fd_\cV$?
For finite dimensional Lie algebras this is true.

\begin{cor}[{\cite[Cor.\ 4.9]{didio25PosToSoS}}]\label{cor:subspaceLieAlgebras}
Let $n\in\nset$ and let $\tilde{\fd}$ be a finite dimensional regular Fr\'echet Lie algebra in $\fd$.
Then there exists a family
\[\cV = \{V_i\}_{i\in\nset_0}\]
of subspaces $V_i\subseteq\rset[x_1,\dots,x_n]$ with
\[\rset[x_1,\dots,x_n] = \bigcup_{i\in\nset_0} V_i \qquad\text{and}\qquad \dim V_i < \infty\]
for all $i\in\nset_0$ such that $\tilde{\fd}\subseteq\fd_\cV$.
\end{cor}
\begin{proof}
Let $g_1,\dots,g_N$ be a vector space basis of $\tilde{\fd}$, i.e.,
\[\dim \tilde{\fd}=N\in\nset.\]
Let
\begin{equation}\label{eq:a1aN}
a_1,\dots,a_N\in\rset \quad\text{with}\quad a_1^2 + \dots + a_N^2 = 1.
\end{equation}
Then for any $\alpha\in\nset_0^n$ there exists a
$D = D(\alpha)\in\nset$
such that
\[(a_1 g_1 + \dots + a_N g_N)^k x^\alpha \quad\subseteq\quad \rset[x_1,\dots,x_n]_{\leq D}\]
for all $a_1,\dots,a_N$ with (\ref{eq:a1aN}), since
\[(a_1 g_1 + \dots + a_N g_N)^k\]
is continuous in $a_1,\dots,a_N$ and the set (\ref{eq:a1aN}) is compact.
From here, continue as in the proof of \Cref{thm:fdCharac} step ``[(i) $\Leftrightarrow$ (ii)] $\Rightarrow$ (iii)''.
\end{proof}

An infinite dimensional regular Fr\'echet Lie algebra of $\fd$ does not need to be contained in some $\fd_\cV$ as the following example shows.

\begin{exm}[{\cite[Exm.\ 4.10]{didio25PosToSoS}}]\label{exm:nonSubspaceLieAlgebra}
Let $n\in\nset$, let $y\in\rset^n$, and let
\[l_y: \rset[x_1,\dots,x_n]\to\rset,\qquad f\mapsto l_y(f) := f(y)\]
be the point evaluation at $y$.
Define
\[\fd_y := \big\{l_y\cdot p \,\big|\, p\in\rset[x_1,\dots,x_n]\big\}.\label{fdy}\]
Then
\begin{enumerate}[(i)]
\item $\fd_y\subseteq\fd$,
\item $\alpha A+\beta B\in\fd_y$ for all $A,B\in\fd_y$ and $\alpha,\beta\in\rset$,
\item $AB\in\fd_y$ for all $A,B\in\fd_y$,
\item $[A,B]\in\fd_y$ for all $A,B\in\fd_y$, and
\item $\fd_y$ is closed,
\end{enumerate}
i.e., $\fd_y$ is a regular Fr\'echet Lie algebra in $\fd$, since
\[\alpha A+\beta B=l_y\cdot(\alpha p+\beta q)\qquad\text{and}\qquad AB=q(y)\cdot l_y\cdot p\]
with $Af = f(y)\cdot p$, $Bf = f(y)\cdot q$, and $\alpha,\beta\in\rset$.
To see that $\fd_y$ is closed let
\[A_k = l_y\cdot p_k \qquad\text{with}\qquad A_k\xrightarrow{k\to\infty} \big[A:\rset[x_1,\dots,x_n]\to\rset[x_1,\dots,x_n]\big].\]
Hence,
\[A_k1 = p_k\xrightarrow{k\to\infty} p\in\rset[x_1,\dots,x_n]\]
and
$A=l_y\cdot p\in\fd_y$.
But since $\deg p$ in $A = l_y\cdot p$ is not bounded, $\fd_y$ is not contained in any $\fd_\cV$.
\exmsymbol
\end{exm}

\section*{Problems}%%%%%%%%%%%%%%%%%%%%%
\addcontentsline{toc}{section}{Problems}

\begin{prob}\label{prob:subDegree}
Let $n\in\nset$, $a,b$ with $-\infty\leq a<b\leq\infty$, and let
\[A:\rset[x_1,\dots,x_n]\to\rset[x_1,\dots,x_n]\]
be linear such that $e^{tA}:\rset[x_1,\dots,x_n]\to\rset[x_1,\dots,x_n]$ is well-defined.
Show that there exists a $D\in\nset_0$ with
\[\sup_{t\in [a,b]} \deg e^{tA}f_0 \leq D.\]
\end{prob}

\begin{prob}\label{prob:expd}
Prove \Cref{cor:taylor}.
\end{prob}

\chapter{Generators of $K$-Positivity Preserving Semi-Groups}%%%
\label{ch:generators}%%%%%%%%%%%%%%%%%%%%%%%%%%%%%%%%%%%%%%%%%%%
%%%%%%%%%%%%%%%%%%%%%%%%%%%%%%%%%%%%%%%%%%%%%%%%%%%%%%%%%%%%%%%%

We now describe linear operators
\[A:\rset[x_1,\dots,x_n]\to\rset[x_1,\dots,x_n]\]
such that
\[e^{tA}\cC \subseteq\cC\]
for a closed convex cone $\cC\subseteq\rset[x_1,\dots,x_n]$.
In the special case $\cC = \pos(K)$ we have the generators of $K$-positivity preserving semi-groups.

\section{Description via the Resolvent}

\begin{dfn}\label{dfn:fDC}
Let $\cC$ be a convex subset of $\rset[x_1,\dots,x_n]$ which is closed in the LF-topology of $\rset[x_1,\dots,x_n]$.
We denote by
\[\fD_\cC := \left\{ T\in\fD \,\middle|\, T\cC\subseteq\cC\right\}\]
the set of all \textit{$\cC$-preservers} and by
\[\fd_\cC := \left\{ A\in\fd \,\middle|\, e^{tA}\in\fD_\cC\ \text{for all}\ t\geq 0\right\}\]
the set of all \textit{generators of $\cC$-preserving semi-groups}.\index{generator!$\cC$-preserving semi-group}
If $\cC = \pos(\rset^n)$, then we abbreviate
\[\fD_+ := \fD_{\pos(\rset^n)} \quad\text{and}\quad \fd_+ := \fd_{\pos(\rset^n)}.\]
\end{dfn}

Positive semi-groups on Banach lattices can be characterized in terms of their resolvents.
It seems that the case of the LF-space $\rset[x_1,\dots,x_n]$ is rarely considered in the literature.
\cite{didio24posPresConst,didio25KPosPresGen} are two of the very few works on this topic, especially for semi-groups on $\rset[x_1,\dots,x_n]$.
Following ideas from the proof of \cite[Prop.\ 2.3]{ouhabaz05}, we obtain the following  description of $\fd_\cC$ via the resolvent.
Note, we only assume $\cC$ to be convex and closed.
It does \textit{not} need to be a cone.

\begin{prop}[{\cite[Prop.\ 6.2]{didio25KPosPresGen}}]\label{thm:fdCcara}
Let $n\in\nset$, let
\[\cC \subseteq\rset[x_1,\dots,x_n]\]
be non-empty, convex, and closed in the LF-topology of $\rset[x_1,\dots,x_n]$.
Let
\[A\in\fd\]
and denote by
\[\{V_i\}_{i\in\nset_0}\]
the family of finite dimensional invariant subspaces from \Cref{thm:fdCharac} (iii).
For $i\in\nset_0$, we set
\[A_i := A|_{V_i} \qquad\text{and}\qquad \cC_i := \cC\cap V_i.\]
Then the following are equivalent:
\begin{enumerate}[(i)]
\item $A\in\fd_\cC$.

\item For every $i\in\nset_0$, there exists an $\varepsilon_i>0$ such that
\[(\one - \lambda A_i)^{-1}\cC_i \subseteq \cC_i\]
for all $\lambda\in [0,\varepsilon_i)$.
\end{enumerate}
\end{prop}
\begin{proof}
(i) $\Rightarrow$ (ii):
Let $i\in\nset_0$ and $p\in\cC_i$.
Let $\|\cdot\|_i$ be the Euclidean norm on $V_i$.
We denote by $\|\cdot\|_i$ the corresponding operator norm for operators $V_i\to V_i$.
For
\[\gamma > \|A_i\|_i\]
and all $t\geq 0$ we have 
\[ \left\| e^{-\gamma t}\cdot e^{tA_i}p \right\|_d \leq \|p\|_d\cdot e^{(\|A_i\|_i-\gamma)\cdot t}\]
and hence
\[ \gamma(\gamma\one - A_i)^{-1}\, p = \gamma\cdot\int_0^\infty e^{-\gamma t}\cdot e^{tA_i} p~\diff t .\]
Therefore, since
\[\gamma\int_0^\infty e^{-\gamma t}~\diff t = 1\]
and $e^{-\gamma t}$ can be approximated by step functions, we conclude that
\[\gamma\cdot (\gamma\one - A_i)^{-1} p\]
is a convex linear combination of elements
\[e^{tA_i}p\]
in the closed convex set $\cC_i$, i.e.,
\[\gamma\cdot(\gamma\one - A_i)^{-1}\cC_i\subseteq\cC_i\]
for all $\gamma> \|A_i\|_i$.
With
\[\varepsilon_d := \|A_i\|_i^{-1}\]
and hence
\[\lambda := \gamma^{-1} \in [0,\varepsilon_d)\]
we get the assertion (ii) for every $i\in\nset_0$.

(ii) $\Rightarrow$ (i): Let $t\geq 0$ and $p\in\cC$.
Let $i\in\nset_0$ such that $p\in V_i$.
Then $p\in\cC_i$ and
\[e^{tA}p\quad  = \quad e^{tA_i}p
\quad \overset{\text{Cor.\ \ref{cor:taylor}}}{=}\quad \lim_{k\to\infty} \left(\one - \frac{tA_i}{k}\right)^{-k}p \quad\in\cC_i\]
since
\[\left(\one - \frac{tA_i}{k}\right)^{-1}\cC_i \quad\subseteq\quad \cC_i\]
by assumption for all $k > \varepsilon_i^{-1}$.
The limit $k\to\infty$ of
\[\left(\one - \frac{tA_i}{k}\right)^{-k}p\]
is in $\cC_i$ because $\cC_i$ is closed.
Since $t\geq 0$ and $p\in\cC$ were arbitrary, we have
\[e^{tA}\cC\subseteq\cC\]
for all $t\geq 0$, i.e., $A\in\fd_\cC$.
\end{proof}

\begin{exms}[{\cite[Exms.\ 6.3]{didio25KPosPresGen}}]\label{exm:resolventProps}
Let $n = 1$, $V_i=\rset[x]_{\leq i}$, and $\cC = \pos(\rset)$.
\begin{enumerate}[\bfseries\; (a)]
\item For $A = a\in\rset$, we have 
\[(\one - \lambda a)^{-1}>0\]
for all $\lambda\in (-\infty,a^{-1})$.
Then $A$ acts as the multiplication with a positive real constant, i.e., it preserves all cones $\cC_d$ with $\varepsilon_d = a^{-1}$.

\item For $A = \partial_x$,
\[(\one - \lambda\partial_x)^{-1} = \sum_{k\in\nset_0} \lambda^k\cdot\partial_x^k \quad\in\fD_\cC\]
for all $\lambda\in\rset$.
Since
\[s_\lambda := (\lambda^k)_{k\in\nset_0}\]
is the moment sequence with representing measure
\[\mu_\lambda := \delta_\lambda\]
and $(k!)_{k\in\nset_0}$ is a moment sequence by
\[k! = \int_0^\infty x^k\cdot e^{-t}~\diff t,\]
we have that
\[(k!\cdot \lambda^k)_{k\in\nset_0}\]
is a moment sequence by \Cref{cor:odot}.
Hence, by \Cref{thm:kPosPresCara}, we have
\[\pm\partial_x\in\fd_\cC.\]

\item For $A = \partial_x^2$,
\[(\one - \lambda\partial_x^2)^{-1} = \sum_{k=0}^\infty \lambda^k\cdot \partial_x^{2k} \quad\in \fD_\cC\]
for all $\lambda\in [0,\infty)$ since
\[s_\lambda = (1,0,\lambda,0,\lambda^2,0,\lambda^3,\dots)\]
is a moment sequence represented by the measure
\[\mu_\lambda = \frac{1}{2}\delta_{-\sqrt{\lambda}} + \frac{1}{2}\delta_{\sqrt{\lambda}},\]
see Problem \ref{prob:st}, i.e., $\partial_x^2\in\fd_\cC$.

\item Let $A = x\partial_x$.
Since
\[(e^{tA}p)(x) = p(e^t x),\]
we have $\pm A\in\fd_\cC$.
Then
\[Ax^d = dx^d\]
for all $d\in\nset_0$, i.e.,
\[(\one - \lambda x\partial_x)^{-1}x^d = (1-\lambda d)^{-1}\,x^d\]
for all $d\in\nset_0$ and $\lambda\in (-\infty,d^{-1})$ with $0^{-1} = +\infty$.
Here, $\varepsilon_d = d^{-1}$ depends explicitly on the degree $d\in\nset_0$ of the restriction $\cC_d$.
\exmsymbol
\end{enumerate}
\end{exms}

\begin{rem}[{\cite[Rem.\ 6.4]{didio25KPosPresGen}}]
In the preceding examples the following cases for $[0,\varepsilon_d)$ appeared:
\begin{enumerate}[\it\; (a)]
\item In Example \ref{exm:resolventProps} (a) the interval
\[[0,\varepsilon_d)\]
can not be extended to $a^{-1}$ or beyond $a^{-1}$.

\item In Example \ref{exm:resolventProps} (c) the interval
\[[0,\varepsilon_d)\]
is bounded from below by $0$ and can not be extended below $0$.
To see this let
\[p(x) = x^2\in\pos(\rset).\]
Then
\[(\one-\lambda\partial_x^2)^{-1} \,x^2 = (1 +\lambda\partial_x^2)\, x^2 = x^2 +2\lambda\]
is negative at $x = 0$ for any $\lambda < 0$, i.e.,
\[(\one-\lambda\partial_x^2)^{-1}x^2\not\in\pos(\rset)\]
for any $\lambda < 0$.

\item In Example \ref{exm:resolventProps} (d) we have $\varepsilon_d =d^{-1} \rightarrow 0$ as $d\to\infty$.
\end{enumerate}
Thus, in \Cref{thm:fdCcara} we can only have $\lambda\in [0,\varepsilon_d)$ with $\lim_{d\to \infty}\varepsilon_d= 0$.
\exmsymbol
\end{rem}

\section{Generators of $\rset^n$-Positivity Preserving Semi-Groups}

The description of generators via the resolvent is complete, but can not be handled properly.
The examples showed that even in the simplest cases it is not easy to check, since we have to check infinitely many conditions.

Easier was the description with the L\'evy--Khinchin formula for the constant coefficient cases, see \Cref{thm:mainPosGenerators}.
We want to adapt this approach also for non-constant coefficient generators.

We restrict to the degree preserving cases, i.e.,
\[A\rset[x_1,\dots,x_n]_{\leq d} \subseteq \rset[x_1,\dots,x_n]_{\leq d}\]
for all $d\in\nset_0$.

\begin{thm}[{\cite[Thm.\ 6.11]{didio25KPosPresGen}}]\label{prop:genK}
Let $n\in\nset$, let $K\subseteq\rset^n$ be closed, and let
\[A = \sum_{\alpha\in\nset_0^n} \frac{a_\alpha}{\alpha!}\cdot \partial^\alpha\]
be with
\[a_\alpha\in\rset[x_1,\dots,x_n]_{\leq |\alpha|}\]
for all $\alpha\in\nset_0^n$.
If
\begin{enumerate}[(i)]
\item for all $y\in K$ the operator
\[A_y := \sum_{\alpha\in\nset_0^n} \frac{a_\alpha(y)}{\alpha!}\cdot\partial^\alpha\]
is a generator of a $K$-positivity preserving semi-group,
\end{enumerate}
then
\begin{enumerate}[(i)]\setcounter{enumi}{1}
\item $e^{tA}$ is a $K$-positivity preserver for all $t\geq 0$.
\end{enumerate}
\end{thm}
\begin{proof}
We introduce $y=(y_1,\dots,y_n)\in K$ as new variables.
Then
\[e^{tA_y} = \sum_{k\in\nset_0} \frac{t^k}{k!} A_y^k\]
for all $t\in\rset$ by \Cref{cor:taylor}, i.e., we have
\[e^{tA_y} = \one + t A_y + t^2 R_y(t)\]
for some error term $R_y(t)$ such that $R_y(t)$ on $\rset[x_1,\dots,x_n]_{\leq d}$ is uniformly bounded for all $t\in [0,1]$ and for fixed $d\in\nset_0$.

Let $f\in\rset[x_1,\dots,x_n]$.
Since
\[\deg a_\alpha\leq |\alpha|\]
for all $\alpha\in\nset_0^n$, we have
\[g(x,y,t,f) := e^{tA_y}f = (\one + tA_y + t^2 R_y(t))f\in \rset[x_1,\dots,x_n,y_1,\dots,y_n]_{\leq \deg f}\]
for all $t\in\rset$.
By (i), $A_y$ is a generator of a $K$-positivity preserver with constant coefficients, i.e.,
\[g(x,y,t,f)\geq 0\]
for all $f\in\pos(K)$, $x,y\in K$, and $t\geq 0$.
Hence, with $x=y$ we obtain
\begin{equation}\label{eq:gpos}
g(x,x,t,f) \geq 0
\end{equation}
for all $f\in\pos(K)$, $x\in K$, and $t\geq 0$.
For $t\in\rset$, we define
\[G_t:\rset[x_1,\dots,x_n]\to\rset[x_1,\dots,x_n]\]
by
\[(G_t f)(x) := (e^{tA_y}f)(x,x) = g(x,x,t,f),\]
i.e.,
\[G_t f = (\one + tA + t^2 R(t)) f\]
for some error term $R(t)$ which is on $\rset[x_1,\dots,x_n]_{\leq d}$, $d\in\nset_0$, uniformly bounded in $t\in [0,1]$.
By (\ref{eq:gpos}), $G_t$ is a $K$-positivity preserver for all $t\geq 0$.
Hence,
\[e^{tA} \overset{\text{Cor.\ \ref{cor:taylor}}}{=} \lim_{k\to\infty} \left(\one + \frac{t A}{k} \right)^k = \lim_{k\to\infty} \left(\one + \frac{t A + k^{-1}\cdot t^2 R(k^{-1}\cdot t)}{k} \right)^k = \lim_{k\to\infty} G_{t/k}^k\]
and since all
\[G_{t/k}\]
are $K$-positivity preservers, also
\[G_{t/k}^k\]
are $K$-positivity preservers as well.
Therefore, since the set of $K$-positivity preservers is closed by \Cref{lem:closedPosPres} (ii),
\[e^{tA}\]
is a $K$-positivity preserver.
Since $t\geq 0$ was arbitrary, (ii) is proved.
\end{proof}

We now give the equivalent characterization as in \Cref{thm:mainPosGenerators} for the constant coefficient for the non-constant coefficient case, i.e., we use the L\'evy--Khinchin formula.

\begin{thm}[{\cite[Main Thm.\ 6.12]{didio25KPosPresGen}}]\label{thm:PosGeneratorsGeneral}
Let $n\in\nset$ and let
\[A = \sum_{\alpha\in\nset_0^n} \frac{a_\alpha}{\alpha!}\cdot\partial^\alpha\]
be with
\[a_\alpha\in\rset[x_1,\dots,x_n]_{\leq |\alpha|}\]
for all $\alpha\in\nset_0^n$.
Then the following are equivalent:
\begin{enumerate}[(i)]
\item $A\in\fd_+$, i.e.,
\[e^{tA}\]
is a positivity preserver for all $t\geq 0$.

\item For every $y\in\rset^n$, there exist a symmetric matrix
\[\Sigma(y) = (\sigma_{i,j}(y))_{i,j=1}^n\]
with real entries such that $\Sigma(y)\succeq 0$, a vector
\[b(y) = (b_1(y),\dots,b_n(y))^T\in\rset^n,\]
a constant
\[a_0\in\rset,\]
and a $\sigma$-finite measure $\nu_y$ on $\rset^n$ with
\[\nu_y(\{0\})=0 \qquad\text{and}\qquad \int_{\rset^n} |x^\alpha|~\diff\nu_y(x)<\infty\]
for all $\alpha\in\nset_0^n$ with $|\alpha|\geq 2$ such that
\begin{align*}
a_{e_i}(y) &= b_i(y) + \int_{\|x\|_2\geq 1} x_i~\diff\nu_y(x) && \text{for}\ i=1,\dots,n,\\
a_{e_i+e_j}(y) &= \sigma_{i,j}(y) + \int_{\rset^n} x^{e_i + e_j}~\diff\nu_y(x) && \text{for all}\ i,j=1,\dots,n,
\intertext{and}
a_\alpha(y) &= \int_{\rset^n} x^\alpha~\diff\nu_y(x) &&\text{for}\ \alpha\in\nset_0^n\ \text{with}\ |\alpha|\geq 3.
\end{align*}
\end{enumerate}
\end{thm}
\begin{proof}
By \Cref{thm:mainPosGenerators}, (ii) is equivalent to the fact that
\[A_y = \sum_{\alpha\in\nset_0^n} \frac{a_\alpha(y)}{\alpha!}\cdot\partial^\alpha\]
is a generator of a positivity preservering semi-group with constant coefficients for all $y\in\rset^n$.

(ii) $\Rightarrow$ (i):
That is \Cref{prop:genK}.

(i) $\Rightarrow$ (ii):
Let $y\in\rset^n$.
By (i),
\[A\in\fd_+,\]
i.e.,
\[T_t := e^{tA}\]
is a positivity preserver for all $t\geq 0$.
Hence,
\begin{equation}\label{eq:yPos}
(T_{t,y} f)(y) = (T_t f)(y)\geq 0
\end{equation}
for $t\geq 0$ and $f\in\pos(\rset^n)$, i.e., the linear map
\[T_{t,y}\]
with constant coefficients preserves positivity of $f\in\pos(\rset^n)$ at $x = y$.

Since $T_{t,y}$ is a linear operator with constant coefficients, it commutes with $\partial_i$ for $i=1,\dots,n$ and hence with
\[e^{z\cdot\nabla}\]
which is the shift
\[(e^{z\cdot\nabla}f)(y) = f(y+z)\]
for any $z\in\rset^n$.
Hence, for $x\in\rset^n$, we have
\[(T_{t,y}f)(x) = (T_{t,y}f)(y+x-y) = (T_{t,y} e^{(x-y)\cdot\nabla} f)(y) = e^{(x-y)\cdot\nabla} (T_{t,y} f)(y) \geq 0,\]
by (\ref{eq:yPos}), i.e., $T_{t,y}$ is a positivity preserver.

By \Cref{cor:taylor},
\[T_t = e^{tA} = \sum_{k\in\nset_0} \frac{t^k A^k}{k!} = \one + tA + t^2 R(t)\]
with some error term $R(t)$ which is uniformly bounded for all $t\in [0,1]$ on $\rset[x_1,\dots,x_n]_{\leq d}$ for all fixed $d\in\nset_0$. 
Hence, 
\begin{multline*}
e^{tA_y} \overset{\text{Cor.\ \ref{cor:taylor}}}{=} \lim_{k\to\infty} \left(\one + \frac{t A_y}{k} \right)^k\\ = \lim_{k\to\infty} \left(\one + \frac{t A_y + k^{-1}\cdot t^2 R_y(k^{-1}\cdot t)}{k} \right)^k = \lim_{k\to\infty} T_{t/k,y}^k.
\end{multline*}
Since
\[T_{t/k,y}\]
is a positivity preserver, so is
\[T_{t/k,y}^k\]
and, by \Cref{lem:closedPosPres} (ii), we have that
\[e^{tA_y}\]
is also a positivity preserver for $t\geq 0$.
In summary, $A_y$ is a generator of a positivity preservering semi-group with constant coefficients.
Since $y\in\rset^n$ was arbitrary, we have proved (ii) by 
\Cref{thm:mainPosGenerators}.
\end{proof}

\begin{rem}[{\cite[Rem.\ 6.13]{didio25KPosPresGen}}]\label{rem:shift}
Let
\[A = a x\partial_x.\]
Then
\[A^k x^m = (am)^k x^m\]
and hence
\[e^{tA} x^m = \sum_{k\in\nset_0} \frac{t^k A^k}{k!}x^m = \sum_{k\in\nset_0} \frac{(atm)^k}{k!} x^m = e^{atm} x^m = (e^{at} x)^m,\]
i.e.,
\[(e^{tA}f)(x) = f(e^{at} x)\]
is the scaling which is a $[0,\infty)$-positivity preserver for all $a\in\rset$.
But for $a=-1$ and $y=1$ we have
\[A_1 = -\partial_x\]
which is not a generator of a $[0,\infty)$-positivity preserving semi-group with constant coefficients.
That is, we do not get a description of all generators because the operators
\[ax\partial_x\]
with $a<0$ are not covered.
The reason for the failure is that in the proof of \Cref{thm:PosGeneratorsGeneral} the positivity at arbitrary $x\in\rset^n$ was derived from the positivity at $x = y$ by applying the shift
\[e^{c\partial_x}\]
which commutes with
\[e^{A_y}.\]
On $[0,\infty)$ this reasoning fails because $[0,\infty)$ is not translation invariant.
\exmsymbol
\end{rem}

\section*{Problems}%%%%%%%%%%%%%%%%%%%%%
\addcontentsline{toc}{section}{Problems}

\begin{prob}
Similar to \Cref{rem:shift}, give a second example of a closed set $K\subseteq\rset^n$ and an operator
\[A:\rset[x_1,\dots,x_n]\to\rset[x_1,\dots,x_n]\]
in $\fd$ such that $A$ is a generator of a $K$-positivity preserving semi-group but is not covered by \Cref{thm:PosGeneratorsGeneral}.
\end{prob}

\chapter{Eventually Positive Semi-Groups}
\label{sec:eventually}

Since we characterized the generators of positive semi-groups for $K=\rset^n$, we want to look at some operators $A$ which \emph{seem} to generate positive semi-groups.

\begin{dfn}[see e.g.\ \cite{daners18} or \cite{glueck23}]
Semi-groups
\[(e^{tA})_{t\geq 0}\]
such that
\[e^{tA}\]
is positive for all $t\geq \tau$ for some $\tau>0$ but not for $t\in (0,\tau)$ are called (uniformly) \textit{eventually positive} (or non-negative) semi-groups.
\end{dfn}

\section{A First Example}

\begin{prop}[{\cite[Prop.\ 7.1]{didio25KPosPresGen}}]\label{thm:sigma}
Let $t\geq 0$ and
\[T_t := e^{t\cdot (x\partial_x)^3}: \rset[x]_{\leq 4}\to\rset[x]_{\leq 4},\]
i.e.,
\[a_0 + a_1x + a_2 x^2 + a_3x^3 + a_4 x^4 \;\mapsto\; a_0 + a_1 e^t x + a_2 e^{8t} x^2 + a_3 e^{27t} x^3 + a_4 e^{64t} x^4.\]
Then there exists a constant $\tau\in\rset$ with
\[1.19688\cdot 10^{-2} ~~ <~~ \tau ~~<~~ 1.19689\cdot 10^{-2}\]
such that
\begin{enumerate}[(i)]
\item $T_t$ is not a positivity preserver for any $t \in (0,\tau)$ and

\item $T_t$ is a positivity preserver for $t = 0$ and all $t\geq \tau$,
\end{enumerate}
i.e.,
\[\left(e^{t\cdot (x\partial_x)^3}\right)_{t\geq 0} \quad\text{on}\ \rset[x]_{\leq 4}\]
is an \emph{eventually positive} semi-group.
\end{prop}
\begin{proof}
We have
\[T_t x^k = \lambda_k(t)\cdot x^k\]
with
\[\lambda_k(t) = e^{t\cdot k^3}\]
for all $k=0,\dots,4$.
Hence, we have
\[\cH(\lambda(t))_0 = (1)\succeq 0\]
since $\det\cH(\lambda(t))_0 = 1 \geq 0$ and
\[\cH(\lambda(t))_1 = \begin{pmatrix} 1 & e^t\\ e^t & e^{8t} \end{pmatrix}\succeq 0\]
since
\[\det\cH(\lambda(t))_1 = e^{8t} - e^{2t} = e^{2t}\cdot (e^{6t} - 1) \geq 0\]
for all $t\geq 0$ with $\det\cH(\lambda(t))_1>0$ for $t>0$.
$\cH(\,\cdot\,)_i$ are the Hankel matrices of order $i$.
It is therefore sufficient to look at $\det\cH(\lambda(t))_2$, which is
\[h_2(t) := \det\cH(\lambda(t))_2 = e^{72t} - e^{66t} - e^{54t} + 2e^{36t} - e^{24t}.\]
From
\[h_2(0) = 0,\quad h_2'(0) = 0,\quad \text{and}\quad h_2''(0) = -72\]
we get
\[h(t) < 0\]
for $t\in (0,\varepsilon)$ for some $\varepsilon > 0$.
By calculations from Mathematica \cite{mathematica}, we get for the smallest eigenvalue $\sigma_3(t)$ that
\[\sigma_3(0.0119688) \approx -3.39928\cdot 10^{-8} \quad\text{and}\quad \sigma_3(0.0119689) \approx 1.7888\cdot 10^{-8},\]
i.e., $\cH(\lambda(t))_2$ becomes positive semi-definite between
\[t = 0.0119688 \qquad\text{and}\qquad t = 0.0119689.\]
The smallest eigenvalues $\sigma_3(t)$ of $\cH(\lambda(t))_2$ are depicted in \Cref{fig:sigma} for $t\in [0,0.015]$, $t\in [0,1]$, and $t\in [0,10]$.
We have
\[h_2(t)<0\] for $t\in (0,\tau)$ with
\[\tau\in (0.0119688,0.0119689)\]
which proves (i).
Since we know $\sigma_3(t)$ for $t\in [0,1]$, see \Cref{fig:sigma} (b), and since
\[h_3'(t) > 0\]
for all $t\geq 1$, we have (ii).
\end{proof}

\begin{figure}[htb!]
\begin{subfigure}{.33\textwidth}
\includegraphics[width=\linewidth]{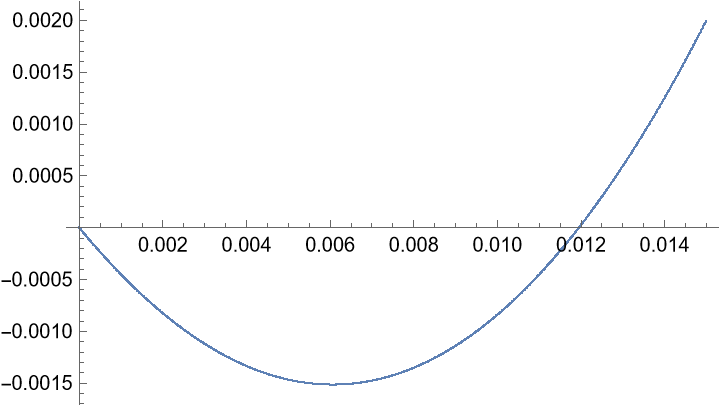}
\caption{$t\in [0,0.015]$}
\end{subfigure}
\begin{subfigure}{.325\textwidth}
\includegraphics[width=\linewidth]{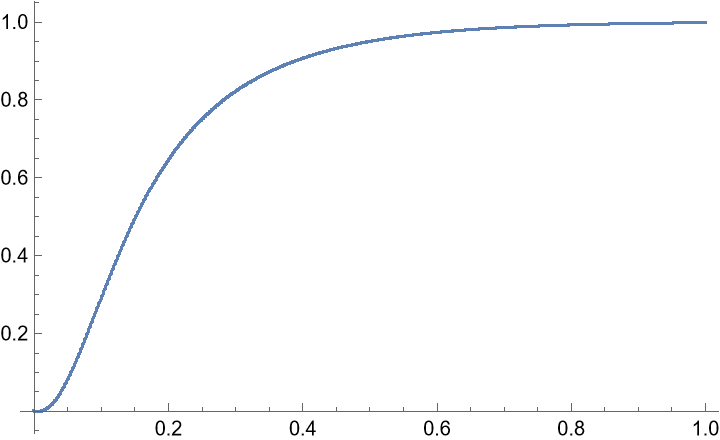}
\caption{$t\in [0,1]$}
\end{subfigure}
%
%\begin{center}
\begin{subfigure}{.33\textwidth}
\includegraphics[width=\linewidth]{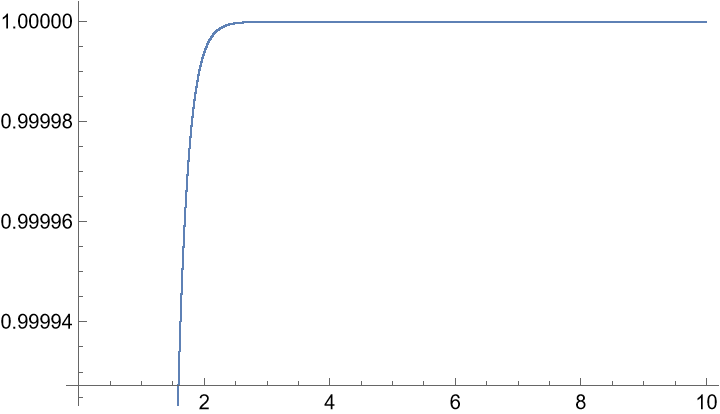}
\caption{$t\in [0,10]$}
\end{subfigure}
%\end{center}
%
\caption{The smallest eigenvalue $\sigma_3(t)$ from \Cref{thm:sigma} for $t\in [0,10]$.
All three figures were generated by \cite{mathematica}.
Originally published in {\cite[Fig.\ 1]{didio25KPosPresGen}}.\label{fig:sigma}}
\end{figure}

\begin{rem}[{\cite[Rem.\ 7.2]{didio25KPosPresGen}}]
Note that
\[T_t = e^{t (x\partial_x)^3}\]
on $\rset[x]$, not just on $\rset[x]_{\leq 4}$, fulfills the following.
For any $p\in\pos(\rset)$, there exists a $\tau_p>0$ such that
\[T_tp \in\pos(\rset)\]
for all $t\geq \tau_p$.
\exmsymbol
\end{rem}

For eventually positive semi-groups see e.g.\ \cite{daners18,denkploss21,glueck23}.
Our example
\[A = (x\partial_x)^3 \quad\text{on}\quad \rset[x]_{\leq 4}\]
in \Cref{thm:sigma} acts only on the $5$-dimensional (Hilbert) space
\[\rset[x]_{\leq 4}\]
and is therefore a matrix example of an eventually positive semi-group, see \cite{daners16b} and literature therein for the matrix cases.
Note, that in the matrix case the invariant cone is (usually) $[0,\infty)^n$.
In our example it is $\pos(\rset)_{\leq 4}$.
So there are slight differences.

\section{A Second Example}

Eventually positive semi-groups are an active field with the need for examples on low-dimensional (Hilbert) spaces which can be calculated and investigated \cite{glueck23}.
Therefore, we provide the smallest possible case of an eventually positive semi-group in our framework: A degree two operator on
$\rset[x]_{\leq 2}$

\begin{prop}[{\cite[Prop.\ 7.3]{didio25KPosPresGen}}]\label{thm:1}
Let $a\in\rset$ be real with $|a|>\frac{1}{\sqrt{5}}$ and let
\[A = a\cdot\partial + \frac{1}{2}(x^2-1)\cdot\partial^2\]
be an operator on the $3$-dimensional Hilbert space $\cH = \rset[x]_{\leq 2}$.
Then
\[\big(e^{tA} \big)_{t\geq 0}\]
is an eventually positive semi-group on $\cH$, i.e., there exists a constant $\tau_a >0$ such that,
\begin{enumerate}[(i)]
\item for all $t\in (0,\tau_a)$,
\[e^{tA}\pos(\rset)_{\leq 2}\quad\not\subseteq\quad\pos(\rset)_{\leq 2}\]
and,

\item for all $t = 0$ and $t\geq \tau_a$,
\[e^{tA}\pos(\rset)_{\leq 2}\quad\subseteq\quad\pos(\rset)_{\leq 2}.\]
\end{enumerate}
For $a\in\rset$ with $|a|\leq \frac{1}{\sqrt{5}}$, we have
\[e^{tA}\pos(\rset)_{\leq 2}\quad\not\subseteq\quad\pos(\rset)_{\leq 2}\]
for all $t\in (0,\infty)$.
\end{prop}
\begin{proof}
Let $\{1,x,x^2\}$ be the monomial basis of $\cH$.
From
\[A1 = 0,\quad Ax = a,\quad \text{and}\quad Ax^2 = -1 + 2ax + x^2\]
we get the matrix representation
\[\tilde{A} = \begin{pmatrix} 0 & a & -1\\ 0 & 0 & 2a\\ 0 & 0 & 1 \end{pmatrix}\]
of $A$ on $\cH$ in the monomial basis.
Therefore,
\[\tilde{B} = \exp(t\tilde{A}) = \begin{pmatrix}
1 & at & (2a^2-1)\cdot (e^t - 1) - 2a^2 t\\
0 & 1 & 2a\cdot (e^t - 1)\\
0 & 0 & e^t
\end{pmatrix} = \begin{pmatrix}
1 & at & f(a,t)\\
0 & 1 & g(a,t)\\
0 & 0 & e^t
\end{pmatrix}\]
for all $t\in\rset$, e.g.\ by using \cite{mathematica}.
$\tilde{B}$ is the matrix representation of
\[B = b_0 + b_1\cdot\partial + \frac{b_2}{2}\cdot\partial^2\]
with
\begin{align*}
B1 = 1 = b_0 \quad &\Rightarrow\quad b_0 = 1,\\
Bx = at + x = b_0 x + b_1 \quad &\Rightarrow\quad b_1 = at,
\end{align*}
and finally
\[Bx^2 = f(a,t) + g(a,t)\cdot x + e^t\cdot x^2 = b_0 x^2 + 2b_1 x + b_2\]
gives
\begin{align*}
b_2 &= f(a,t) + g(a,t)\cdot x + e^t\cdot x^2 - b_0 x^2 - 2b_1 x\\
&= (2a^2 - 1)\cdot (e^t-1) - 2a^2 t + 2a\cdot (e^t-1 - t)\cdot x + (e^t-1)\cdot x^2.
\end{align*}
By \Cref{thm:kPosPresCara},
\[B\pos(\rset)_{\leq 2}\subseteq\pos(\rset)_{\leq 2}\]
is equivalent to
\begin{equation}\label{eq:3}
\begin{pmatrix}
b_0(x) & b_1(x)\\ b_1(x) & b_2(x)
\end{pmatrix}\succeq 0\quad \text{for all}\ x\in\rset.
\end{equation}
Since $b_0 = 1>0$, (\ref{eq:3}) is equivalent to
\begin{equation}\label{eq:4}
h(x) := \det\begin{pmatrix}
b_0(x) & b_1(x)\\ b_1(x) & b_2(x)
\end{pmatrix} = b_0(x)\cdot b_2(x) - b_1(x)^2 \geq 0\quad \text{for all}\ x\in\rset
\end{equation}
with
\[h(x) = (2a^2 -1)\cdot (e^t - 1) - 2a^2 t - a^2 t^2 + 2a\cdot (e^t-1-t)\cdot x + (e^t-1)\cdot x^2.\]
For $t>0$, we have that $h$ is a quadratic polynomial in $x$ and hence (\ref{eq:4}) is equivalent to
\[m(a,t) := \min_{x\in\rset} h(x) \geq 0\]
with
\begin{equation}\label{eq:min}
m(a,t) = 1-e^t + a^2\cdot\frac{5 + 8t + 4t^2 - (10+8t+t^2)\cdot e^t + 5\cdot e^{2t}}{e^t-1}.
\end{equation}
Hence, for each $a\in\rset$ with $|a|> \frac{1}{\sqrt{5}}$ there exists a $\tau_a>0$ such that
\[m(a,t)<0\]
for all $t\in (0,\tau_a)$ and
\[m(a,t) > 0\]
for all $t\in (\tau_a,\infty)$, i.e., (i) and (ii) are proved.

Let $a\in\rset$ be with $|a|<\frac{1}{\sqrt{5}}$.
We see that
\[m(a,t)<0\]
for all $t>0$, and for the boundary case $a=\pm\frac{1}{\sqrt{5}}$ we have
\[m(\pm 5^{-1/2},t) = \frac{4 e^{-t}\cdot t\cdot (t+2) - t\cdot (t+8)}{5 - 5e^{-t}} < 0\]
also for all $t>0$ which proves
\[e^{tA}\pos(\rset)_{\leq 2}\not\subseteq\pos(\rset)_{\leq 2}\]
for all $t>0$.
\end{proof}

It is clear from (\ref{eq:min}) that we have
\[\tau_a = \tau_{-a}\]
for all $a\in\rset$ with
\[|a|>\frac{1}{\sqrt{5}}.\]
In \Cref{fig:2} we give the minima $m(a,t)$ from (\ref{eq:min}) for two values $a$ close to $\frac{1}{\sqrt{5}}$.
\begin{figure}[htb!]
\begin{subfigure}{.5\textwidth}
\includegraphics[width=0.95\linewidth]{min1.png}
\caption{$a=0.44721359$.}
\end{subfigure}\quad
\begin{subfigure}{.5\textwidth}
\includegraphics[width=0.95\linewidth]{min2.png}
\caption{$a=0.44721360$.\label{fig:1b}}
\end{subfigure}
\caption{The minima $m(a,t)$ from (\ref{eq:min}) for two values $a$ close to $\frac{1}{\sqrt{5}} \approx 0.4472135955$ in the range $t\in [0,25]$.
Both figures were generated by \cite{mathematica}. Originally published in {\cite[Fig.\ 2]{didio25KPosPresGen}}.\label{fig:2}}
\end{figure}
For
\[a = 0.44721359 < \frac{1}{\sqrt{5}},\]
the minima $m(a,t)$ is always negative, i.e.,
\[(e^{tA})_{t\geq 0}\]
is not eventually positive.
For
\[a = 0.44721360 > \frac{1}{\sqrt{5}},\]
the minima $m(a,t)$ changes sign and is in particular $>0$ for
\[t\geq \tau_a \approx 22.66,\]
i.e.,
\[(e^{tA})_{t\geq 0}\]
is eventually positive.

\begin{exms}[{\cite[Exm.\ 7.4]{didio25KPosPresGen}}]\label{exm:1}
In the following we give the approximate values of $\tau_{\pm a}$ for several $a\in\rset$ with $|a| > \frac{1}{\sqrt{5}}$:
\begin{enumerate}[\bfseries\; (a)]
\item $a=0.44721360$ (\Cref{fig:1b}):
\[22.655 \quad < \quad \tau_{\pm a} \quad < \quad 22.656\]

\item $a=0.45$:
\[7.5504 \quad < \quad \tau_{\pm a} \quad < \quad 7.5505\]

\item $a=1$:
\[1.1675 \quad < \quad \tau_{\pm a} \quad < \quad 1.1676\]

\item $a=10$:
\[9.7541\cdot 10^{-2} \quad < \quad \tau_{\pm a} \quad < \quad 9.7542\cdot 10^{-2}\]

\item $a=100$:
\[9.6219\cdot 10^{-3} \quad < \quad \tau_{\pm a} \quad < \quad 9.6220\cdot 10^{-3} \tag*{$\circ$}\]
\end{enumerate}
\end{exms}

\section*{Problems}%%%%%%%%%%%%%%%%%%%%%
\addcontentsline{toc}{section}{Problems}

\begin{prob}
In \Cref{thm:1} with the operator
\[A = a\cdot\partial + \frac{1}{2}(x^2-1)\cdot\partial^2\]
\begin{enumerate}[\bfseries\qquad a)]
\item find a heuristic explanation, why this could be a generator of an eventually positive semi-group.

\item What is the effect of the parameter $a\in\rset$ resp.\ the effect of the operator part $a\cdot\partial$?
\end{enumerate}
\end{prob}

\begin{prob}
In \Cref{exm:1}, determine upper and lower bounds for $\tau_{\pm a}$ for
\begin{enumerate}[\bfseries\qquad a)]
\item $a= 5$,
\item $a= 25$, and
\item $a=50$.
\end{enumerate}
\textit{Hint:} Use $m(a,t)$ in (\ref{eq:min}) with the bisection method.
\end{prob}

\appendix
\part*{Appendices}
\addcontentsline{toc}{part}{Appendices}
%\motto{All's well that ends well}

\backmatter%%%%%%%%%%%%%%%%%%%%%%%%%%%%%%%%%%%%%%%%%%%%%%%%%%%%%%%

%\Extrachap{Solutions}
%\chapter*{Solutions}
%\addcontentsline{toc}{chapter}{Solutions}
%\setcounter{chapter}{19}
\renewcommand{\thechapter}{\Alph{chapter}}

%\section*{Problems of \Cref{ch:mom}}
%%%%%%%%%%%%%%%%%%%%%%%%%%%%%%%%%%%%

%\begin{sol}{prob:}

%\end{sol}

\newpage
%\bibliographystyle{amsalpha}
%\bibliography{../../../bibdata}

\providecommand{\bysame}{\leavevmode\hbox to3em{\hrulefill}\thinspace}
\providecommand{\MR}{\relax\ifhmode\unskip\space\fi MR }
% \MRhref is called by the amsart/book/proc definition of \MR.
\providecommand{\MRhref}[2]{%
  \href{http://www.ams.org/mathscinet-getitem?mr=#1}{#2}
}
\providecommand{\href}[2]{#2}

%\Extrachap{List of Symbols}
\chapter*{List of Symbols}
\addcontentsline{toc}{chapter}{List of Symbols}

%\section*{\dots}
%%%%%%%%%%%%%%%%%%%
%\addcontentsline{toc}{section}{\dots}

\noindent
$[\,\cdot\,,\,\cdot\,]$ \dotfill \pageref{bracket}

\noindent
$\fB(\rset^n)$: Example \ref{exm:moments} (b) \dotfill \pageref{fBX}

\noindent
$\fd$: \Cref{dfn:fd} \dotfill \pageref{dfn:fd}

\noindent
$\fd_c$: \Cref{dfn:fDcfdc} \dotfill \pageref{dfn:fDcfdc}

\noindent
$\fd_\cC$: \Cref{dfn:fDC} \dotfill \pageref{dfn:fDC}

\noindent
$\fD_c$: \Cref{dfn:fDcfdc} \dotfill \pageref{dfn:fDcfdc}

\noindent
$\fD_\cC$: \Cref{dfn:fDC} \dotfill \pageref{dfn:fDC}

\noindent
$\fd_{c,+}$: \Cref{dfn:fd+} \dotfill \pageref{fdc+}

\noindent
$\fD_{c,+}$: \Cref{dfn:fd+} \dotfill \pageref{fDc+}

\noindent
$\fd_\cV$: \Cref{cor:fdV} \dotfill \pageref{fdV}

\noindent
$\fd_+$: \Cref{dfn:fDC} \dotfill \pageref{dfn:fDC}

\noindent
$\fD_+$: \Cref{dfn:fDC} \dotfill \pageref{dfn:fDC}

\noindent
$\fd_y$: \Cref{exm:nonSubspaceLieAlgebra} \dotfill \pageref{fdy}

\noindent
$\delta_y$: Example \ref{exm:moments} (a) \dotfill \pageref{deltay}

\noindent
$\gl(n,\rset)$ \dotfill \pageref{glnr}

\noindent
$\Gl(n,\rset)$ \dotfill \pageref{GlnR}

\noindent
$\cL(\cV)$: \Cref{dfn:momCone} \dotfill \pageref{dfn:momCone}

\noindent
$L_s$: \Cref{dfn:RieszFunctional} \dotfill \pageref{dfn:RieszFunctional}

\noindent
$l_y$: \Cref{exm:momFunctionals} (a) \dotfill \pageref{dfn:ly}

\noindent
$\fo(n,\rset)$ \dotfill \pageref{onR}

\noindent
$\mathrm{O}(n,\rset)$ \dotfill \pageref{OnR}

\noindent
$\pos(K)$: Eq.\ (\ref{eq:posKdfn}) \dotfill \pageref{posK}, \pageref{eq:posKdfn}

\noindent
$\cS(\cV,v)$: \Cref{dfn:RieszFunctional} \dotfill \pageref{dfn:RieszFunctional}

\noindent
$\fsl(n,\rset)$ \dotfill \pageref{slnR}

\noindent
$\Sl(n,\rset)$ \dotfill \pageref{SlnR}

\noindent
$\so(n,\rset)$ \dotfill \pageref{sonR}

\noindent
$\mathrm{SO}(n,\rset)$ \dotfill \pageref{SOnR}

\noindent
$T_eG$ \dotfill \pageref{tangentspace}

\noindent
$T_y$: \Cref{dfn:Ty}, Equation (\ref{eq:Ty}) \dotfill \pageref{dfn:Ty}

\noindent
$\cV_+$: \Cref{dfn:cV+} \dotfill \pageref{dfn:cV+}

\noindent
$\cZ(f)$: Equation (\ref{eq:zeroSet}) \dotfill \pageref{eq:zeroSet}

%%%%%%%%%%%%%%
%%% Index %%%%
\printindex%%%
%%%%%%%%%%%%%%

%%%%%%%%%%%%%
%%% E N D %%%
%%%%%%%%%%%%%
\end{document}